\documentclass[12pt,a4paper]{article}
\usepackage[english]{babel}
\usepackage[left=2cm,right=2cm,top=3cm]{geometry}
\usepackage{amsmath,amsthm,amssymb,datetime,enumerate,mathrsfs,enumitem}
\setlist{topsep=.2ex,itemsep=.2ex}
\usepackage[colorlinks,citecolor=blue,pdfstartview=FitH,pdfstartpage=1]{hyperref}
\usepackage[dvipsnames]{xcolor}
\usepackage{graphicx}
\usepackage{tikz, tikz-cd}
\usetikzlibrary{matrix,decorations.pathmorphing,decorations.markings,arrows}
\usepackage[textsize=footnotesize,textwidth=18ex]{todonotes}

\newtheorem{theorem}{Theorem}
\numberwithin{theorem}{section}

\newtheorem*{theorem*}{Theorem}
\newtheorem{proposition}[theorem]{Proposition}
\newtheorem{lemma}[theorem]{Lemma}
\newtheorem{corollary}[theorem]{Corollary}

\theoremstyle{definition}
\newtheorem{definition}[theorem]{Definition}
\newtheorem{example}[theorem]{Example}
\newtheorem{remark}[theorem]{Remark}

\DeclareMathOperator{\im}{im}
\DeclareMathOperator{\Spec}{Spec}

\DeclareMathOperator{\Tan}{T}

\DeclareMathOperator{\End}{End}

\DeclareMathOperator{\Hom}{\mathbf{hom}}
\DeclareMathOperator{\Set}{{(\mathbf{set})}}

\DeclareMathOperator{\Catt}{{(\mathfrak{cat})}}
\DeclareMathOperator{\Sch}{{(\mathbf{sch})}}
\DeclareMathOperator{\MSch}{{(\mathbf{msch})}}
\DeclareMathOperator{\Ring}{{(\mathbf{ring})}}
\DeclareMathOperator{\MRing}{{(\mathbf{mring})}}
\DeclareMathOperator{\TRing}{{(\mathbf{tring})}}
\DeclareMathOperator{\TSch}{{(\mathbf{tsch})}}
\DeclareMathOperator{\pardim}{pardim}

\DeclareMathOperator{\coker}{{\mathrm{coker}}}

\let\phi\varphi
\let\epsilon\varepsilon

\let\tops\texorpdfstring
\let\onto\twoheadrightarrow
\let\into\rightarrowtail
\let\hat\widehat
\let\tilde\widetilde
\let\bar\overline

\newcommand{\flift}{\triangle}
\newcommand{\Fr}{{\mathrm{Fr}}}
\newcommand{\fr}{{\mathrm{fr}}}
\newcommand{\C}{{\mathscr C}}
\newcommand{\D}{{\mathscr D}}

\newcommand{\M}{{\mathscr M}}

\newcommand{\ob}{\mathrm{ob}}
\newcommand{\op}{\mathrm{op}}
\newcommand{\opcit}{\emph{op.\@cit.}}
\newcommand{\Cdual}{{\widehat{\mathscr C}}}
\newcommand{\alg}[1]{{#1}{\text-}\mathbf{alg}}
\newcommand{\mix}{{\mathrm m}}

\newcommand{\antimix}{{\overline{\mathrm m}}}

\newcommand\comp{{\mathrm{c}}}
\newcommand\tdd{{\mathrm{tdd}}}
\newcommand\pstr{{\langle p\rangle}}

\newcommand{\fatarrow}{ \rar[phantom,">" description]\rar[dash, shift left=0.25ex] \rar[dash,shift right=0.30]}
\newcommand{\fatlarrow}{ \rar[phantom,"<" description]\rar[dash, shift left=0.25ex] \rar[dash,shift right=0.30]}

\newcommand{\dynkin}{ %
	\tikzcdset{diagrams= %
		{nodes={inner sep=0pt},row sep=10pt}
	}
}

\tikzset{
	closed/.style = {decoration = {markings, mark = at position 0.5 with { \node[transform shape, yscale = 0.8] {$\mid$}; } }, postaction = {decorate} },
	open/.style = {decoration = {markings, mark = at position 0.5 with { \node[transform shape, scale = 1.2, yshift = -0.45pt] {$\circ$}; } }, postaction = {decorate} },
	nonlinear/.style = {decoration = {markings, mark = at position 0.5 with { \node[transform shape, scale = 1.2, yshift = -0.45pt] {$\circ$}; } }, postaction = {decorate} }
}

\author{Karsten Naert\footnote{\href{mailto:karsten.naert@gmail.com}{karsten.naert@gmail.com}}}
\date{\today}

\usepackage[backend=biber,maxbibnames=10,citestyle=alphabetic,style=alphabetic]{biblatex}
\bibliography{bibliography}
\title{\textcolor{blue}{\bf Twisting and mixing}}

\begin{document}
\maketitle

\begin{abstract}
{\let\thefootnote\relax\footnote{{ \emph{2010 Mathematics Subject Classification.} 20G15, 18A05, 14L15.}}} We present a framework that connects three interesting classes of groups: the twisted groups (also known as Suzuki-Ree groups), the mixed groups and the exotic pseudo-reductive groups. 

For a given characteristic $p$, we construct categories of twisted and mixed schemes. Ordinary schemes are a full subcategory of the mixed schemes. Mixed schemes arise from a twisted scheme by base change, although not every mixed scheme arises this way. The group objects in these categories are called twisted and mixed group schemes.

Our main theorems state: (1) The twisted Chevalley groups ${}^2\mathsf B_2$, ${}^2\mathsf G_2$ and ${}^2\mathsf F_4$  arise as rational points of twisted group schemes. (2) The mixed groups in the sense of Tits arise as rational points of mixed group schemes over mixed fields. (3) The exotic pseudo-reductive groups of Conrad, Gabber and Prasad are Weil restrictions of mixed group schemes.
\end{abstract}

\renewcommand{\baselinestretch}{0.75}\normalsize
\tableofcontents
\renewcommand{\baselinestretch}{1.0}\normalsize

\section*{Introduction}

One of the cornerstones of algebraic group theory is the structure theory of connected reductive groups over an algebraically closed field, which is largely due to Chevalley. The theory was quickly extended to a theory of reductive groups over arbitrary fields, and in fact, over an arbitrary base scheme by many others, most notably the authors of \cite{SGA3}. However, during the second half of the 20th century, on a number of occasions, groups have been encountered which appear to be  closely related to reductive groups, but in a strange, exotic manner.

The first time this happened was probably in 1960. In the process of classifying a class of finite simple groups, Suzuki discovered discovered a new class, now known as the Suzuki groups. His discovery was a precursor for the discovery of a more general construction by Ree later that year which produces also other similar classes of groups: the twisted Chevalley groups. Somewhat later---probably around 1970---Tits was studying reductive groups by means of his theory of buildings. In the process of classifying certain buildings, he discovered that although most of these buildings came from reductive groups, there were a few that were related, but not so directly: these were the mixed buildings and groups. In 1997 then, Weiss, in the process of completing the classification of another class of buildings, discovered groups that are arguably even stranger, but nonetheless still recognisable as distant cousins of reductive groups. Around 2010 finally, Conrad, Gabber and Prasad, in the process of developing a structure theory for a class of algebraic groups named pseudo-reductive groups, discovered that---as the name seems to suggest!---most of them are actually quite closely related to reductive groups. But again, there are some estranged and exotic family members that are only distantly related. 

All of these occurences share two common features. One feature is that it is always the combinatorics of root systems with roots of two different lengths that makes the construction work. Another is that the construction requires fields of a specific characteristic; this characteristic is always $2$ or $3$ and depends only on the ratio of lengths of roots in the root system. Moreover the construction requires certain ingredients which are very specific to `positive characteristic mathematics', in particular they always require the Frobenius endomorphism of a field or algebraic group and sometimes also depend on the occurrence of inseparable field extensions.

The present work aims to deepen our understanding of all these groups. It has been our point of view that we should really focus on the second of these features, i.e. the `positive characteristic mathematics'. It is easy enough to smuggle in the combinatorics of root systems via a backdoor, namely by assuming the existence of certain isogenies. But in our approach, this is really more an afterthought.

Here is the basic idea. We think of an algebraic group as a group object in the category of schemes. But if we are considering the category of schemes for a fixed positive characteristic only, there is a special feature under the guise of the absolute Frobenius. We use this feature to create new but closely related categories: the categories of twisted and mixed schemes. Then there are also group objects in these categories and this is where these distant cousins of reductive groups find their origin.

In fact, we are relieved to admit that we did not need any difficult scheme theory or commutative algebra to achieve this. Since the essential bit of information that we wish to manipulate consists of a category $\C$ together with a gadget $F$, this is exactly what we study throughout most of the work. Certainly one could try to think of other applications than the one where $\C$ is the category of schemes in characteristic $p$ and $F$ is the absolute Frobenius, for instance by taking $\C$ arbitrary and $F$ trivial, or $\C$ and abelian category and $F$ the zero endomorphism.

The category of twisted schemes obtained in this manner lies a bit deeper than the category of schemes. We have resisted the temptation to denote its terminal object by $\Spec \mathbf F_{\sqrt p}$ (the field with $\sqrt p$ elements!) throughout our text, but it would have been not entirely unreasonable to do so. If one takes the slice category of the twisted schemes over the  object $\Spec \mathbf F_p$---which is still present, but not as terminal object---one obtains the category of mixed schemes, which contains the ordinary schemes as a full subcategory. 

The outcome of all this is fomulated in our main theorems, and is summarized as follows: the twisted and mixed groups of Suzuki, Ree and Tits arise as group objects in our categories of twisted and mixed schemes. Of course, every ordinary algebraic group is also a mixed group in a natural way so we are not throwing away anything; we are only adding extra objects that we call invisible---which is why they are difficult to construct and observe. Moreover, the twisted and mixed groups are closely related with one another: one arises from the other by a form of base change along the extension $\mathbf F_p / \mathbf F_{\sqrt p}$ (although we will use different notations). The exotic groups of Conrad, Gabber and Prasad are then closely related to these invisible mixed reductive groups, in practically the same way that most of the pseudo-reductive groups are related to reductive groups.  Finally, the exotic groups of Weiss are one class of groups for which we do not explicitly show that they can be described with the tools under our belt. We believe that there is enough evidence to be convinced that this is indeed the case, but it may require some technical virtuosity to show this rigorously.

\subsection*{Organisation of the paper}

\begin{enumerate}
	\item Section~\ref{sec:preliminaries} provides some preliminaries. We first provide a colloquial description of the classes of groups that we eventually wish to investigate. We avoid going  too deeply into the technicalities here---a more technical description can be found as a first step in the proofs of our main theorems in Section~\ref{sec:gps}---but rather we will focus on what we believe are shortcomings of these descriptions. Note that some of the history and state of affairs was incorporated in the historical overview Section~\ref{sec:history}. We then point out some reference works in category theory which contain all that we need for our purposes.
	\item Section~\ref{sec:mixing-objects} studies categories endowed with an endomorphism $F$ of the identity functor. Sections~\ref{sec:def-cat} and~\ref{sec:functors} certainly contain the core ideas of our work by introducing the twisted and mixed categories and functors between them. It may be advisable to follow up with Sections~\ref{sec:def-sch} and~\ref{sec:examples} immediately to see how these definitions are used in practice.
	
	Sections \ref{sec:twisted-descent} and \ref{sec:cat-psh} do not play an important role in what follows directly, but they are nonetheless of fundamental importance in the theory.  Section \ref{sec:twisted-descent} explains how mixed objects can descend to twisted objects---for instance after everything else is proven it explains how one can attach a descent datum to a mixed group of type $\mathsf G_2$ to obtain a Ree group of type ${}^2\mathsf G_2$. Section \ref{sec:cat-psh}, in particular Remark~\ref{remark:mixtor}, explains how mixed objects can `stay hidden' in a precise way. Although the technical content of this remark is just that ordinary objects form a reflective subcategory of mixed objects, we draw attention to it because it lies at the basis of the observation that some exotic pseudo-reductive groups arise as one half of a mixed group; see also Remark \ref{rem:mixed-groups}.\ref{bullet:c2appears}.
	
	In Section \ref{sec:relfac} and in particular Corollary~\ref{cor:relfac}, we explain how certain mixed objects can be constructed from data within the ordinary category---this lies at the basis of Proposition~\ref{prop:veryspecialisogeny2} in Section \ref{sec:existence} where we will use this to construct interesting mixed group schemes.
	
	Sections \ref{sec:bookkeeping} and \ref{sec:restriction} are really just preparation for our Theorem~\ref{theorem:pseudo-reductive-groups}, the key being Proposition~\ref{prop:weilrestriction}. In fact, Section~\ref{sec:bookkeeping} has nothing to do with twisted and mixed categories and can be read independently; it was included because our  proof of Proposition~\ref{prop:weilrestriction} requires some arrow manipulation that we could not justify any other way.
	
	\item In Section~\ref{sec:mixed-schemes} the definitions from section \ref{sec:mixing-objects} are applied to the category of schemes in characteristic $p$, with the absolute Frobenius in the role of $F$. 
	
	We set the stage in Section~\ref{sec:def-sch} where we define twisted and mixed schemes and rings. In Section~\ref{sec:sch-ratpts} we give Propositions~\ref{prop:fixedpoints} and~\ref{prop:mixedpoints}, which describe the rational points of twisted and mixed schemes. Although these results could also be formulated for an arbitrary category, we believe the discussion benefits from  being concrete, with a view towards the examples in Section~\ref{sec:examples}. Section~\ref{sec:modules-and-sheaves} contains a brief digression which serves to introduce the notion of partial dimensions of mixed scheme. Although this is not used in any of our main theorems, we will sometimes refer to it in comments, such as \ref{rem:mixed-groups}.\ref{bullet:dimensions}, to provide some justification that our results are  morally desirable in one way or another. Finally, as already noted, Section~\ref{sec:examples} groups together some examples which are intended to demonstrate that the theory works as intended.
	
	\item Section~\ref{sec:gps} focuses on group objects in the categories of twisted and mixed schemes and provides precise statements and proofs which connect our categorical constructions with known constructions due to Ree, Tits and Conrad--Gabber--Prasad. 
	
	In Section~\ref{sec:gps-statement} we will first informally present our theorems. We believe this is necessary because the state of affairs with respect to twisted and mixed groups in the literature is sufficiently confused that one cannot simply write down and prove a precise statement that will be universally understood.  Nonetheless, we have tried to be as brief as possible and postponed an extensive discussion to Section~\ref{sec:history}. For pseudo-reductive groups, the situation is a bit better, thanks to the work of Conrad--Gabber--Prasad so can we go over these groups even more quickly.
	
	In Section~\ref{sec:existence} we use the constructions from Section~\ref{sec:relfac} together with the \emph{very special isogenies} from \cite{PRG} to construct a number of mixed groups schemes. In Section~\ref{sec:twisted-groups} we formulate and prove Theorem~\ref{theorem:twisted-groups}, which states that the twisted groups arise as groups of rational points of a twisted group scheme; the main ingredient is Proposition~\ref{prop:fixedpoints}. In Section~\ref{sec:mixed-groups}, we formulate and prove Theorem~\ref{theorem:mixed-groups}, which states that the mixed groups arise as groups of rational points of a mixed group scheme from Section~\ref{sec:existence}; the main ingredient here is Proposition~\ref{prop:mixedpoints}. In section~\ref{sec:gps-prg}, we formulate and prove Theorem~\ref{theorem:pseudo-reductive-groups}, which states that the exotic pseudo-reductive groups arise as a Weil restriction of a mixed group scheme from Section~\ref{sec:existence}; the main ingredient here is Proposition~\ref{prop:weilrestriction}.
	
	\item Section~\ref{sec:history} contains a historical overview of many discoveries and results related to mixed and twisted groups. We find such an overview indispensible for understanding the present work for three reasons. First of all, although quite a lot has been written about twisted and mixed groups, often from the point of view of incidence geometry and the theory of buildings, much of it comes down to the study of specific examples. The literature is very vague on what the words mixing and twisting actually mean and on whether one is a special case of the other or not. For instance, one of most important works in the field is \cite{moufangpolygons}, which puts Moufang octagons ---a manifestly twisted structure--- on the list with the mixed buildings. So we attempt to shine some light on this situation. A second reason is that we have many expectations on how some facts from the literature tie in with our own work.  We needed a way to structure these ideas to give other people the opportunity to work with them.  For instance, a mixed hexagon can be given the structure of a mixed projective variety; the Ree unital has the structure of a twisted projective variety and it arises from the mixed hexagon by twisted descent. So perhaps some of the success of Galois descent will carry over to twisted descent and eventually allow for a better understanding of the Ree unital. Finally, we attach great importance to attributing the ideas that were important for our work to the right people, so we hope we achieved at least that and apologize for any omissions.
\end{enumerate}

\subsection*{Acknowledgments}
This article presents results from the authors PhD; I would like to thank my advisor Tom De Medts for suggesting mixed groups as a research topic and for his advice and support throughout. Furthermore I wish to thank Hendrik Van Maldeghem for proofreading \S\ref{sec:history} and  Michel Brion for asking an interesting question after a talk on the subject in Lens.

\section{Preliminaries} \label{sec:preliminaries}
\subsection{Groups}
In 1960, Suzuki found an infinite class of finite simple groups, explicitly described by matrices as  subgroups of $\mathsf{GL}(4,2^{2e+1})$. Not much later, Ree showed how these Suzuki groups can be obtained from a Chevalley group of type $\mathsf B_2$ as follows: if the characteristic of the underlying field $k$ field is $2$, the Chevalley group has an exceptional graph automorphism $g$ with the strange property that it squares to the Frobenius endomorphism. So if we assume that the field $k$ admits an automorphism $\phi$ such that $x^{\phi^2}=x^2$, then $\phi^{-1} \circ g$ is an automorphism of order $2$ of the group $\mathsf B_2(k)$ and its fixed points form a subgroup ${}^2\mathsf B_2(k,\phi)$ which is a Suzuki group. Extending this procedure to Chevalley groups of type $\mathsf G_2$ in characteristic $3$ --- where the condition on $\phi$ becomes $x^{\phi^2}=x^3$ --- and $\mathsf F_4$ in characteristic $2$, Ree found the \emph{small Ree groups} ${}^2\mathsf G_2(k,\phi)$  and the \emph{large Ree groups} ${}^2\mathsf F_4(k,\phi)$. Somewhat later, Tits showed how to define these groups over non-perfect fields, when $\phi$ is a non-invertible endomorphism of $k$. These groups are now known as \emph{twisted (Chevalley) groups} or as the Suzuki-Ree groups; an endomorphism of a field which squares to the Frobenius endomorphism is often called a \emph{Tits endomorphism}. When we say \emph{twisted group}, we always mean an abstract group from one of these families.

The second exotic class of groups which merits our attention has not been studied nearly as well: the \emph{groups of mixed type} or \emph{mixed groups} for short. Besides a parenthetical remark in Steinbergs lecture notes \cite[153]{steinberg}, their first appearance in the literature seems to be in Tits' lecture notes \cite[(10.3.2)]{tits74} on buildings. Tits introduces buildings as a tool to study algebraic groups, and achieves a complete classification of the important class of spherical buildings of rank $\geq 3$. The classification does mostly what it was intended to do: \emph{with some exceptions} all of these buildings originate from algebraic groups or classical groups.\footnote{For instance $\mathsf{PSL}_n(D)$ is classical non-algebraic if $D$ is a division ring of infinite dimension over its center.} The exceptions are the \emph{buildings of mixed type}. These are buildings of type $\mathsf X_n = \mathsf B_n/\mathsf C_n$, $\mathsf F_4$ or $\mathsf G_2$ whose definition requires a \emph{pair of fields} $k, \ell$ of respective characteristic $p=2$, $2$ or $3$ such that $\ell^p \subseteq k\subseteq \ell$.  In a nutshell, Tits' construction comes down to observing that one can restrict some of the generators of the group $\mathsf X_n(\ell)$ to the subfield $k$ in a way which preserves the commutation relations which define the group. In this way one gets a subgroup which acts on a subbuilding. Tits' notation for these abstract groups is then $\mathsf X_n(k,\ell)$, where of course $\mathsf X_n(\ell,\ell)=\mathsf X_n(\ell)$. Since these groups act on a building, there is a corresponding BN-pair which can be used to study them.

Several problems arise with these mixed groups. First and foremost, they are merely abstract groups and there is no corresponding algebraic group. Second, since they are defined very explicitly by means of a set of generators, the construction only works well for split groups. Third, it was a favorite observation of Tits that the fields $k$ and $\ell$ involved play essentially the same role! Actually they are just two random consecutive members of the infinite sequence of inclusions 
\[ \dots \subseteq  \ell^{p^2} \subseteq k^p \subseteq \ell^p \subseteq k \subseteq \ell \subseteq k^{1/p} \subseteq \ell^{1/p} \subseteq \cdots  \]
Moreover, $\mathsf B_n(k,\ell) \cong \mathsf C_n(\ell^2,k)$, and $\mathsf X_n(k,\ell)\cong\mathsf X_n(\ell^p,k)$ for $\mathsf X_n$ equal to $\mathsf G_2$ or $\mathsf F_4$ and $p$ equal to $3$ or $2$. (And of course $\ell^p \cong \ell$.) But this phenomenon is not apparent from the description, where one field is the larger field and one the smaller.  Finally, it is intuitively clear that there is some connection with the twisted groups, but it seemed difficult to make this precise. 

Finally, we introduce the class of pseudo-reductive groups. Recall that a linear algebraic group $G$ over a field $k$ is reductive if its unipotent radical becomes trivial after base change to the algebraic closure $\bar k$. The corresponding notion over $k$ is weaker and called \emph{pseudo-reductivity}. Where the structure theory of reductive groups has been an important part of algebraic group theory for more than half a century, a structure theory for pseudo-reductive groups is a rather recent addition due to Conrad--Gabber--Prasad \cite{PRG}, building on some older work of Tits. Although we found the structure theory in its entirety not particularly accessible, the gist of it can be phrased rather elegantly by saying that that \emph{most} pseudo-reductive groups arise from a certain standard construction. The standard construction takes as input a reductive group $G'$ over a field $k'$, and a purely inseparable field extension $k'/k$. It then applies a process known as Weil restriction to $G'$ to obtain a group $G = \mathsf R_{k'/k}G'$ over $k$ which is pseudo-reductive. In fact, this is only the easy part of the standard construction: there is another important step which consist of replacing the Cartan subgroups by a different commutative group which satisfies certain conditions, but this will not play a role in our work and we refer to \cite[\S1.4]{PRG} for details. Exceptional groups which do not arise from the standard construction exist only in characteristics~$2$ and~$3$. A first class of exceptions consists of the \emph{exotic pseudo-reductive groups}, which are constructed and studied in Chapter~7 of \cite{PRG}. The authors first construct what they call \emph{very special isogenies} between certain algebraic groups and then perform a rather elaborate construction which starts from such an isogeny $\pi:G\to\bar G$ together with a purely inseparable field extension $k'/k$ to produce a pseudo-reductive group $\mathscr G$. In their own words, the construction roughly comes down ``thickening the short root groups from $k$ to $k'$ and part of the torus from $k^\times$ to $k'^\times$''. Clearly as an abstract group, this should correspond to the mixed group $\mathsf X_n(k,k')$ and in fact, we have no doubt that this is how Tits found them! Of course, this doesn't really explain what is going on. It is certainly not very elegant that most of Tits' spherical buildings correspond to reductive groups, but the mixed buildings suddenly require a few pseudo-reductive groups to finish the picture---not to mention the problems that arise when $[k':k]$ is infinite. This is another issue which we address in our work by constructing these exotic pseudo-reductive groups via a standard construction---in fact simply a purely inseparable Weil restriction---but performed in the category of mixed schemes. Tits' spherical buildings can then still correspond to (mixed) reductive groups and the standard construction of Conrad--Gabber--Prasad becomes even more ubiquitous since it also produces the exotic examples.

\subsection{Categories}
We will rely heavily on the language and elementary properties of categories; in particular we will frequently use the notions of slice categories, (co)limits, adjoint functors and the Yoneda embedding. The uninitiated reader may find it useful to keep a reference at hand. The standard reference in the field is \cite{CWM}, but we also suggest the more recent \cite{leinster} which is also freely available online on the arXiv. It covers less ground but is more accessible and contains everything that is required for our purposes. Finally, expos\'es I and II of \cite{SGA3} and the nLab on \url{ncatlab.org} are highly recommended for some more in-depth coverage of certain topics.

We will mainly borrow the notations of the French school of algebraic geometry: we denote a slice category of a category $\C$ over an object $X$ by $\C_{/X}$; we denote the structural morphism of an object $Y\in\ob(\C_{/X})$ typically by $q_Y$ (this makes $Y$ an abuse of notation for $q_Y$); we denote the categories of sets, rings, schemes by $\Set$, $\Ring$, $\Sch$; we use the arrow $\leadsto$ to define a functor on objects if we leave the definition on arrows to the reader; we denote the category of presheaves on $\C$ by $\Cdual=\Hom(\C^\op,\Set)$; we denote the internal hom in a category by the boldface $\Hom$; we denote the Yoneda embedding by $\C\to\Cdual:X\leadsto \mathbf h_X$ with $\mathbf h_X(Y)=\hom(Y,X)$; we denote the functor of base change along an arrow $f$ in a category which admits fibered products by $f^*$; we denote an adjunction of functors $L:\C\to\D$ and $R:\D\to\C$ by $L\dashv R$ and we denote the unit and counit transformations of an adjunction typically by $\eta$ and $\epsilon$.

We hope secretly that the categories in our text will serve as a motivation for interested group theorists and incidence geometers to learn the language, rather than a motivation not to read the text.

\section{Twisting and mixing objects} \label{sec:mixing-objects}
In this section $\C$ denotes an arbitrary category endowed with an endomorphism of the identity functor $F: \mathrm{id}_{\C}\to\mathrm{id}_{\C}$, this means that
for every object $X$ there is an endomorphism $F_X\in \End_\C(X)$ such that for every arrow $f:X\to Y$ in $\C$, we have $F_Y\circ f = f\circ F_X$.

\subsection{The twisted and mixed categories} \label{sec:def-cat}

\begin{definition} \label{def:tC} The \emph{twisted category} $t\C$ is defined as follows. The objects are the pairs $\tilde X = (X,\Phi_X)$ where $X\in\ob(\C)$ and $\Phi_X\in\End_\C(X)$ satisfies $\Phi_X\circ\Phi_X=F_X$. The morphisms $f:\tilde X\to\tilde Y$ are those morphisms $f:X\to Y$ for which $\Phi_Y\circ f=f\circ\Phi_X$.

For a twisted object $\tilde X=(X,\Phi_X)$, we call $X$ the \emph{underlying ordinary object} and $\Phi_X$ the \emph{twister}.
Note that $t\C$ is itself a category with an endomorphism $\Phi$ of the identity functor, and that there is a forgetful functor $\mathbf f : t\C\to\C :  \tilde X\leadsto X$
which we call the \emph{untwisting functor}.

\end{definition}

\begin{definition} \label{def:mC} The \emph{mixed category} $m\C$ is defined as follows. The objects are the quadruples \[\tilde X = (X_1, X_2, \Phi_{X_1}, \Phi_{X_2}) \] where $X_1,X_2\in\ob(\C)$ and $\Phi_{X_i}\in \hom_\C(X_i,X_{2-i})$ satisfy $\Phi_{X_{2-i}}\circ\Phi_{X_i}=F_{X_i}$. The morphisms $f:\tilde X\to\tilde Y$ are those pairs $(f_1,f_2)$ of morphisms $f_i:X_i\to Y_i$ for which $\Phi_{Y_i} \circ f_i = f_{2-i}\circ \Phi_{X_i}$.
\end{definition}

We will depict a morphism of mixed objects diagrammatically as

\[ \begin{tikzcd}
\tilde X \dar["f"] \ar[dr,phantom, "\text{or}" description] & X_1 \rar[shift left=.5ex,"\Phi_{X_1}"] \dar["f_1"] & X_2 \lar[shift left=.5ex,"\Phi_{X_2}"] \dar["f_2"] \\
\tilde Y & Y_1 \rar[shift left=.5ex,"\Phi_{Y_1}"] & Y_2. \lar[shift left=.5ex,"\Phi_{Y_2}"]
\end{tikzcd} \]

This is not a commutative diagram since the pair of arrows\begin{tikzcd} \bullet \rar[shift left=0.5ex] & \circ\lar[shift left=0.5ex]\end{tikzcd} does not compose to the identity but rather to $F_\bullet$ and $F_\circ$; one should think of it as an abbreviation for the bigger diagram
\[
\begin{tikzcd}
	X_1 \dar["f_1"] \rar[swap,"\Phi_{X_1}"]\ar[rr,bend left,"F_{X_1}"] & X_2 \dar["f_2"]\rar[swap,"\Phi_{X_2}"] \ar[rr, bend left, "F_{X_2}"] & X_1\dar["f_1"]\rar[swap,"\Phi_{X_1}"] & X_2\dar["f_2"] \\
	Y_1 \rar["\Phi_{Y_1}"]\ar[rr,bend right,swap,"F_{Y_1}"] & Y_2\rar["\Phi_{Y_2}"] \ar[rr, bend right, swap, "F_{Y_2}"] & X_1\rar["\Phi_{Y_1}"] & Y_2.
\end{tikzcd}
\]
The maps $\Phi_{X_1}$ and $\Phi_{X_2}$ are called the \emph{mixing maps} or \emph{mixers}. If they are clear from the context, we will also denote $\tilde X$ simply by $(X_1,X_2)$.

To make constructions, it is important that some of the good properties of $\C$ are carried over to $m\C$ and $t\C$. For instance, if $\C$ admits fibered products $X\times_S Y$ or coproducts $X\sqcup Y$, we would like the same thing to be true for $t\C$. The following lemma reassures us that this is always the case.

\begin{lemma} \label{prop:limits} If $\C$ admits (co)limits for diagrams of shape $\mathscr J$, then so do $t\C$ and $m\C$. \end{lemma}
\begin{proof} Consider a diagram $\mathbf D:\mathscr J\to t\C$, set $\mathbf D'=\mathbf{f}\circ\mathbf D$ and let $X=\lim \mathbf D'$ together with the morphisms $\chi_U : X\to \mathbf D'(U)$ for every  $U\in\ob(\mathscr J)$. Then the object $X$ together with the morphisms $\mathbf{f}(\Phi_{\mathbf D(U)})\circ\chi_U$ forms a cone and so defines a unique morphism $\Phi_X : X\to X$. Moreover, it is clear that the morphism  $\Phi_X\circ\Phi_X$ is the unique morphism which makes the cone coming from all $\Phi_{\mathbf D(U)}^2$ commute, but the morphism $F_X$ has this property as well so $\Phi_X\circ\Phi_X=F_X$. It is then immediately verified that $(X,\Phi_X)$ is a limit for $\mathbf D$.
An analogous argument holds for colimits. The proof for $m\C$ is similar but using each of the functors $(X_1,X_2,\Phi_{X_1},\Phi_{X_2})\leadsto X_i$ to construct an appropriate object. (We omit the proof because we will only use this in a context where it follows directly from the case of $t\C$, thanks to Proposition~\ref{prop:equivalenceQ}.)
\end{proof}

In particular, if $\C$ has a terminal object $\mathbf 1$, then $\mathbf 1_{t\C}=(\mathbf 1,\mathrm{id}_{\mathbf 1})$ is a terminal object for $t\C$ and $\mathbf 1_{m\C}=(\mathbf 1,\mathbf 1,\mathrm{id}_{\mathbf 1},\mathrm{id}_{\mathbf 1})$ is a terminal object for $m\C$.

The categories $m\C$ and $t\C$ are closely related: we will now observe that under a mild assumption $m\C$ is a slice category of $t\C$. Our assumption is that $\C$ is an \emph{extensive category}, which means:
 \begin{itemize}
 \item $\C$ admits small limits; in particular, there is a terminal object $\mathbf 1$ and there are fibered products $X\times_S Y$;
 \item $\C$ admits binary coproducts $X\sqcup Y$;
 \item The functor $\C_{/A} \times\C_{/B} \to \C_{/A\sqcup B} : (X,Y)\leadsto (X\sqcup Y)$ is an equivalence for every two objects $A$ and $B$.
 \end{itemize}

An inverse to this equivalence is then given by
\[ \C_{/A\sqcup B} \to \C_{/A} \times\C_{/B} :  T \leadsto \big(  T \underset{A\sqcup B}\times A ,  T \underset{A\sqcup B}\times B\big).  \]

Under these conditions $\mathbf 2=\mathbf 1\sqcup \mathbf 1$ exists in $\C$ and $\mathsf E=(\mathbf 2,\tau)$, where $\tau$ swaps the components, is an object of $t\C$. So  there exists a functor
\[ \mathsf Q :   m\C \to (t\C)_{/\mathsf E} : (X_1,X_2,\Phi_{1},\Phi_{2}) \leadsto (X_1\sqcup X_2,\Phi_{1}\sqcup\Phi_{2}),  \]
where the implied structural morphism $X_1\sqcup X_2\to \mathbf 2$ is the obvious one. So we conclude:

\begin{proposition} \label{prop:equivalenceQ} If $\C$ is extensive, then there is an equivalence $\mathsf Q : m\C \to (t\C)_{/\mathsf E}$. Consequently, for every $\tilde M\in \ob(m\C)$, there is an equivalence $(m\C)_{/\tilde M} \simeq (t\C)_{/\mathsf Q(\tilde M)}$.
\end{proposition}

\subsection{Many functors} \label{sec:functors}

\begin{definition} \label{def:functors}
We now define a number of functors between the categories $\C$, $t\C$ and $m\C$. The following overview will be helpful:
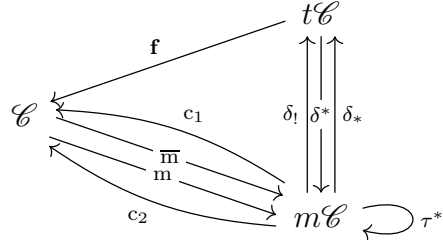
\[
\begin{tikzcd}[column sep=3cm]
 & t\C \ar[dd,xshift=0.0ex,"\delta^*" description] \ar[dl,swap,shift right=0.5ex,"\mathbf f"] \\
 \C\ar[dr,shift right=0.75ex,swap,"\mix" description]\ar[dr,shift left=0.75ex,"\antimix" description] & \\
 & m\C   \ar[uu,xshift=-1.0ex,"\delta_!"] \ar[uu,xshift=1.0ex,swap,"\delta_*"]  \ar[ul,bend left=15,shift left=0.75ex,"\comp_2"] \ar[ul,shift right=0.75ex,bend right=15,swap,"\comp_1"] \arrow[loop right]{}{\tau^*}
\end{tikzcd}
\]
\begin{enumerate}[label={\rm (\roman*)}]
	\item  We already introduced the \emph{untwisting} functor, which forgets the twister: 
	\[ \mathbf f:t\C\to\C : (X,\Phi_X) \leadsto X.
	\] 
	\item We have two functors  $\delta_!:m\C\to t\C$ and  $\delta_*:m\C\to t\C$ called the \emph{left and right misting functors} (\emph{mi}xed to tw\emph{ist}ed). They are defined by
\begin{alignat*}{2}
	\delta_! &: (X_1,X_2,\Phi_{1},\Phi_2) && \leadsto (X_1 \sqcup X_2,\Phi_1\sqcup\Phi_2), \\ 
	\delta_* &: (X_1,X_2,\Phi_{1},\Phi_2) && \leadsto (X_1 \times X_2,\Phi_1\times\Phi_2).
\end{alignat*}
	\item The \emph{twixing functor} (\emph{twi}sted to m\emph{ix}ed) $\delta^*:t\C\to m\C$ and the \emph{twisting functor} $\tau^*:m\C\to m\C$ are most easily expressed as
\begin{alignat*}{2}
	\delta^* &: (X,\Phi) &&\leadsto (X,X,\Phi,\Phi), \\
	\tau^* &: (X_1,X_2,\Phi_{1},\Phi_{2}) &&\leadsto (X_2,X_1,\Phi_{2},\Phi_{1}).
\end{alignat*}
Alternatively, using $(t\C)_{/\mathsf E}$ instead of $m\C$, they are base changes in $t\C$ along
\[ \delta : \mathsf E \to \mathbf 1 \text{ and } \tau : \mathsf E\to\mathsf E, \]
i.e. $\delta^*(X) = X \times \mathsf E$ and $\tau^*(X)= X\times_{q_X,\tau}\mathsf E$, hence the notation.

	\item We define the \emph{component functors} by
  \[ \comp_i : m\C \to \C : (X_1,X_2,\Phi_{X_1},\Phi_{X_2}) \leadsto  X_i.  \]
Alternatively, using $(t\C)_{/\mathsf E}$ instead of $m\C$ they are base changes in $\C$ along the inclusions $\mathrm{inc}_i : \mathbf 1_{\C}\to\mathbf 2_{\C}$, namely $\comp_i(\tilde X) = \mathbf f(\delta_!(\tilde X)) \times_{q_X,\mathrm{inc}_i} \mathbf 1_\C$.
	\item Finally, we define the \emph{mixing and anti-mixing functors} by
 \begin{align*}
  \mix &: \C\to m\C : X \leadsto (X,X,F_X,\mathrm{id}_X) \text{ and }  \\
 \antimix &: \C\to m\C : X \leadsto (X, X,\mathrm{id}_X,F_X).
\end{align*}
\end{enumerate}
\end{definition}

\begin{proposition} We have the following properties:
\begin{enumerate}[label={\rm (\roman*)}]
\item There are adjunctions $\comp_1\dashv\mix\dashv\comp_2\dashv\antimix\dashv\comp_1$.
\item Each of the functors $\comp_1$, $\comp_2$, $\mix$ and $\antimix$ preserves all limits and colimits (hence in particular products, coproducts and terminal objects).
\item The functors $\mix$ and $\antimix$ are full and faithful.
\item There is an adjunction $\delta_! \dashv \delta^* \dashv \delta_*$.
\item $\tau^*$ is an equivalence and $\tau^*\circ\tau^* \simeq \mathrm{id}_{m\C}$.
\item $\tau^*\circ\mix = \overline m$; $\comp_i = \comp_{2-i}\circ\tau$; $\delta^*=\tau^* \circ \delta^*$.
\end{enumerate}
\end{proposition}
Note that $\tau^*\circ\tau^*$ \emph{is} the identity functor on $m\C$, whereas it is merely isomorphic to it when considered on $(t\C)_{/\mathsf E}$.

\begin{proof}
\begin{enumerate}[label={\rm (\roman*)}]
\item  We will only verify that $\comp_1\dashv\mix$; the other pairs are completely analogous. Consider objects $\tilde X=(X_1,X_2,\Phi_{X_1},\Phi_{X_2})\in\ob(m\C)$, $Y\in\ob(\C)$ and the map
\[ \hom_{m\C}(\tilde X,\mix(Y)) \to \hom_{\C}(\comp_1(X),Y) : (\alpha,\beta) \mapsto \alpha.   \]
Clearly the map $\alpha\mapsto (\alpha,\alpha\circ\Phi_{X_2})$ is an inverse and these bijections are natural in $\tilde X$ and $Y$.
\item Since each of the functors $\comp_1$, $\mix$, $\comp_2$ and $\antimix$ is now both a left an right adjoint, they preserve all limits, colimits, epimorphisms and monomorphisms.
\item Clearly the counit $Y\to (\comp_1\circ\mix)(Y)=Y$ is the identity, which implies that $\mix$ is full and faithful by \cite[(IV.3.1)]{CWM}. A similar argument holds for $\antimix$.
\end{enumerate}
The remaining statements are obvious from the definitions.
\end{proof}

\begin{definition} \label{def:visible} Since $\mix$ (resp.\@ $\antimix$) is fully faithful, its essential image is equivalent to $\C$, so we call the mixed objects
isomorphic to $\mix(X)$ (resp.\@ $\antimix(X)$) for some $X \in \ob(\C)$ \emph{visible} (resp.\@ \emph{anti-visible}). The mixed objects that are not visible are called \emph{invisible}. In what follows we will occasionally identify an object $X\in\ob(\C)$ with the corresponding visible object $\mix (X)\in \ob(m\C)$. In other words, we consider $\C$ as a full subcategory of $m\C$ through $\mix$.
\end{definition}

\begin{remark} \label{remark:visible} The verification of the following observations is straightforward and left to the reader.
\begin{enumerate}[label={\rm (\roman*)}]
\item A mixed object $(X_1,X_2,\Phi_1,\Phi_2)$ is visible (resp.~anti-visible) if and only if the map $\Phi_2$ (resp.~$\Phi_1$) is an isomorphism.
\item  \label{bullet:mixtor} If $X\in\ob(m\C)$ is visible then we have the following bijection, natural in $X$ and $\tilde Y$:
\[ \hom(X,\tilde Y) \to \hom(X,\comp_2\tilde Y) : (f_1,f_2)\mapsto f_2. \]
Note that the hom-set $\hom(X,\comp_2\tilde Y)$ can be interpreted in either $\C$ or $m\C$, but in the latter case the map is given by $(f_1,f_2)\mapsto (f_2,f_2)$ instead.
\item \label{bullet:ratpts} If $\tilde X=(X_1,X_2)$ and $\tilde Y=(Y_1,Y_2)$ are mixed objects, then there is a fibered product in $\Set$
	\[ \hom(\tilde X,\tilde Y) = \hom(X_1,Y_1) \underset{\hom(X_1,Y_2)\times\hom(X_2,Y_1)}\times \hom(X_2,Y_2),   \]
  where the maps are $u \mapsto ( \Phi_{Y_1}\circ u, u\circ \Phi_{X_2})$ and $v \mapsto ( v\circ \Phi_{X_1}, \Phi_{Y_2}\circ v)$. 
 \item The projection of this fibered product to its first component corresponds to the map
 \[ \comp_1 : \hom(\tilde X,\tilde Y) \to \hom(X_{1},Y_{1}) : (f_1,f_2)\mapsto f_1.  \]
 If $F_{X_{2}}$ is epic or $F_{Y_{2}}$ is monic, the reader may verify that this map is injective. We will postpone a much more precise version of this statement until Proposition~\ref{prop:mixedpoints}. 

\end{enumerate}
\end{remark}

We conclude this section with the following proposition, which claims good behaviour of all these constructions under functors that are sufficiently nice.
\begin{proposition} \label{prop:functoriality}
	The formation of $t\C$ and $m\C$ commutes with $\op$: $(t\C)^\op = t\C^\op$ and $(m\C)^\op = m\C^\op$. Furthermore, if $\mathbf G:\mathscr D\to\C$ is a functor between categories with endomorphisms of the identify functor, both denoted by $F$ and $H$, such that $\mathbf G(F_X)=H_{\mathbf G(X)}$ for all $X \in \ob(\mathscr D)$, then there are functors $t\mathbf G \colon t\mathscr D \to t\C$ and $m\mathbf G \colon m\mathscr D \to m\C$ which commute with all the functors defined in \ref{def:functors}, i.e. with  $\mathbf f$, $\comp_1$, $\comp_2$, $\mix$, $\antimix$, $\delta_!$, $\delta^*$, $\delta_*$ and $\tau^*$. For instance, the diagram
	\[ \begin{tikzcd}
	{m\mathscr D}\arrow{r}{\delta_!} \arrow{d}{m\mathbf G} & {t\mathscr D}\arrow{r}{{\mathbf f}}\arrow{d}{t\mathbf G} & \mathscr D\arrow{d}{\mathbf G} \\
	m\C\arrow{r}{\delta_!} & t\C\arrow{r}{{\mathbf f}} & \C
	\end{tikzcd} \]
	commutes.
\end{proposition}
\begin{proof}
	This is immediately clear from the definitions.
\end{proof}

\subsection{Twisted descent} \label{sec:twisted-descent}
We will now study the twixing functor $\delta^*:t\C\to m\C$ in greater detail and characterize its essential image. Considered as a functor of base change $\delta^* : (t\C)_{/\mathbf 1}\to (t\C)_{/\mathsf E}$, this is an instance of a descent problem. A motivation is that the category $m\C$ is much better behaved than $t\C$ thanks to the (anti)-mixing and component functors, so we prefer to perform constructions in $m\C$ and understand how they descend to $t\C$.

\begin{definition} A \emph{(twisted) descent datum} on $\tilde X\in\ob(m\C)$ is a morphism $f: \tilde X\to \tau^*\tilde X$ such that $\tau^* f\circ f = \mathrm{id}_{\tilde X}$.  We form a category $m\C[\tdd]$ where objects are pairs $(\tilde X,f)$ consisting of an object $\tilde X$ of $m\C$ together with a descent datum $f$ on $\tilde X$; an $m\C[\tdd]$-arrow $u:(\tilde X,f)\to (\tilde Y,g)$ is an $m\C$-arrow $u:\tilde X\to\tilde Y$ such that $g\circ u=\tau^*u\circ f$.
\end{definition}

Since applying $\tau^*$ to $\tau^* f\circ f = \mathrm{id}_{\tilde X}$ yields $ f\circ \tau^* f = \mathrm{id}_{\tau^*\tilde X}$, we see that a descent datum $f$ is always an isomorphism, with $\tau^*f$ as its inverse.

\begin{proposition} The functor $\delta^*$ factors as
	\[  t\C \overset{\alpha}{\longrightarrow} m\C[\tdd]\overset{\mathrm{forget}}{\longrightarrow}m\C\]
	where $\alpha$ is an equivalence and $\mathrm{forget}: (\tilde X,f)\leadsto \tilde X$ forgets the descent datum.
\end{proposition}
\begin{proof}
  Since $\delta^* = \tau^*\delta^*$, the identity is a map  $\mathrm{id} : \delta^* \tilde X \to \tau^*\delta^*\tilde X$, for every $\tilde X\in\ob(t\C)$ which is trivially a descent datum on $\delta^* \tilde X$. So we define
  \[  \alpha : t\C\to m\C : \tilde X \leadsto (\delta^* \tilde X, \mathrm{id}),  \]
  and it is clear that composing $\alpha$ with the functor forgetting the descent datum is indeed $\delta^*$. From this it is also clear that $\alpha$ is faithful.

  We now show that $\alpha$ is essentially surjective. Consider an arbitrary object $(\tilde X,f)\in\ob(m\C[tdd])$, and write $\tilde X = (X_1,X_2,\Phi_1,\Phi_2)$ and $f=(f_1,f_2)$. Then the diagram
  \[
   \begin{tikzcd} X_1 \rar["\Phi_1",shift left=.5ex] \dar["f_1"] & X_2 \dar["f_2"] \lar["\Phi_2", shift left=.5ex] \\ X_2 \rar["\Phi_2",shift left=.5ex] & X_1 \lar["\Phi_1", shift left=.5ex] \end{tikzcd} \]
  commutes, and moreover $f_1=f_2^{-1}$.  It follows that also the diagram
  \[
   \begin{tikzcd} X_1 \rar["\Phi_1",shift left=.5ex] \dar["f_1"] & X_2 \dar["\mathrm{id}"] \lar["\Phi_2", shift left=.5ex] \\ X_2 \rar["\Phi_1\circ f_2",shift left=.5ex] & X_2 \lar["\Phi_1\circ f_2", shift left=.5ex] \end{tikzcd}
  \]
 commutes.  Hence there is an isomorphism $(f_1,\mathrm{id}):\tilde X \to \delta^*(X_2,\Phi_1\circ f_2)$ which respects the descent data $(f_1,f_2)$ and $(\mathrm{id},\mathrm{id})$, so it determines an isomorphism
  $(\tilde X,f) \cong \alpha(X_2,\Phi_1\circ f_2)$ in $m\C[\tdd]$.
 Thus $\alpha$ is essentially surjective.

 Finally, we show that $\alpha$ is full. Let $u:(\tilde X,f)\to (\tilde Y,g)$ be a morphism in $m\C[\tdd]$. Since $\alpha$ is essentially surjective, we may as well assume that $\tilde X = \delta^*(X,\Phi)$ and $\tilde Y=\delta^*(Y,\Psi)$ with in both cases the identity as descent datum. The morphism $u = (u_1,u_2)$ respects the descent data, which boils down to $u_2=u_1$ and then $u = \alpha(u_1)$ so $\alpha$ is full.
 \end{proof}

The following corollary is particularly useful to remember.
\begin{corollary} \label{cor:twisted-descent}
	A mixed object $\tilde X$ has a descends to a twisted object if and only if it has an endomorphism $f:\tilde X\to\tilde X$ such that $\tau^*f\circ f = \mathrm{id}_{\tilde X}$.
    In particular, it is necessary that $\comp_1(\tilde X)\cong\comp_2(\tilde X)$.
\end{corollary}

As an example of how this is useful, let us mention that  later in this article we will construct what we call \emph{mixed algebraic groups} $\tilde X$ of type $ (\mathsf B_n,\mathsf C_n)$. The components $\comp_1(\tilde X)$ and $\comp_2(\tilde X)$ are isomorphic only when $n=2$, so the group can only admit twisted descent in this case. It turns out that when $n=2$, the group indeed admits twisted descent and this produces the Suzuki groups ${}^2\mathsf B_2$. 

\subsection{Categories of presheaves} \label{sec:cat-psh}
The category of presheaves on $\C$ is denoted by $\hat\C=\Hom(\C^\op,\Set)$. It is canonically endowed with an endomorphism of the identity functor $\mathbf F : \mathrm{id}_{\Cdual} \to \mathrm{id}_{\Cdual}$ which respects the Yoneda embedding ${\mathbf Y}_\C : \C\to \hat \C:X\leadsto \mathbf h_X$ in the sense of Proposition~\ref{prop:functoriality}. So we have the following corollary of that proposition:

\begin{corollary} Formation of $t\C$ and $m\C$ commutes with the Yoneda embedding. \end{corollary}

\begin{remark} \label{remark:mixtor}
	To the functor $\mix:\C\to m\C$ corresponds a functor $\mix^*$, given by
	\[  \mix^* : \widehat{m\C} \to \Cdual : \mathbf F \leadsto \mathbf F\circ \mix.\]
	We call $\mix^*\mathbf F$ the \emph{mixtor restriction} (\emph{mix}ed \emph{t}o \emph{or}dinary) of the presheaf $\mathbf F$. If $\mathbf F$ is represented by $\tilde X$, then $\mix^*\mathbf F$ is represented by $\comp_2\tilde X$.
	
	This implies the following statement, which is of fundamental importance
	for understanding mixed objects: \emph{if $\tilde X$ is a mixed object, but we are only willing to probe it ---in the sense of computing $\mathbf h_{\tilde X}(-)$--- on visible objects, then we can only observe its second component}.
\end{remark}

The following proposition answers a natural question, although it will play no role in what follows.
\begin{proposition} If $\C$ is small, there is an equivalence $t\hat\C \simeq \widehat{t\C}$. \end{proposition}

\begin{proof}
By Proposition~\ref{prop:functoriality} there is a functor $t{\mathbf Y}_\C : t\C\to t\Cdual : (X,\Phi)\leadsto (\mathbf h_X,\mathbf h_\Phi)$. We may construct a functor $\mathbf G:t\Cdual\to\widehat{t\C}$ such that its composition with $t{\mathbf Y}_\C$ is the Yoneda embedding for $t\C$:
\[ {\mathbf Y}_{t\C} : t\C \overset{t{\mathbf Y}_{\C}}{\longrightarrow} t\Cdual \overset{\mathbf G}{\longrightarrow} \widehat{t\C}.  \]
Indeed, we can define $\mathbf G: \widetilde{\mathbf X} = (\mathbf X,\mathbf\Phi) \leadsto \mathbf G(\widetilde{\mathbf X})$ by
\[
	\mathbf G(\mathbf X,\mathbf\Phi) : (Y,\Psi) \leadsto  \mathrm{eq} \Big(
\begin{tikzcd}
 \mathbf X(Y) \ar[r,shift left=.5ex,"{\mathbf\Phi}_Y"] \ar[r,shift right=.5ex,swap,"\mathbf X(\Psi)"] & \mathbf X(Y)
\end{tikzcd}
 \Big),
\]
where $\mathrm{eq}$ denotes the equalizer in the category of sets, i.e.,
\[ 	\mathbf G(\mathbf X,\mathbf\Phi)(Y,\Psi) = \{ u\in \mathbf X(Y) \mid {\mathbf\Phi}_Y(u)=\mathbf X(\Psi)(u) \}.  \]
In particular, if $\mathbf G(\mathbf X,\mathbf\Phi) = (\mathbf h_X,\mathbf h_\Phi)$ for some $(X, \Phi) \in \ob(t\C)$, we get
\begin{align*}
    \mathbf G(\mathbf h_X,\mathbf h_\Phi)(Y,\Psi)
    &= \{ u\in \mathbf h_X(Y) \mid ({\mathbf h_\Phi})_Y(u)=\mathbf h_X(\Psi)(u) \} \\
    &= \{ u\in \mathbf h_X(Y) \mid \Phi\circ u = u\circ\Psi \} \\
    &= \mathbf h_{(X, \Phi)}(Y, \Psi) ,
\end{align*}
from which it follows that $\mathbf G\circ t{\mathbf Y}_{\C}={\mathbf Y}_{t\C}$. On the other hand $t\Cdual$ is a cocomplete category by Proposition~\ref{prop:limits}, so by the universal property of $\widehat{t\C}$ as the free cocompletion of $t\C$, the functor $t{\mathbf Y}_\C$ extends uniquely to a cocontinuous functor $\mathbf F:\widehat{t\C}\to t\Cdual$ (\cite[p.\ 43\ Cor.\ 4]{topostheory}). The pair $\mathbf F,\mathbf G$ provides the equivalence of categories.
\end{proof}

\subsection{Mixed objects over a visible base object} \label{sec:relfac}
In this section, we discuss a first way how mixed objects can appear in the ordinary world. The simplest occasion is when one encounters an \emph{absolute factorization}, i.e.~a factorization of $F_X:X\to X$ through another object $Y$ such that the other composition is $F_Y$.

However, absolute factorizations rarely occur in practice, because often there is a base object $S$ and one prefers to work \emph{relatively} with respect to $S$. In such situations the entire reasoning takes place in the category $\C_{/S}$ where the only admitted arrows between two $S$-objects are the $S$-linear ones; but if $X$ is an $S$-object, then $F_X$ is not expected to be $S$-linear unless $F_S=\mathrm{id}_S$.  The role of the absolute factorizations is then played by \emph{relative} factorizations which proceed to introduce. The upshot of it all will be that mixed objects over visible base objects $\mix(S) = (S,S,F_S,\mathrm{id}_S)$ are still easily described in terms of linear arrows.\medskip

First, recall that for a morphism $f: T\to S$ in $\C$ and objects $X\in \C_{/S}$ and $Y\in \C_{/T}$, we have a natural identification between the following sets:
\begin{align*}
	\hom_{f}(Y,X) &= \{ g \in \hom_{\C}(Y,X) \mid q_Y\circ g = f\circ q_X \} \\
	& \simeq \{ h \in \hom_{\C}(Y,X \underset{q_X,f}{\times} T) \mid q_Y = p_2\circ h  \}  \\
	& \simeq \hom_{\C_{/T}}(Y,f^*X),
\end{align*}
where $f^*X = X\times_{q_X,f}T$ is called the pullback of $X$ from $\C_{/S}$ to $\C_{/T}$, with the projection to $T$ as is structural morphism.

In particular, for an object $X\in\ob\,\C$ with a structural morphism $q_X:X\to S$, the arrow $F_X:X\to X$ has a relative version in $\C_{/S}$, for which we introduce the notation 
\[F_{X/S}:X\to \flift X,\]
where we have also denoted $\flift X = F_S^* X$.\footnote{Our main excuse for introducing a new notation for $F_S^*$ is that we want it to be mentally processed in one step and not via the chain of thoughts $S\to F_S\to F_S^*$.} This arrow is completely determined by the properties that it is $S$-linear and that its composition with the canonical projection $P_{X/S} : \flift X\to X$ is equal to $F_X$; we will sometimes write this as $F_{X/S}=\mathrm{id}_S \times F_X$. Let us define formally:
\begin{definition} \label{def:relfac} A \emph{relative} $S$-factorization in $\C$ is a is a diagram of $S$-morphisms
	\[ X \overset{\pi}{\longrightarrow} \overline X \overset{\pi'}{\longrightarrow} \flift X , \]
	such that $\pi'\circ\pi = F_{X/S}$ and $\flift\pi\circ\pi'=F_{\overline X/S}$, where $\flift \pi:\flift X\to\flift {\overline X}$ is the base change of $\pi$ along $F_S$.
\end{definition}
The following observation is helpful to construct examples:
\begin{lemma} \label{lem:relfac} An $S$-morphism $X \overset{\pi}{\longrightarrow} \overline X \overset{\pi'}{\longrightarrow} \flift X $ with  $\pi'\circ\pi = F_{X/S}$ and $\pi$ epic, is a relative factorization.
\end{lemma}
\begin{proof}
	Since $\pi$ is epic, it suffices to show that
	\[ \flift \pi \circ \pi'\circ \pi = F_{\overline X/S}\circ \pi.  \]
	However the left hand side is equal to $\flift \pi\circ F_X$ and the resulting equality is an immediate consequence of $F_{\overline X}\circ\pi = \pi\circ F_X$.
\end{proof}

\begin{proposition} \label{prop:relfac}
	Mixed objects over a visible base correspond to relative factorizations.
\end{proposition}
\begin{proof} More specifically, for $S\in\ob(\C)$, we will show that $S$-factorizations correspond to mixed $\mix(S)$-objects in $m\C$.
	
	Let us start from an $S$-factorization $\pi : X\to \bar X$. If we glue a pullback square for $\flift X$ to the diagram defining the relative factorization, we obtain:
	\begin{center}
		\begin{tikzcd}
			X\arrow{r}{\pi}\arrow{d}{q_X} & \overline X \arrow{r}{\pi'}\arrow{d}{q_{\overline X}} & \flift X\arrow{r}{p_1}\arrow{d}{p_2} & X \arrow{d}{q_X} \\
			S\arrow{r}{\mathrm{id}_S} & S\arrow{r}{\mathrm{id}_S} & S\arrow{r}{F_S} & S
		\end{tikzcd}
	\end{center}
	Now $(\overline X,X,p_1\circ\pi',\pi)$ will be our $\mix(S)$-object; all we must do is compute the compositions:
	\begin{align*}
		p_1\circ\pi'\circ\pi &= p_1 \circ F_{X/S} = F_X \\
		\pi\circ p_1\circ\pi' &= p_1'\circ \flift\pi \circ \pi' = p_1' \circ F_{\overline X/S} = F_{\overline X},
	\end{align*}
	where $p_1' : \flift \overline{X}\to \overline X$ is the canonical projection.
	
	Conversely, it is clear that starting from a mixed $\mix(S)$-object $(\overline X, X,\phi,\pi)$, one can linearize $\phi$, i.e. factor it uniquely through $\flift X$ into $\phi=p_1\circ\pi'$ and verify immediately that $\pi'\circ\pi$ and $\flift \pi\circ\pi'$ satisfy the defining properties of $F_{X/S}$ and $F_{\overline X/S}$, namely that they are $S$-linear and that composition with the projections gives $F_X$ and $F_{\bar X}$.
\end{proof}

By combining a relative factorization with an absolute factorization of the base, we obtain a large number of mixed objects over invisible base objects. 

\begin{corollary} \label{cor:relfac} The following data in $\C$ determines an $\tilde S$-object $\tilde X$: an object $S$, a relative $S$-factorization $X \overset{\pi}{\longrightarrow} \overline X \overset{\pi'}{\longrightarrow} \flift X$, and morphisms  $S\overset\alpha\longrightarrow S'\overset\beta\longrightarrow S$ composing to $F_S$ resp.~$F_{S'}$.
\end{corollary}
\begin{proof} Of course $\tilde S=(S,S',\alpha,\beta)$ is a mixed object and there is a morphism $\rho : \tilde S\to \mix(\comp_1(\tilde S))=\mix(S)$ explicitly determined by $\rho : (\mathrm{id}_S,\beta)$ --- this is the unit of the adjunction $\comp_1\dashv \mix$. On the other hand, by Proposition \ref{prop:relfac}, the relative factorization corresponds to a morphism $\tilde Y\to \mix(S)$. So the pullback $\rho^*(\tilde Y) = \tilde Y \underset{\mix(S)}{\times} \tilde S \to \tilde S$ determines an $\tilde S$-object with components $(\overline X,\beta^* X)$.
\end{proof}

\begin{remark} \label{rem:less-accessible}
	\begin{itemize}
		\item One can turn the collection of relative factorizations into a category and interpret Proposition~\ref{prop:relfac} as an equivalence of categories.
		\item If $\tilde S = \mix(S)$ is visible, every mixed $\tilde S$-object arises from Proposition~\ref{prop:relfac}, but it is not the case that for arbitrary $\tilde S=(S,S')$ every $\tilde S$-object can be constructed as in Corollary~\ref{cor:relfac}! This is clear because in general one does not expect the functor of base change $\tilde X\to \tilde X \times_{\mix(S)}\tilde S$ to be essentially surjective. In some sense the objects that do not arise this way are even less accessible. 
		\item Starting from an $\tilde S$-object $\tilde X$, one may base change through the morphism $\mix(\comp_2(S))=(S',S',\mathrm{id}_{S'},F_{S'})\to \tilde S$ ---  the counit of $\mix\dashv\comp_2$ --- to obtain a $\mix(S')$-object and thus a relative $S'$-factorization. So it is true that every mixed $\tilde S$-object is a form of a relative factorization.  In particular, starting from an $S$-factorization one obtains an $S'$-factorization by base changing twice, which comes down to base changing along $\beta: S'\to S$.
		\item We suspect that the reason why the \emph{mixed quadrangles of type $\mathsf F_4$} are so peculiar is because they are such inaccessible objects which do not arise from a base change. See \S\ref{sec:his-1990} for some additional discussion.
	\end{itemize}
\end{remark}

\subsection{The concept of a fairy} \label{sec:bookkeeping}
This section, which is independent of the previous sections, consists of preparation for the next section, and in particular Proposition~\ref{prop:weilrestriction}, which is itself the crucial ingredient for our Theorem~\ref{theorem:pseudo-reductive-groups}. The main proposition in this section is Proposition~\ref{prop:extension-adjunction} but we need to make a few definitions before we can even state it. 

Let us first recall some basic facts about an arbitrary functor $u:\C\to\D$. The \emph{Yoneda extension} of $u$ is a functor 
\[  \mathbf u_! : \widehat{\C}\to\widehat{\D}  : \mathbf h_X \leadsto \mathbf h_{u(X)},  \]
defined here on representable objects and extended to arbitrary presheaves by taking limits, as explained in \cite[Cor.~4 p.~44]{topostheory}. Furthermore, composition with $u$ defines a functor in the opposite direction
\[ \mathbf u^* :   \widehat{\D}\to \widehat{\C} : \mathbf F\leadsto \mathbf F\circ u  \]
which is a right adjoint: $\mathbf u_!\dashv \mathbf u^*$. Since $\mathbf u_!$ extends $u$, one expects that $\mathbf u^*$ extends a right adjoint to $u$, whenever it exists.

To make this last claim more precise, we must consider a representable presheaf $\mathbf F\simeq\mathbf h_X$. Then representability of  $\mathbf F\circ u$,  implies an isomorphism $\mathbf F\circ u\simeq\mathbf h_U$ of functors, for a certain $U\in\ob(\C)$. In other words, we obtain a collection of bijections
\[ \hom_{\D}(uV,X) = \mathbf h_X(uV)\simeq \mathbf h_U(V) = \hom_{\C}(V,U),  \]
natural in $V$ and in $X$ whenever $\mathbf F\circ \mathbf h_X$ is representable. Therefore the assignment $v: \D\to\C:X\leadsto U$ is a partially defined functor, which is right adjoint to $u$. 

 Constructing $v$ explicitly in this way requires the global choice of a representing object each time $\mathbf F\circ\mathbf h_X$ is representable, this is a technical difficulty that we will pass over quickly by stating that all such functors are naturally isomorphic.

Now we consider a fixed arrow $f:T\to S$ in our category $\C$ and consider specifically functors between the corresponding slice categories $\C_{/T}$ and $\C_{/S}$---the arrows in these categories are said to be $T$-linear and $S$-linear.\medskip

There is an adjoint pair of functors
\[ f_! \dashv f^*  \::\:  \begin{tikzcd} \C_{/T} \rar[shift left=0.65ex,"f_!"]   &  \C_{/S} \lar[shift left=0.6ex,"f^*"]   \end{tikzcd}. \]
The functors $f_!$ and $f^*$ are defined on objects and their structure morphisms by
\begin{align*}
f_! &: \C_{/T}\to\C_{/S} : (X,q_X) \leadsto (X,f\circ q_X) \\
f^* &: \C_{/S}\to\C_{/T} : (X,q_X) \leadsto (X \times_S T,p_2),
\end{align*}
and the fact that these form an adjoint pair follows immediately from the universal property of a pullback as fibered product. The question whether $f^*$ admits a right adjoint $f_*$ is of particular interest --- this is closely related to the existence of an internal Hom-functor. The right adjoint can be found through the formalism introduced above, with $f^*$ playing the role of $u$. In general, $f_*$ will only be partially defined. If for some $X\in\ob(\C_{/T})$ the corresponding object $f_*X\in\ob(\C_{/S})$ is defined, there are bijections
\[  \hom_{\C_{/T}}(f^* Y,X) \to \hom_{\C_{/S}}(Y,f_*X) : f \mapsto f^\flat,   \]
natural in $Y$ and $X$, whenever $f_*$ is defined at $X$. The inverse map of $\flat$ will be denoted by $\sharp$. All this information is implied by writing
\[ f_! \dashv f^* \dashv f_* \::\:  \begin{tikzcd} \C_{/T} \rar[shift left=1.25ex,"f_!"] \rar[shift left=-1.25ex,dashed,swap,"f_*"]  &  \C_{/S} \lar[shift left=0ex,"f^*" description]   \end{tikzcd}. \]

We generalize this to our context of a category $\C$ with endomorphism $F$ of the identity functor. In such a situation, we are often confronted with a diagram as depicted here on the left and want to obtain a new diagram, as depicted on the right.
\[
	\begin{tikzcd}
		f^*A \ar[r,"u"]\ar[d] & B\ar[d] \\
		T  \ar[r,"F_{T}"] & T
	\end{tikzcd}
	\quad \overset{!?}\iff\quad
	\begin{tikzcd}
			A \ar[d]\ar[r,"u^\flat"] & f_*B\ar[d] \\
			S  \ar[r,"F_{S}"] & S
		\end{tikzcd}
\]

The difficulty is that it is not \emph{a priori} possible to consider $u^\flat$, because $u$ is not $T$-linear. In other words, $u$ is simply not in the domain of a $\flat$-map. So we seek to extend the calculus of the adjunction $f^*\dashv f_*$ to $F_T$-linear maps, and more generally $F_T^n$-linear maps, for any natural number $n$.\medskip

Our first step is to formulate this problem. This requires us to define the \emph{fairies} $\C_{/S}^{(F)}$ as follows. (The word fairy is short for $F$-ary category; one could also call them more verbosely \emph{semi-linear slice categories}.) 

\begin{definition} The objects of the fairy $\C_{/S}^{(F)}$ are the arrows $q_X : X\to S$. A morphism between $X\to S$ and $Y\to S$ is a pair $(f,n)$, where $f:X\to Y$ and $n$ is a natural number, such that $F_S^n\circ q_X = q_Y\circ f$. 
\end{definition}

\begin{remark} \label{remark:semi-linear} The following remarks are all important for working with fairies. We leave the easy proofs to the reader, insofar required.
	\begin{enumerate} 
		\item  Let $\mathscr N$ be a category with a single object $\bullet$ such that $\End_{\mathscr N}(\bullet)=(\mathbb N,+)$. Then we have a diagram of categories
		\[  \C_{/S} \overset{\mathrm{inc}}{\into} \C_{/S}^{(F)} \overset{\mathrm{pr}}{\onto} \mathscr N, \]
		where the first functor sends an object to itself and an arrow $u$ to $(u,0)$ and the latter functor sends every object to $\bullet$ and an arrow $(u,n)$ to $n$. We will always see $\C_{/S}$ as the \emph{wide subcategory} (i.e. a subcategory containing all the objects) of \emph{linear morphisms} in the fairy $\C_{/S}^{(F)}$ through the functor $\mathrm{inc}$. We also denote the inclusion simply by  $\mathrm{inc}:X\mapsto \overline X$ and $u\mapsto \overline u$. We will use the notation \begin{tikzcd} \strut \rar[nonlinear] & \strut \end{tikzcd} to warn the reader that an arrow is possibly non-linear when drawing fairy diagrams.
		\item There is a functor $\mathsf G : \mathscr N\to\C$ defined by $(\bullet\overset n\to\bullet)\mapsto (S\overset{F_S^n}\to S)$ and with the help of this functor one may define $\C_{/S}^{(F)}$ as either the fibered product $\vec{\C}\times_{p,\C,\mathsf G}\mathscr N$ of categories, where $\vec\C$ is the category of arrows in $\C$ and $p:\vec\C\to\C:(X\to Y)\leadsto Y$ the codomain fibration, or as the comma category $\mathrm{id}_\C\downarrow \mathsf G$. These constructions provide a natural variation on the theme of a slice category.
		\item \label{bullet:functor-extends} We will say a functor $u:\C_{/S}\to \D_{/T}$ between slice categories \emph{extends semi-linearly} or simply \emph{extends} if there is a functor $\mathsf u : \C_{/S}^{(F)}\to \D_{/T}^{(F)}$ such that $\mathsf u\circ\mathrm{inc} =\mathrm{inc}\circ u$. 
		\item The Yoneda extension of the inclusion $\mathrm{inc}$ is a functor $\mathbf{inc}_!$ which fits into the following diagram with the Yoneda embeddings:
		\[ \begin{tikzcd} \C_{/S} \dar["\mathbf y"] \rar["\mathrm{inc}"] & \C_{/S}^{(F)}  \dar["\mathbf y"] \\ \widehat{\C_{/S}} \rar["\mathbf{inc}_!"]  & \widehat{\C_{/S}^{(F)}} \end{tikzcd}  \]
		If we also denote $\mathbf{inc}_!(\mathbf F) =\overline{\mathbf F}$, commutativity of the diagram can be written as $\overline{\mathbf h_X} = \mathbf h_{\overline X}$.
		\item The category $\C_{/S}^{(F)}$ is itself a category with an endomorphism of the identity functor, also denoted by $F$ or by $F_{/S}$ if confusion is possible, with component at $X$ given by $F_{\bar X}=(F_{X},1)$.
		\item \label{bullet:decomp} The universal property of the pullback $Y\leadsto Y\times_{q_Y,F_S^n}S = (F_S^n)^*Y$ implies that every arrow $(f,n+m):X\to Y$ in $\C_{/S}^{(F)}$ factors into a morphism denoted $\langle f\mid m\rangle$ followed by a projection $(p_1,n)$:	
			\[ \begin{tikzcd} X \rar["\langle f\mid m\rangle"] & Y\times_{q_Y,F_S^n} S \rar[nonlinear,"{(p_1,n)}"] & Y 
			\end{tikzcd} \]
		Thanks to natural isomorphism $(F_S^n)^*\simeq (F_S^*)^n$ we may and will in practice identify $\langle\cdots\langle\langle f\mid 1\rangle\mid 1\rangle\cdots\mid 1\rangle$ with $\langle f\mid n\rangle$ and also $\langle f\mid 0\rangle$ with $f$. This is mainly of importance when $m=0$ and we have factored $(f,n)$ into a linear morphism $\langle f\mid n\rangle$ followed by a projection.
		\item \label{bullet:diag} 	For instance $F_{\overline X}$ factors via $\langle F_{\overline X}\mid 1\rangle$ which we identify with $F_{X/S}$ (see \S\ref{sec:relfac}) through the inclusion $\mathrm{inc}$. Therefore, recalling our notation $\flift X=F_S^*X = X\times_{q_X,F_S}S$, we obtain a canonical factorization 
			\[ \begin{tikzcd} X \rar["F_{X/S}"] & \flift X \rar[nonlinear,"P_{X/S}"] & X. 
			\end{tikzcd} \]
		We may organise this information as follows: there is a functor $\flift$, sometimes denoted $\flift_{S}$ for clarity, together with natural transformations $F_{/S}$ and $P_{/S}$ between $\flift$ and the identity functor as follows:
		\[  \begin{tikzcd}[column sep=huge] \C_{/S}^{(F)} \rar[bend left=45,"\mathrm{id}"{name=U}] \rar[bend right=45,swap,"\flift"{name=V}]& \C_{/S}^{(F)}. \ar[Rightarrow,from=U,to=V,bend left=10,"F_{/S}",shorten >=10pt,shorten <=10pt] \ar[Rightarrow,from=V,to=U,bend left=10,"P_{/S}",shorten >=10pt,shorten <=10pt] \end{tikzcd}   \]
	\end{enumerate}
\end{remark} 

Our goal for the rest of the paragraph is to prove the following proposition. 
\begin{proposition} \label{prop:extension-adjunction} The adjunction $f_! \dashv f^* \dashv f_*$ extends (cfr.~\ref{remark:semi-linear}.\ref{bullet:functor-extends}) to the  fairies
	\[ \mathsf f_! \dashv \mathsf f^* \dashv \mathsf f_* \::\:  \begin{tikzcd} \C_{/T}^{(F)} \rar[shift left=1.25ex,"\mathsf f_!"] \rar[shift left=-1.25ex,dashed,swap,"\mathsf f_*"]  &  \C_{/S}^{(F)} \lar[shift left=0ex,"\mathsf f^*" description]   \end{tikzcd}. \]
\end{proposition}

Proving this proposition requires us to extend these functors to all arrows of the fairies $\C_{/.}^{(F)}$ and proving that the resulting functors are still adjoint pairs. This holds no serious difficulty for $f_!$ and $f^*$, as we will see immediately in Proposition~\ref{prop:extend-f!-f*}. For $f_*$ however, we know no direct way to extend the domain of definition to the non-linear morphisms and this causes most of the technical difficulties in this section. So we will use an indirect approach where introduce the notion of a bewitched functor and study their right adjoints; the proof then follows by observing that $f^*$ is indeed bewitched.

\begin{proof}[Proof of \ref{prop:extension-adjunction}] The first part will be shown in Proposition~\ref{prop:extend-f!-f*}; for the second part apply Proposition~\ref{prop:beck-chevalley} to Remark~\ref{remark:bewitched}.\ref{bullet:pullback-is-bewitched}.
\end{proof}

\begin{proposition} \label{prop:extend-f!-f*} The adjunction $f_!\dashv f^*:\C_{/T}\to\C_{/S}$ extends semi-linearly.
\end{proposition}
\begin{proof} Of course we must define $\mathsf f_!(X)=f_!(X)$ on objects. (And in fact, $\mathsf f_!(u,0)=(f_!u,0)$.) An arrow $(g,n):X\to Y$ induces a commutative diagram in $\C$
	\[ 
	{\begin{tikzcd}
		X \dar["q_X"]\rar["g"] & Y\dar["q_Y"] \\
		T \dar["f"]\rar["F_S^n"] & T\dar["f"] \\
		S \rar["F_T^n"] & S,
		\end{tikzcd} }
	\] 
	and therefore an arrow $\mathsf f_!(g,n) : \mathsf f_!X\to \mathsf f_!Y$. We leave to the reader the straightforward verification that this defines a functor.\medskip
	
	Defining $\mathsf f^*$ on objects --- and linear arrows --- is trivial, so let $(g,n): U\to V$ be an arrow in $\C_{/S}^{(F)}$. The following diagram commutes, since every square commutes:
	\[
	\begin{tikzcd}
	& U\times_S T \dlar[swap,"p_2"] \drar["p_1"] & \\[-2ex] T\ar[dd,"F_T^n"] \drar["q_T"] & & U \ar[dd,"g"] \dlar[swap,"q_U"] \\[-2ex] & S \ar[dd,"F_S^n"] & \\[-2ex] T \drar[swap,"q_T"] & & V \dlar["q_V"] \\[-2ex] & S. & 
	\end{tikzcd}
	\]
	It we erase the interior and replace it with a pullback square for $V\times_S T$, we  get that the following diagram (without the dashed arrow) commutes in $\C$.
	\[
	\begin{tikzcd}
	U\times_S T \ar[rr,"p_1"]\ar[dd,"p_2"]\ar[dr,dashed] & & U \dar["g"] \\
	& V\times_S T \rar["p_1"]\dar["p_2"] & V \dar["q_V"] \\
	T\rar["F_T^n"] & T\rar["q_T"] & S
	\end{tikzcd}
	\]
	This implies the dashed arrow is uniquely defined by the pullback $V\times_S T=\mathsf f^*V$ and this is $\mathsf f^*g:\mathsf f^*U\to \mathsf f^*V$.  We leave to the reader the straightforward verification that this defines a functor. Also the fact that the extended functors define an adjoint pair $\mathsf f_!\dashv \mathsf f^*$ on the fairies is easy to verify and left to the reader. (Note that the unit and counit transformations are already determined by $f_!$ and $f^*$.)
\end{proof}

Now comes the hard part: extending $f_*$ on its domain. It is easy enough to define a functor $\mathsf f_*$ formally as the partially defined right adjoint of $\mathsf f^*$ but what is not obvious is that the functor $\mathsf f^*$ extends $f^*$, i.e.~that both functors agree on objects and linear arrows.

\begin{definition} \label{def:bewitched} A functor  $\alpha: \C_{/S}^{(F)}\to\D_{/T}^{(G)}$ between fairies is \emph{bewitched} if $\alpha$,  $G_{\alpha(X)}=\alpha(F_X)$ for every $X\in \ob(\C_{/S})$ and it preserves the decomposition from Remark~\ref{remark:semi-linear}.\ref{bullet:decomp}.
\end{definition}

\begin{remark} \label{remark:bewitched}
	\begin{enumerate}
		\item 	To avoid any confusion, we provide some details concerning the Definition~\ref{def:bewitched}: it says that there is a natural isomorphism $\flift_T\circ\alpha\simeq\alpha\circ\flift_S$ (with linear components) such that for every arrow $(u,n):X\to Y$ in $\C_{/S}^{(F)}$ with its decomposition into $\langle u\mid n\rangle:X\to \flift^n Y$ and $P_{Y/S}^n:\flift^n Y \to Y$ we have $\alpha \langle u\mid n\rangle =\langle \alpha(u)\mid n\rangle$ and $\alpha(P_{Y/S}^n)=P_{\alpha(Y)/T}^n$, up to the identification $\alpha\flift^n Y\cong \flift^n \alpha Y$, as in the following diagram: 
		\[ 
		\begin{tikzcd}[column sep=large,row sep=tiny]   & \alpha(\flift_S^n Y) \ar[dr,nonlinear,"\alpha(P_{Y/S}^n)"]  & \\ \alpha(X) \ar[ur,"\alpha \langle u\mid n\rangle "] \ar[dr,swap,"{ \langle \alpha(u)\mid n\rangle}"]& & \alpha(Y)  \\ & \flift_T^n(\alpha Y) \ar[ur,swap,nonlinear,"P_{\alpha(Y)/T}^n"] \ar[uu,"\rotatebox{90}{\(\sim\)}"] &		\end{tikzcd}
		\]
		\item Let us take in particular $n=0$. Since the  lower path in the above diagram is the (unique) decomposition in a linear arrow and a projection, and since the vertical identification is linear, we have that $\alpha\langle u\mid n\rangle=\alpha(u)$ is linear. Therefore, a bewitched functor sends linear arrows to linear arrows and thus restricts to a functor $\alpha^\circ:\C_{/S}\to\D_{/T}$. 
		\[ \begin{tikzcd}
		\C_{/S} \rar["\alpha^\circ"] \dar[swap,"\mathrm{inc}"] & \D_{/T} \dar["\mathrm{inc}"] \\
		\C_{/S}^{(F)} \rar[swap,"\alpha"] & \D_{/T}^{(F)}  \end{tikzcd}
		\]
		In Proposition~\ref{prop:beck-chevalley} will denote a bewitched functor by $\bar\alpha$ and its restriction by $\alpha$ so that this diagram reads on objects: $\bar\alpha(\bar X)=\overline{\alpha(X)}$.
		\item Another immediate consequence of the definition is that a bewitched functor preserves the entire diagram that we drew in Remark~\ref{remark:semi-linear}.\ref{bullet:diag}, since this diagram encodes the decomposition of $F = P_{/S}\circ F_{/S}$.
		\item \label{bullet:pullback-is-bewitched} The functor $\mathsf f^*$ as defined in Proposition~\ref{prop:extend-f!-f*} is bewitched, this is easy to verify with the natural isomorphism
		\[ F_T^* \circ f^* \simeq (f\circ F_T)^* = (F_S\circ f)^* \simeq f^*\circ F_S^*.
		\]
	\end{enumerate}
\end{remark}

\begin{proposition} Let $\alpha : \C_{/S}^{(F)}\to \D_{/T}^{(G)}$ be a bewitched functor together with its (perhaps partially defined) left and right adjoints $\gamma\dashv \alpha\dashv \beta$. If $X\in\ob(\D_{/T}^{(F)})$, then $\beta G_X = F_{\beta(X)}$ and $\gamma G_X = F_{\gamma(X)}$.
\end{proposition}
\begin{proof} Let us show this for $\beta$, the proof for $\gamma$ is similar. Consider an arbitrary $X\in \ob(\D_{/T}^{(G)})$. Since $\alpha$ is bewitched, we have that $\alpha F_{\beta X}=G_{\alpha\beta X}$. Therefore there is a diagram

\[ 
\begin{tikzcd} \alpha\beta X \rar["\eta_X"] \dar["\alpha F_{\beta X}"] &  X \dar["G_{X}"] \\ \alpha\beta X \rar["\eta_X"] & X, 
\end{tikzcd}
\]
where the horizontal arrows are units of the adjunction $\alpha\dashv\beta$. We can use the adjunction to push $\alpha$ to the right, taking into account that $(\eta_X)^\flat = \mathrm{id}_{\beta X}$, we get that $\beta G_X=F_{\beta X}$. 
\end{proof}

\begin{proposition} \label{prop:beck-chevalley} Consider a bewitched functor $\bar\alpha: \C_{/S}^{(F)}\to \C_{/T}^{(F)}$ with restriction $\alpha$. Consider the partial right adjoints $\alpha\dashv\beta$ and $\bar\alpha\dashv\bar\beta$. Then $\bar\beta$ extends $\beta$.
\end{proposition}

\begin{proof} Consider left diagram below, which is a commuting diagram in the category of categories. Note that all functors there are part of an adjunction such as $\boldsymbol\alpha_! \dashv \boldsymbol\alpha^*$. This implies there is a Beck--Chevalley transformation $\mathbf{inc}_!\circ\boldsymbol\alpha^* \implies \bar{\boldsymbol\alpha}^*\circ\mathbf{inc}_!$ as depicted in the right diagram. We will show that this transformation is a natural isomorphism, one says that the \emph{Beck--Chevalley} condition holds.
			\[  \begin{tikzcd}
			\widehat{\C_{/S}} \rar["{\boldsymbol\alpha_!}"] \dar[swap,"\mathbf{inc}_!"] & \widehat{\C_{/T}} \dar["\mathbf{inc}_!"] \\
			\widehat{\C_{/S}^{(F)}} \rar[swap,"{\bar{\boldsymbol\alpha}_!}"] & \widehat{\C_{/T}^{(F)}}  \end{tikzcd} \quad
 \begin{tikzcd} \widehat{\C_{/S}}  \dar[swap,"\mathbf{inc}_!"] \ar[dr,Rightarrow] & \widehat{\C_{/T}} \lar[swap,"\boldsymbol\alpha^*"]\dar["\mathbf{inc}_!"] \\
	 	 \widehat{\C_{/S}^{(F)}}  & \widehat{\C_{/T}^{(F)}} \lar["\bar{\boldsymbol\alpha}^*"] \end{tikzcd}  
	 	\]
	Let us first explain how this implies the statement. If $\beta$ is then defined at $X$ then we have $\mathbf h_X\circ\alpha\simeq \mathbf h_{\beta(X)}$; applying the inclusion we get
	\[\overline{\mathbf h_X\circ\alpha} \simeq \overline{\mathbf h_{\beta(X)}} \simeq \mathbf h_{\overline{\beta(X)}}.\]  The Beck--Chevalley isomorphism says that
	\[ \overline{\mathbf h_X\circ\alpha} = \mathbf{inc}_!(\boldsymbol\alpha^*(\mathbf h_X)) \simeq \boldsymbol\alpha^*(\mathbf{inc}_!(\mathbf h_X)) =  \overline{\mathbf h_{X}}\circ\overline\alpha\simeq {\mathbf h}_{\overline X}\circ\overline\alpha. \]	 	 
	Therefore ${\mathbf h}_{\overline X}\circ\bar\alpha \simeq {\mathbf h}_{\overline{\beta(X)}}$. Therefore $\overline{\beta}$ is defined at $\overline X$ and in fact represented by $\overline{\beta(X)}$.
	 	 
We will now apply the criterium \cite[{}1.18]{guitart2014} --- proven in \cite{guitart1980} --- to verify that that the following diagram is an \emph{exact square}. This implies the Beck-Chevalley condition as in \cite[{}1.15]{guitart2014}. There is also an exposition of this material on the nlab \cite{nlab2}.

	\[ \begin{tikzcd}
	\C_{/S} \rar["\alpha"] \dar[swap,"\mathrm{inc}"] & \C_{/T} \dar["\mathrm{inc}"] \\
	\C_{/S}^{(F)} \rar[swap,"\bar\alpha"] & \C_{/T}^{(F)}  \end{tikzcd}
	\]

To apply the criterium, we need to consider arbitrary objects $Y$ in $\C_{/T}$ and $\overline Z$ in $\C_{/S}^{(F)}$ and a morphism $(u,n) : \overline Y \to \alpha(\bar Z)=\bar\alpha(Z)$ in $\C_{/T}^{(F)}$. First we need to show that there exists  a triple $(X,y,(z,m))$ where $X$ is an object in $\C_{/S}$ and
\begin{align*}
	y &: Y \to \alpha(X) \\
	(z,m) &: \bar X \to \bar Z
\end{align*}
are morphisms such that the composition in the fairy $\C_{/T}^{(F)}$
\[    \begin{tikzcd}
  \bar Y \rar["\bar y"] & \overline{\alpha(X)}  \rar[nonlinear,"{\bar\alpha(z,m)}"] & \bar{ \alpha (Z)} 
\end{tikzcd} \]
is equal to $(u,n)$. To see this, let us choose $X=\flift^n Z$, and recall that we have an identification $\bar{\alpha(X)}\cong \flift^n\bar\alpha \bar Z$. Thus there is a decomposition of $(u,n)$ into a linear morphism $\langle u \mid n\rangle : Y\to\alpha(Z)$ and a projection $\flift^n\bar\alpha \bar Z\to\bar \alpha Z$ which, since the functor is bewitched, is also $\bar\alpha$ of the projection $P^n_{Z/T}: \flift^n \bar Z \to \bar Z$. Thus the choice $y=\langle u\mid n\rangle$ and $(z,n)=(P_{Z/T}^n,n)$ provides the sought decomposition.

Next, we need to assume that $(X',y',(z',m))$ is another solution to the same problem and show that $X$ and $X'$ are connected in the category $\C_{/S}$ through a zigzag in such a way that the induced diagram commutes. (This is the so called \emph{lantern diagram}.) More precisely, we will show that there is a morphism $v:X'\to X$ in $\C_{/S}$ such that the following diagrams commute:
\[ 
\begin{tikzcd}
& Y \ar[dl,swap,"y"] \ar[dr,"y'"] & \\   \alpha(X') \ar[rr,"\alpha(v)"] & & \alpha(X) \\[0pt] \bar{X}' \ar[rr,"v"] \ar[dr,nonlinear,swap,"{(z',n)}"]  & & \bar{X} \ar[dl,nonlinear,"{(z,n)}"] \\ & \bar Z & 
\end{tikzcd}
\]
Obviously $m=n$ and it is clear that linearizing the morphism $(z',n) : \overline X'\to \overline Z$ via $\bar X = \flift^n \overline Z$, we obtain a \emph{linear} morphism $v: X'\to X$ which fits into the lower diagram. Applying $\alpha$ to this diagram we obtain the following commuting diagram, where the dashed arrow is defined by composition.

\[ \begin{tikzcd}[row sep=tiny] & \alpha(X) \ar[dr,nonlinear,"{\bar\alpha(z,n)}"] & \\ Y\ar[dr,swap,"y'"] \ar[ur,dashed]  & & \alpha(Z) \\ &  \alpha(X') \ar[uu,"\alpha(v)" description] \ar[ur,nonlinear,swap,"{\bar\alpha(z',n)}"] & \end{tikzcd}  \]
Up to the canonical identification $\alpha(X)\cong \flift^n(\alpha(Z))$, the dashed arrow is a linear arrow with the property that composing it with the projection $ \flift^n(\alpha(Z)) \to \alpha(Z)$ yields $u$. But this property defines $y = \langle u \mid  n\rangle$. 
\end{proof}

\subsection{Restriction to visible objects} \label{sec:restriction}
In \S\ref{sec:bookkeeping} we studied the functor of base change $f^*:\C_{/S}\to\C_{/T}$ and its partially defined right adjoint $f_*:\C_{/T}\to\C_{/S}$ and extended this formalism to semi-linear morphisms as provided by the fairies. We will now apply these ideas to study a right adjoint to base change in the mixed category $\mix\C$. Before we do so, let us motivate why we are interested in this situation. The \emph{real} motivation is of course that we would like to prove Theorem~\ref{theorem:pseudo-reductive-groups} eventually. Nonetheless, we provide some intrinsic motivation, based on the observation that this functor provides a way for invisible mixed objects to invade the visible world, giving rise to exotic phenomena. 

Recall once again that we consider the category $\C$ as a (full) subcategory of $m\C$ through the functor $\mix$; recall that these objects are called \emph{visible}. Also recall the adjunctions
\[  \comp_1\dashv \mix \dashv \comp_2 \::\: \begin{tikzcd}
m\C \ar[r,shift right=-1ex,"\comp_2"]\ar[r,shift right=1ex,swap,"\comp_1"] & \C \ar[l,swap,"\mix" description]\end{tikzcd}  \]
It is thanks to these adjunctions that the mixed object $\tilde X$ can be understood as some kind of mixture of its components $\comp_i(\tilde X)$.

Let us now focus on the relative situation, with respect to a base object $\tilde S$. If we assume that $\tilde S = \mix(S)$ is itself visible then by good behaviour of adjunctions with slice categories, as is explained nicely on the nLab \cite[(3.1)]{nlab}, we get new adjunctions for free:
\[  (\comp_1)_{/S}\dashv \mix_{/S} \dashv (\comp_2)_{/S} \::\: \begin{tikzcd}
m\C_{/S} \ar[r,shift right=-1ex,"\comp_2"]\ar[r,shift right=1ex,swap,"\comp_1"] & \C_{/S} \ar[l,swap,"\mix" description]\end{tikzcd}  \]
This too holds for us the interpretation that a mixed $S$-object is a mixture of two ordinary $S$-objects. This is not a surprise, since they are related to relative factorizations as we saw in \S\ref{sec:relfac}.

If $\tilde S=(S,S',\alpha,\beta)$ is not necessarily visible, the situation becomes more interesting. The \emph{morally correct} way of seeing an object $\tilde X=(X,X',\phi,\psi)$ as a mixture of two ordinary objects, would be to see it as a mixture of an $S$-object $X$ and $S'$-object $X'$ through the adjunctions between $m\C$ and $\C$. But often one \emph{insists} on working over a fixed base object $S$ and there, more adjunction hocus-pocus as in \cite[(3.1)]{nlab} can only give us:

\[  (\comp_1)_{/\tilde S}\dashv \mix_{/\tilde S} \dashv \:?? \::\: \begin{tikzcd}
m\C_{/\tilde S} \ar[r,shift right=-1ex,"??"]\ar[r,shift right=1ex,swap,"\comp_1"] & \C_{/S} \ar[l,swap,"\mix" description]\end{tikzcd}  \]

The question marks mean that there that defining a right adjoint to $\mix_{/\tilde S}$ is not in general possible. But as we explained in the introduction to \S\ref{sec:bookkeeping}, such a functor can still be partially defined at some objects. In  these cases, we obtain objects in the category $\C_{/S}$ which somehow look exotic. (Our Theorem~\ref{theorem:pseudo-reductive-groups} says that this is how the exotic pseudo-reductive groups arise.)

To better understand the right adjoint to $\mix_{/\tilde S}$, let us recall how $\mix_{/\tilde S}$ is defined in this case. Let us denote by $f : \tilde S\to S=\mix(\comp_1(\tilde S))$ the unit of the adjunction $\comp_1\dashv\mix$:
\begin{center}
	\begin{tikzcd}
		\tilde S \arrow{d}{f} & S \arrow{d}{\mathrm{id}_S} \arrow{r}{\alpha} & S'\arrow{r}{\beta} \arrow{d}{\beta} & S\arrow{d}{\mathrm{id}_S} \\
		S & S  \arrow{r}{F_S}& S\arrow{r}{\mathrm{id}_S} & S
	\end{tikzcd}
\end{center}

The functor $\mix_{/\tilde S}$ is constructed as the composition of $\mix_{/S}$ with the base change
\[ f^* : (m\C)_{/S} \to (m\C)_{/\tilde S} : T\leadsto T\times_S \tilde S.   \]

We already know that $\mix_{/S}$ admits a right adjoint $(\comp_2)_{/S}$, so the question becomes: \emph{what can we say about a right adjoint to $f^*$}? In particular, can we understand such a right adjoint in terms of a right adjoint $\beta_*$ for $\beta^*$ to  $f_*$? The following proposition says that this is indeed the case.

\begin{proposition} \label{prop:weilrestriction} Consider a mixed object $\tilde S$ together with its morphism $f:\tilde S\to S=\mix\comp_1(\tilde S)$. Let $\tilde X=(X,X',\phi,\psi)$ be an $\tilde S$-object and assume $\beta_*\beta^*X$ and $\beta_*X'$ exist. If we define
\[ f_*(\tilde X) = \big( X, X \underset{\beta_*\beta^*X}\times \beta_* X' ,\pi , p_1 \big),  \]
then for all $S$-objects $\tilde T$:
\[ \hom_{S}(\tilde T,f_*\tilde X) \simeq \hom_{\tilde S}(f^*\tilde T, \tilde X)   \]
\end{proposition}
The map $\pi$ and the maps defining $Y = X \times_{\beta_*\beta^*X}\beta_*X'$ are specified in the proof.

\begin{proof} In the following reasoning, we will be working with the adjunction of fairies 
	\[ \beta^* \dashv \beta_* : \begin{tikzcd} \C_{/S}^{(F)} \rar[shift left=0.5ex] & \C_{/S'}^{(F)}. \lar[shift left=0.5ex] \end{tikzcd} \]
Although we will still use the notation \begin{tikzcd} {}\rar[nonlinear] & {} \end{tikzcd} to warn the reader for non-linear arrows, we will denote the arrow $(u,n)$ simply by $u$.  The number $n$ is always $0$ or $1$ so it can be read off from the diagrams.

\textbf{Step 1: construction of $\pi$.} Let us start from the base change $\tilde X \times_{\tilde S} S'$. It is given by the $S'$-object
\[ (\beta^* X,X',\phi\circ p_1,\psi'),  \]
where $p_1\circ\psi' = \psi$ and $p_1:\beta^*X\to X$ denotes the canonical projection. So if we apply the adjunction $\beta^*\dashv \beta_*$ to $F_{\beta^*X} : \begin{tikzcd}\beta^* X\rar[nonlinear,"\phi\circ p_1"] & X' \rar["\psi'"] & \beta^* X \end{tikzcd} $ we get

\begin{equation} \tag{$Y_1$} \label{eq:formationY1}
 (F_{\beta^*X})^\flat : \: \begin{tikzcd}
		X \rar[nonlinear,"(\phi\circ p_1)^\flat"] & \beta_*X' \rar["\beta_*\psi'"] & \beta_*\beta^* X
	\end{tikzcd}
\end{equation}
On the other hand, starting from $ F_{\beta^*X} : \begin{tikzcd}\beta^*X \rar["\mathrm{id}_{\beta^*X}"] & \beta^*X \rar[nonlinear,"F_{\beta^*X}"] & \beta^* X \end{tikzcd}$
instead, we obtain
\[  (F_{\beta^*X})^\flat : \begin{tikzcd}X \rar["(\mathrm{id}_{\beta^*X})^\flat"] & \beta_*\beta^*X \rar[nonlinear,"\beta_*F_{\beta^*X}"] & \beta_*\beta^* X \end{tikzcd}. \]
The first arrow is just a cumbersome notation for $\eta_X$, the component at $X$ of the unit $\eta$ of the adjunction $\beta^*\dashv \beta_*$. Expressing that $\eta$ is a natural transformation $\mathrm{id}_\C\to \beta_*\beta^*$, we have a diagram
\[  
\begin{tikzcd} X \dar["F_X" ] \rar["\eta_X"] & \beta_*\beta^* X \dar["\beta_*\beta^*F_X"] \\ 
X \rar["\eta_X"] & \beta_*\beta^* X,
\end{tikzcd}
\]
where we note that $\beta_*\beta^*F_X=\beta_*F_{\beta^*X}$ to obtain
\begin{equation} \tag{$Y_2$} \label{eq:formationY2}
	(F_{\beta^*X})^\flat : \begin{tikzcd}X \rar[nonlinear,"F_X"] & X \rar["\eta_X"] & \beta_*\beta^* X. \end{tikzcd}
\end{equation}
Combining \eqref{eq:formationY1} and \eqref{eq:formationY2} we get the following diagram, where the dotted arrow is implied by the pullback square, i.e.~we define $\pi = F_x \times (\phi\circ p_1)^\flat$.
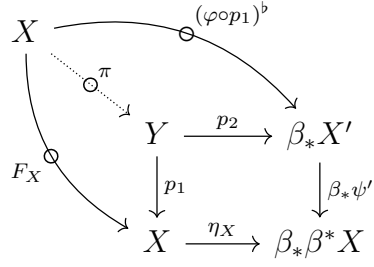
\begin{center}
\begin{tikzcd}
X \ar[dr,nonlinear,densely dotted,"\pi"] \ar[drr,nonlinear,bend left,"(\phi\circ p_1)^\flat"] \ar[ddr,nonlinear,bend right,swap,"F_X"] & & \\
& Y\ar[r,"p_2"]\ar[d,"p_1"] & \beta_* X' \ar[d,"\beta_*\psi'"] \\
& X\ar[r,"\eta_X"] & \beta_*\beta^* X
\end{tikzcd}
\end{center}

\textbf{Step 2: verifying $p_1\circ\pi = F_X$ and  $\pi\circ p_1 = F_Y$.} The first of these identities is clear from the construction of $\pi$. To show the second one, we first consider the following diagram:
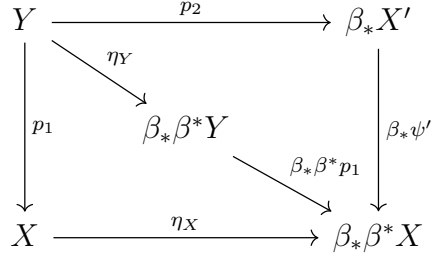
\[
	\begin{tikzcd}
		Y \ar[rr,"p_2"] \ar[dd,"p_1"] \ar[dr,"\eta_Y"] & & \beta_*X' \ar[dd,"\beta_*\psi'"]  \\
		& \beta_*\beta^*Y \ar[dr,"\beta_*\beta^*p_1"] & \\
		X\ar[rr,"\eta_X"] &  & \beta_*\beta^* X
	\end{tikzcd}
\]
The bottom triangle commutes again because $\eta$ is a natural transformation; the big square commutes by definition of $Y$. Thus the upper triangle commutes. Using this trangle, but noting that the diagonal is also $(\beta^* p_1)^\flat$, we obtain the diagram below on the left; using the adjunction we can push $\beta_*$ from the diagonal to $Y$ and  get the diagram on the right:
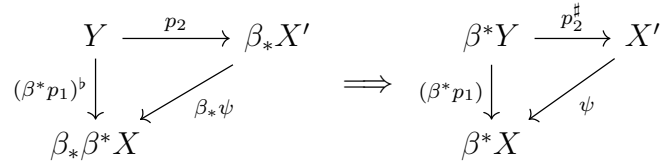
\[
	\begin{tikzcd}
	Y \dar[swap,"(\beta^*p_1)^\flat"] \ar[r,"p_2"] & \beta_* X' \ar[dl,"\beta_*\psi"] \\
	\beta_*\beta^*X
	\end{tikzcd}
	\implies
	\begin{tikzcd}
	\beta^*Y \ar[d,swap,"(\beta^*p_1)"] \ar[r,"p_2^\sharp"] &   X' \ar[dl,"\psi"] \\
	\beta^*X
	\end{tikzcd}
\]
Let us add to this diagram two more triangles which clearly commute in order to obtain the diagram on the left below; omitting the dotted arrows and pushing  $\beta^*$ to the right with the adjunction again, we get the diagram on the right.
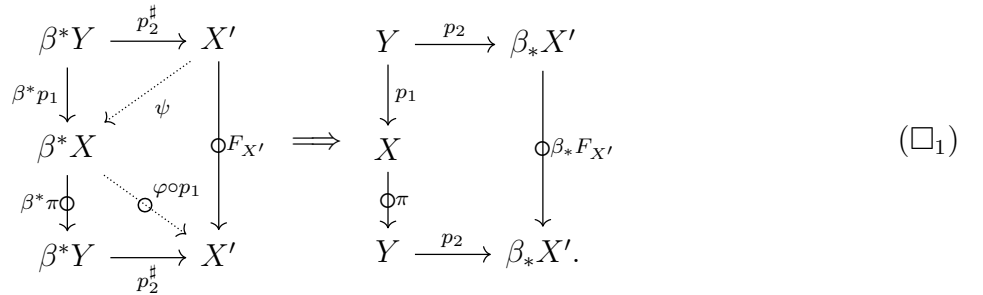
\begin{equation} \tag{$\Box_1$} \label{eq:square1}
	\begin{tikzcd}
	\beta^* Y \rar["p_2^\sharp"]\dar[swap,"\beta^*p_1"] & X'\ar[dd,nonlinear,"F_{X'}"] \ar[dl,densely dotted,"\psi"] \\
	\beta^*X \dar[swap,nonlinear,"\beta^*\pi"] \ar[dr,densely dotted,nonlinear,"\phi\circ p_1"] & \\
	\beta^*Y \rar[swap,"p_2^\sharp"] & X'
	\end{tikzcd}
	\implies
	\begin{tikzcd} Y \rar["p_2"] \dar["p_1"] & \beta_* X' \ar[dd,nonlinear,"\beta_* F_{X'}"] \\
	X \dar[nonlinear,"\pi"] & \\
	Y \rar["p_2"] & \beta_* X'.
	\end{tikzcd}
\end{equation}
On the other hand, we clearly have a square
\begin{equation}\tag{$\Box_2$} \label{eq:square2}
	\begin{tikzcd}
	Y \dar["p_1"] \rar["p_1"] & X \rar[nonlinear,"\pi"] \dar[densely dotted,"\mathrm{id}_X"] & Y \dar["p_1"] \\
	X \rar["\mathrm{id}_X"] & X \rar[nonlinear,"F_X"] & X.
	\end{tikzcd}
\end{equation}
Combining \eqref{eq:square1} and \eqref{eq:square2} with the pullback defining $Y$, we get a commuting diagram
\[
\begin{tikzcd}
Y \ar[rr,"p_2"]\ar[dd,"p_1"] \ar[dr,densely dotted,nonlinear,"\pi\circ p_1"] &  & \beta_*X' \dar[nonlinear,"\beta_*F_{X'}"] \\
& Y \ar[d,"p_1"] \ar[r,"p_2"] & \beta_*X' \ar[d,"\beta_*\psi'"] \\
X\rar[nonlinear,"F_X"] & X\rar["\eta_X"] & \beta_*\beta^* X.
\end{tikzcd}
\]
On the other hand, it is clear that if we replace the dotted arrow with $F_Y$, the diagram commutes as well. But this arrow is uniquely determined since the square in the bottom right is a pullback. So $\pi\circ p_1=F_Y$.

\textbf{Step 3: the maps between homsets.} Now we consider an arbitrary mixed $S$-object $\tilde T$, say $\tilde T = (T,T^\circ,\zeta,\xi)$. We compute $f^*\tilde T = T\times_S \tilde S = (T,\beta^*T^\circ,\zeta',\xi\circ p_1)$, where $p_1:\beta^*T^\circ\to T^\circ$ is the canonical projection and $\zeta=p_1\circ\zeta'$. Now consider the following diagrams:

\begin{center}
\begin{tikzcd}
	T \ar[r,shift left=0.5ex,"\zeta'"] \ar[d,"u"] & \beta^* T^\circ \ar[l,shift left=0.5ex,"\xi\circ p_1"]\ar[d,"v"] & T \ar[r,shift left=0.5ex,"\zeta"]\ar[d,"x"] & T^\circ \ar[l,shift left=0.5ex,"\xi"]\ar[d,"y"] \\
	X\ar[r,shift left=0.5ex,"\phi"]\ar[d,"q_X"] & X'\ar[l,shift left=0.5ex,"\psi"] \ar[d,"q_{X'}"] & X\ar[r,shift left=0.5ex,"\pi"]\ar[d] & Y\ar[l,shift left=0.5ex,"p_1"]\ar[d] \\
	S\ar[r,shift left=0.5ex,"\alpha"] & S'\ar[l,shift left=0.5ex,"\beta"] & S\ar[r,shift left=0.5ex,"F_S"] & S\ar[l,shift left=0.5ex,"\mathrm{id}_S"]
\end{tikzcd},
\end{center}

What we must show is this: to every pair of morphisms $(u,v)$ which makes the left diagram commute, there corresponds uniquely a pair  $(x,y)$ which makes the right diagram commute. We stress again that, as in \S\ref{sec:def-cat}, this is an abbreviation for a bigger commutative diagram since \begin{tikzcd} \bullet \ar[r,shift left=0.5ex] & \circ\ar[l,shift left=0.5ex]\end{tikzcd} does not compose to the identity but rather to $F_{\bullet}$ and $F_\circ$; furthermore all vertical compositions give the appropriate structural morphism. In fact, the correspondence is given by the formulas
\[
		\begin{cases}
			x = u \\
			y = (u\circ\xi) \times v^\flat
		\end{cases}\text{ and }
		\begin{cases}
				u = x \\
				v = (p_2\circ y)^\sharp
		\end{cases}
\]
Let us first verify that $(u\circ\xi) \times v^\flat$ actually defines a morphism $y:T^\circ\to Y$, i.e.~that $\eta_X\circ u\circ\xi  = \beta_*\psi'\circ v^\flat$. Using that $v^\flat = \beta_*v \circ \eta_T$, where $\eta_T:T\to \beta_*\beta^*T$ is the unit morphism, we must show that $\eta_X \circ (u\circ\xi) = \beta_*(\psi'\circ v)\circ \eta_T$. By naturality of the adjunction, it suffices to verify that $\psi'\circ v = \beta^*(u\circ \xi)$. To see this, we observe that the equality $\psi\circ v = u\circ\xi\circ p_1$ implies there is a diagram
\[
\begin{tikzcd}
\beta^* T^\circ \dar["p_1"]\rar["v"] & X' \rar["\psi'"] & \beta^*X \dar["p_1"] \\
T^\circ \rar["\xi"] & T \rar["u"] & X
\end{tikzcd}
\]
Since the arrows $v$ and $\psi'$ are both $S'$-linear, we have that $\psi'\circ v$ is an $S'$-linear arrow which makes the square above commute. The unique arrow with that property is $\beta^*(u\circ\xi)$, which completes our verification.

\textbf{Step 4: verifying bijectivity.} We can quickly verify that the two maps defined in step 3 are inverses of one another. In one direction, it is immediately clear that
\[ \Big(p_2\circ \big((u\circ\xi)\times v^\flat\big)\Big)^\sharp =  (v^\flat)^\sharp = v.  \]
In the other direction, we must verify that
\[  (x\circ\xi) \times (p_2\circ y) = y  \]
but this is a direct consequence of $p_1\circ y = x\circ\xi$.

\textbf{Step 5: naturality.} We now verify that the constructed isomorphism is natural in both arguments. Let us only illustrate this with in one case and leave the other verifications to the reader. Let $\tilde Y$ be another $\tilde S$-object together with a morphism $(a,b):\tilde X\to\tilde W$. Then we have corresponding square:
\[ 
\begin{tikzcd}
\hom_S(\tilde T,f_*\tilde X) \dar \rar & \hom_{\tilde S}(f^*\tilde T,\tilde X) \dar  \\
\hom_S(\tilde T,f_*\tilde W) \rar & \hom_{\tilde S}(f^*\tilde T,\tilde W)
\end{tikzcd}
\]
The vertical arrows are given by composition with $(a,s)$ and $(a,b)$ respectively, where 
\[ s : X\underset{\beta_*\beta^*X}\times \beta_*X' \to W\underset{\beta_*\beta^*W}\times \beta_*W' \] is constructed in the expected manner from $a:X\to W$ and $\beta_*b:X'\to W'$ and satisfies in particular $p_2\circ s = \beta_*b\circ p_2$. This immediately implies naturality condition
\[ (p_2\circ s\circ y)^\sharp = b\circ (p_2\circ y)^\sharp.  \qedhere \]
\end{proof}

\begin{remark} A few notes. 

\begin{enumerate}
\item Here is a heuristic: the object that one would like to write there is $f_*\tilde X = (X,\beta_*X')$. But there is no obvious map $\beta_*X'\to X$, so we form a product with $X$ and let the projection play the role of that map.
\item In truth, we guessed this description of $Y$ from \cite[\S7.2]{PRG} which in turn relies on Tits' notes \cite[\S4]{tits-cdf-9192}. Tits probably started from his description of mixed abstract groups \cite[(10.3.2)]{tits74} and somewhere along the way used some version of our Proposition~\ref{prop:mixedpoints} to come up with the corresponding exotic pseudo-reductive groups. In Remark~\ref{remark:mixedpints}.\ref{remark:mixedpoints-weil} we will explain this more in detail.
\end{enumerate}
\end{remark}

\begin{proposition} Let $f:S\to T$ be a morphism of visible objects and let $\tilde X = (X_1,X_2,\Phi_1,\Phi_2)$ be a mixed $S$-object. Then a right adjoint $f_*$ to the base change $f^*$ in $\C$ also defines a right adjoint to base change in $m\C$.
\end{proposition}
\begin{proof}
	It is easily verified that $f_*(\tilde X)=(f_*X_1,f_*X_2,f_*\Phi_1,f_*\Phi_2)$ has the required property that $\hom_T(\tilde Y,f_*\tilde X)=\hom_S(f^*\tilde Y,\tilde X)$ for every mixed $T$-object $\tilde Y$.
\end{proof}

\section{Twisting and mixing schemes}
\label{sec:mixed-schemes}
We will now take a step towards the applications that we have in mind and apply the results of the previous section to the categories of schemes and rings of a fixed characteristic $p>0$, \emph{chosing the absolute Frobenius as endomorphism of the identity functor}.

It should be noted that one could also choose the identity endomorphism of the identity functor. We will only mention that this comes down to the study of schemes with an involution and would eventually lead to a slightly different description of the Steinberg groups ${}^2\mathsf A_n$, ${}^2\mathsf D_n$, ${}^2\mathsf E_6$. The main difference here is a shift of viewpoint: usually one regards, say, $\mathsf{PSU}$ in the real/complex case  as an algebraic group over $\mathbb R$, whereas in our approach it would become a twisted group over the twisted field $(\mathbb C,\tau)$, where $\tau$ denotes complex conjugation. Since our main goal is to study and describe groups such as ${}^2\mathsf B_2$, ${}^2\mathsf G_2$, ${}^2\mathsf F_4$ we will not pursue this route any further. \medskip

Some examples will be grouped together in \S\ref{sec:examples}; the reader is encouraged to skip ahead.

\subsection{Twisted and mixed schemes} \label{sec:def-sch}
Applying the results of the previous section to the category $\Sch_p$ of schemes of characteristic $p$ with their absolute Frobenius $\Fr$ provides us with a number of categories and functors. We use the following notations for the occuring categories and their terminal objects, where we recall that $\mathbf 1_{t\C}=(\mathbf 1_{\C},\mathrm{id})$ and $1_{m\C}=\mathsf E = (\mathbf 2_{\C},\tau)$:

\[ \begin{array}{llllllll}
\text{general:}& \C & F & t\C & m\C & \mathbf 1_{\C}  &\mathbf 1_{t\C}  & \mathbf 1_{m\C} \\ \hline
\text{schemes:} &\Sch_p & \Fr & \TSch_p & \MSch_p & \Spec \mathbf F_p & \mathsf F & \mathsf E \\
\text{rings:} & \Ring_p & \fr & \TRing_p & \MRing_p & \mathbf F_p & \mathbf f & \mathbf e
\end{array} \]

We call $\TSch_p$ resp.~$\MSch_p$ the \emph{category of twisted resp.~mixed schemes} (of characteristic $p$). From now on, we will often omit the subscript $p$. Recall the notion of the underlying ordinary object, in this case the \emph{underlying ordinary scheme} of a twisted scheme $(X,\Phi_X)$, which is just the scheme $X$.

\begin{proposition} $\TSch$ and $\MSch$ have fibered products $X\times_S Y$ and terminal objects.
\end{proposition}
\begin{proof} This is an immediate consequence of Proposition \ref{prop:limits}.
\end{proof}

Although the category $\Ring$ is \emph{not} extensive, its opposite category $\Ring^\op = (\mathbf{affsch})$ is. Actually, let us show this to convince the reader it is not a deep fact. If we have an algebra morphism $f:\mathbf F_p\times \mathbf F_p\to A$, then $e_1=f(1,0)$ and $e_2=f(0,1)$ are orthogonal idempotents, so we obtain a decomposition $A \cong A_1\times A_2$ and the map $f$ is actually built from $f_1: \mathbf F_p\to A_1$ and $f_2:\mathbf F_p\to A_2$, and the product functor
\[  \Ring \times \Ring \to \Ring_{/(\mathbf F_p\times\mathbf F_p)}\alg{(\mathbf F_p\times\mathbf F_p)} : (A_1,A_2) \leadsto A_1\times A_2 \]
is an equivalence.

If $\tilde S$ is a twisted scheme, then by an $\tilde S$-scheme, we mean an arrow $\tilde X\to \tilde S$, in other words an object of the slice category $\TSch_{/\tilde S}$. Similarly if $\tilde R$ is a twisted ring, then by an $\tilde R$-algebra we mean an arrow $\tilde R\to \tilde S$, in other words an object of the coslice category $((\Ring^\op)_{/\tilde R})^{\op}$. 

\begin{proposition} We have equivalences $\TSch \cong \MSch_{/\mathsf E}$ and $\MRing \cong \alg{\mathbf e}$. The contravariant functor $\Spec:\Ring\to\Sch$ extends to the mixed and twisted categories in a way which commutes with all the functors $\mathbf f$, $\comp_1$, $\comp_2$, $\mix$, $\antimix$, $\delta_!$, $\delta^*$, $\delta_*$ and $\tau^*$.
\end{proposition}
\begin{proof} This is a consequence of Propositions~\ref{prop:equivalenceQ} and \ref{prop:functoriality}.
\end{proof}

An important observation is that the functors $\mix,\antimix : \Sch \to \MSch$ are not essentially surjective. Just like in the general case (Definition~\ref{def:visible}),  we will use the adjectives \emph{visible} resp.~\emph{anti-visible} for those mixed schemes that are isomorphic to $\mix(X)$ resp.~$\antimix(X)$, for some ordinary scheme $X$. A mixed scheme that is not visible is called \emph{invisible}.

Since the Frobenius of a scheme acts trivally on the underlying topological space, a twister acts in orbits of length $1$ or $2$.

\begin{definition} We will use the adjective \emph{blended} for twisted schemes where all orbits have length $1$, i.e.~the twister acts trivially on the underlying topological space.
\end{definition}

In particular every twisted field must be blended. These objects are also known as \emph{fields with Tits endomorphism}: they are just pairs $(k,\theta)$ where $k$ is a field and $\theta$ an endomorphism such that $x^{\theta^2}=x^p$ for every $x\in k$. A mixed ring can never have a field as its underlying ordinary ring, since mixed schemes are never connected. This allows us to define unambiguously:

\begin{definition} A \emph{mixed field} is a mixed ring $m=(k, \ell,\kappa,\lambda)$ such that $k$ and $\ell$ are fields.
\end{definition}

Equivalently, it is a mixed ring without twisted ideals in the sense of Definition~\ref{def:twisted-ideal}.

\subsection{Rational points and functor of points} \label{sec:sch-ratpts}
A standard tool in algebraic geometry is the functor of points associated to a scheme, for instance see \cite[exp.~I]{SGA3}. We apply it to our setting by using for their category $\C$ one of the categories $\TSch$ or $\MSch$.

To begin, we observe that a notion of rational points is provided by Definition~1.2 in \opcit: we define $\Gamma(X)=\hom(\mathbf 1,X)$, where $\mathbf 1$ is a terminal object in $\C$. \footnote{If there is no terminal object, one must rely on the Yoneda embedding to provide a non-representable terminal object; but we will not need that here.} We will sometimes write $\Gamma_\C(X)$ for clarity. For $S$-objects $X$ and $T$ will also use the notation
\[ \Gamma_\C(X/T) =  \{f \in \hom_\C(T,X) \mid q_X\circ f = q_T  \}. \]
 Often we also denote $\Gamma(X/T)=X(T)$. We will also use the canonical identifications
 \begin{align*}
	 \Gamma_{\C_{/S}}(X) &= \Gamma_\C(X/S) \\
	 \Gamma_{\C_{/S}}(X/T) &= \Gamma_{\C_{/T}}(X\times_S T/T)
 \end{align*}

When it is clear that $X$ and $T$ are $S$-objects, we can also use the notation $X(T) = \Gamma_{\C_{/S}}(X/T)$, so that the last equality reads $X(T)=X_T(T)$, where we also used the common index notation $X_T$ for the pullback $X\times_S T$. Finally, we note that in the examples we will sometimes write $\Gamma(R)$ or $\Gamma(R/S)$ rings $R$, $S$ a ring, which can be ordinary, twisted, or mixed. In such cases, we always mean $\Gamma(\Spec R)$ and $\Gamma(\Spec R/\Spec S)$.\medskip

Let us now study the interaction between rational points of ordinary schemes, twisted schemes, and mixed schemes.\medskip

To obtain something recognizable, we need to assume the base-scheme behaves well. For twisted schemes, this means that the twister must be invertible; for mixed schemes we need to assume that at least one of the mixing maps is epic.

\begin{proposition} \label{prop:fixedpoints} Let $\tilde S = (S,\Phi)$ be a twisted
	scheme \emph{with $\Phi$ invertible}. Let $\tilde X = (X,g)$ be an $\tilde S$-scheme. Define $\alpha$ as follows:
	\[ \alpha  :  X(S)\to X(S) : x \mapsto g \circ x \circ \Phi^{-1} .   \]
	Then $\alpha$ is an involution on $X(S)$ and $\tilde X(\tilde S)=X(S)^\alpha$ its set of fixed points.
\end{proposition}
\begin{proof}We have $(\alpha\circ\alpha)(x) = (g^2)\circ x \circ (\Phi^2)^{-1} = \Fr_X \circ x \circ \Fr_S^{-1} = x$ and  $\Gamma(\tilde X/\tilde S) = \{  x \in X(S) \mid x\circ\Phi = g\circ x\} = X(S)^\alpha$.
\end{proof}

In other words: \emph{the set of rational points of a twisted scheme over a base with invertible twister is the set of fixed points of an involution acting on the sets of rational points of an ordinary scheme.} The application that we have in mind is where $S$ is the spectrum of a perfect blended field. In \S\ref{sec:twisted-groups} we will see that this proposition generalizes the construction of the Suzuki-Ree groups over perfect fields. (Over non-perfect fields, there is actually nothing to show! More on this in Remark~\ref{remark:perfect-fields}.)

\begin{lemma} \label{lem:mixedpoints} Let $\tilde S=(S_1,S_2)$ be a mixed scheme $\tilde X=(X_1,X_2)$ an $\tilde S$-scheme. If $S_1$ is reduced, then $\Phi_{S_2}$ is epic.
\end{lemma}
Note that we now suppress the maps $\Phi_{X_1}$ etc.~ from the notation if we can, as we remarked just before Lemma~\ref{prop:limits}.

\begin{proof} The absolute Frobenius $\Fr_{S_1}:S_1\to S_1$ is topologically bijective, so in particular it is topologically surjective. Since $S_1$ is reduced, it is in addition injective on the stalks. This implies that $\Fr_{S_1}$ is an epimorphism and since $\Fr_{S_1} = \Phi_{S_2}\circ\Phi_{S_1}$, so is $\Phi_{S_2}$.
\end{proof}

\begin{proposition} \label{prop:mixedpoints} Let $\tilde S=(S_1,S_2)$ be a mixed scheme $\tilde X=(X_1,X_2)$ an $\tilde S$-scheme. \emph{Assume that $\Phi_{S_2}$ is epic.}

\begin{enumerate}
	\item \label{prop:mixedpoints2} The following maps are injective
	\begin{align*}
		g : X_1(S_1) \into X_1(S_2) &: v\mapsto v\circ\Phi_{S_2} \\
		\comp_1: \tilde X(\tilde S) \into X_2(S_2) &: (u,v)\mapsto u
	\end{align*}
	\item \label{prop:mixedpoints3} The following square is a pullback in $\Set$:
	\[
	\begin{tikzcd}
	\tilde X(\tilde S) \dar["\comp_2"] \rar[tail,"\comp_1"] & X_2(S_2) \dar["f"] \\ X_1(S_1) \rar[tail,"g"] & X_1(S_2)
	\end{tikzcd}
	\]
	where $f:X_2(S_2)\to X_1(S_2) : u \mapsto \Phi_{X_2}\circ u$.
	\item We have: $\tilde X(\tilde S) = f^{-1}(X_1(S_1))$.
\end{enumerate}
\end{proposition}

\begin{proof} \strut
	\begin{enumerate}
		\item Injectivity of $g$ is trivial; for $\comp_1$, if  $(u,v)$ and $(u',v)$ are two morphisms $\tilde S\to\tilde X$ then $u\circ\Phi_{S_2}=\Phi_{X_2}\circ v = u'\circ\Phi_{S_2}$ and so $u=u'$.
		\item  Let $u:S_1\to X_1$ and $v=S_2\to X_2$ be a pair of morphisms of schemes. We claim that $(u,v)$ is a morphism of mixed schemes $\tilde S\to \tilde X$ if (and only if) the condition $\Phi_{X_2}\circ v = u\circ \Phi_{S_2}$ is satisfied. In other words, we show that $v\circ\Phi_{S_1}=\Phi_{X_1}\circ u$ is a consequence. Because $\Phi_{S_2}$ is epic, this is equivalent to showing that
		\[ v\circ\Phi_{S_1}\circ \Phi_{S_2}=\Phi_{X_1}\circ u\circ\Phi_{S_2},  \]
		which is clear, since the left hand side is also
		\[ v\circ \Fr_{S_2} = \Fr_{S_2}\circ v=\Phi_{X_1}\circ\Phi_{X_2}\circ v. \] So we have
	\begin{align*}
		\tilde X(\tilde S) &= \{ (u,v) \in X_1(S_1) \times X_2(S_2) \mid \Phi_{X_2}\circ v = u\circ \Phi_{S_2}  \}
	\end{align*}
	In other words, this is the claimed fibered product in the category of sets.
	\item This is a reformulation of  \eqref{prop:mixedpoints3} with some abuse of language which takes \eqref{prop:mixedpoints2} into account. \qedhere
	\end{enumerate}
\end{proof}

\begin{remark} \label{remark:mixedpints} A few notes. \end{remark}

\begin{enumerate}
\item The proposition is purely categorical and holds not just for schemes but it was placed in this section for the good interaction with Lemma~\ref{lem:mixedpoints}.
\item The notation $f^{-1}(..)$ is slightly ambiguous in our situation. Since we have maps going from $S_1$ to $S_2$ and back, there are maps $X_2(S_2) \to X_2(S_1) \to X_2(S_2)$. So one has to be specific how $X_2(S_2)$ is considered a subset of $X_2(S_1)$ to avoid confusion.
\item In a typical situation $\tilde S=\Spec m$ could be the spectrum of a mixed field $m=(k,\ell)$. By Lemma~\ref{lem:mixedpoints} the proposition applies and in fact it applies \emph{mutatis mutandis} to the other component. So $\tilde X(m)$ is a subset of both $X_1(k)$ and $X_2(\ell)$.
\item An important situation where $\Phi_{S_2}$ is epic but $S_2$ is non-reduced is where $\tilde S$ is a non-reduced visible scheme so that in fact $\Phi_{S_2}$ is an isomorphism. For instance, take $\tilde S=\mix(\Spec k(\epsilon))$, where $k(\epsilon)$ are the dual numbers over a field $k$.
\item \label{remark:mixedpoints-weil} Assume one \emph{insists} on realizing the set $\tilde X(\tilde S)$ with $S_1$-objects. Then one has to find suitable $S_1$-objects such that $Y(S_1) = X_1(S_2)$ and $Z(S_1) = X_2(S_2)$ such that the induced maps are indeed $f$ and $g$. Denoting for simplicity $\beta = \Phi_{S_2}: S_2\to S_1$, we see that $Y = \beta_*\beta^*X_1$ and $Z = \beta_*X_2$ fit the bill. This implies that if we define a new $S_1$-scheme by $U = X_1 \times_{\beta_*\beta^* X_1} \beta_* X_2$, then $U(S_1)=\tilde X(\tilde S)$. This suggests the statement of Proposition~\ref{prop:weilrestriction}:  observing that $(\comp_2 f_* \tilde X)(S) = \tilde X(f^*\mix S)$ gives away the second component of $f_*\tilde X$. A similar reasoning, starting from Remark~\ref{remark:visible}.\ref{bullet:ratpts} could perhaps lead to a generalization of that proposition for arbitrary morphisms $(\alpha,\beta):(T,T')\to (S,S')$.
\end{enumerate}

\subsection{Modules and sheaves} \label{sec:modules-and-sheaves}
  We want to introduce a notion of modules over twisted rings, and more generally sheaves of modules over twisted schemes. The main purpose is to define partial dimensions of a mixed scheme. The underlying idea is that if we have a scheme over, say, a mixed field $(k,\ell)$ we want to measure how much of it is defined over $k$ and how much is defined over $\ell$. In \S\ref{sec:gps} certain mixed reductive groups, which are roughly reductive groups where the long and short roots live over different halves of a mixed field. In Remark~\ref{rem:parabolics} we will explain how these partial dimensions count the number of dimensions determined by short and long roots.

Let us first recall the notion of a $p$-structure (also called Frobenius structure) before considering its `square root'.

\begin{definition} A \emph{$p$-structure on a module $M$} over a ring $R$ of characteristic $p$ is a map $M\to M:x\mapsto x^\pstr$ such that $a^px^\pstr=(ax)^\pstr$ and $(x+y)^\pstr=x^\pstr+y^\pstr$, for $a\in R$ and $x,y\in M$.
\end{definition}
Note that in particular, $0 : x\mapsto 0$ is always a $p$-structure.

Here is a different angle on this definition: let $\mathrm{ad}:R\to\End(M,+)$ be the multiplication map $\mathrm{ad}(a)(x)=a\cdot x$. Then the $p$-structure $\pstr:M\to M$ satisfies
\[ \mathrm{ad}(a^p)\circ\pstr=  \mathrm{ad}(a)^p\circ\pstr = \pstr\circ \mathrm{ad}(a).  \]

\begin{definition} Let $\tilde R=(R,f)$ be a twisted ring. Then a \emph{twisted $\tilde R$-module} $M$ is an $R$-module together with a map $\psi:M\to f_*M$, where $f_*M$ is $M$, considered as an $R$-module through the multiplication $a\circ x=f(a)x$.
\end{definition}

In other words, it is a map $\psi:M\to M$ which is semi-linear in the following sense: \[ \psi(ax+y)=a\circ\psi(x)+\psi(y)=f(a)\psi(x)+\psi(y).\]

Next we globalize these notions to obtain those of twisted sheaves.

\begin{definition}A \emph{$p$-structure on a sheaf $\mathscr F$ of modules} over a scheme $(X,\mathscr O_X)$ with absolute Frobenius $\mathrm{Fr}$ is a  is a morphism $\pstr: \mathscr F \to {\mathrm{Fr}}_*\mathscr F$.
\end{definition}

\begin{definition} A \emph{twisted sheaf of modules} $(\mathscr F,\psi)$ over a twisted scheme $(X,\Phi)$ is a sheaf $\mathscr F$ on $X$ together with a morphism $\psi : \mathscr F\to\Phi_*\mathscr F$.
\end{definition}

For the convenience of the reader, let us break down this definition. For every open $U\subseteq |X|$, denote $\phi(U)=\phi^{-1}(U)=\overline U$. Then the twisted structure defines a morphism
\[   \psi_U : \mathscr F(U) \to \mathscr F(\overline U),  \]
of $\mathscr O_X(U)$-modules, where $\mathscr F(\overline U)$ is considered an $\mathscr O_X(U)$-module through the map $\phi^\sharp : \mathscr O_X(U)\to\mathscr O_X(\overline U)$, and for every inclusion $V\subseteq U$, the obvious diagram commutes. For future reference, let us draw attention to the case of a mixed field.

\begin{definition} A \emph{mixed vector space} over a mixed field $m=(k,\ell,\kappa,\lambda)$ is a tuple  $V_m = (V_k,V_\ell,\hat\kappa,\hat\lambda)$, consisting of a $k$-vectorspace  $V_k$, an $\ell$-vectorspace  $V_\ell$ and a pair of semi-linear maps $\hat\kappa: V_k\to V_\ell$, $\hat\lambda : V_\ell \to V_k$, where semi-linear means that $\hat\kappa(ax+y) = \kappa(a)\hat\kappa(x) + \hat\kappa(y)$ whenever $a\in k$, $x,y\in V_k$ and vice versa for $\lambda$. The partial dimenions are given by
	\[  \pardim V_m = \big( \dim_k (V_k/\ker\hat\kappa ) , \dim_\ell (V_\ell/\ker\hat\lambda) \big)  \]
\end{definition}


An important class of twisted modules comes from twisted ideals:

 \begin{definition} \label{def:twisted-ideal} A \emph{twisted ideal} $\mathfrak a$ of a twisted ring $(R,f)$ is an ideal $\mathfrak a\unlhd R$ such that $f(\mathfrak a)\subseteq \mathfrak a$.  A \emph{twisted sheaf} $(\mathscr J,\psi)$ on a twisted scheme $(X,\psi)$ is a subsheaf of $\mathscr O_X$ such that the inclusion $\mathscr J\into\mathscr O_X$ respects twisters.
 \end{definition}

In other words, the twisted ideals are precisely those ideals for which $f$ induces a twisted structure on $R/\mathfrak a$ and a twisted sheaf of ideals provides the structure of a twisted scheme on the corresponding closed subscheme. In particular, observe that the structure sheaf $\mathscr O_X$ on a twisted scheme is a twisted sheaf, the twister being  $\phi^\sharp : \mathscr O_X\to \phi_*\mathscr O_X$.

\begin{proposition} The sheaf of differentials $\Omega_{X/S}$ on a twisted scheme $X$ over a base twisted scheme $S$ is canonically endowed with the structure of a twisted sheaf.
\end{proposition}

\begin{proof} Taking differentials of the commutative square $\Phi_S\circ q_X=q_X\circ\Phi_X$ results in a map $d\Phi : \Omega_{X/S}\to (\Phi_X)_*\Omega_{X/S}$ in the usual manner.\end{proof}

\begin{remark} Originally, we had planned to define a twisted version of the tangent bundle $\Tan X$. The ambition was to endow the tangent space $\Tan_e G = \Tan G \times_G S$ over the neutral element $e:S\to G$ of a mixed $S$-group scheme $G$ (see Definition~\ref{def:group-objects}) with the structure of a   (to be defined) mixed Lie algebra $(L_1,L_2,\psi_1,\psi_2)$ which, in the visible case, would boil down to the Lie algebra endowed with the $p$-operation $(L,L,[p],0)$. 

Unfortunately, we couldn't make this idea work. The problem always boils down to the observation that homomorphisms between twisted $\tilde R$-modules are not themselves twisted $\tilde R$-modules, and in particular the dual $M^\vee=\hom(M,\tilde R)$ of a twisted module $(M,\psi)$ is not itself a twisted module in a canonical way --- unless $\psi$ is bijective. In other words, we lack \emph{internal hom objects} and because of this, there is no obvious way to dualize the twisted cotangent bundle $(\Omega_{X/S},d\Phi)$. We want to share several ideas which may contribute to resolve this difficulty, but so far we could not make any subset of them work to our satisfaction. 
\begin{itemize}
\item The hom-sets are canonically endowed with a $p$-structure, perhaps we should accept that for a twisted object $X\in\ob(t\C)$ its tangent bundle is an ordinary object $\Tan X \in \ob(\C)$.
\item We can just accept that the cotangent sheaf is the fundamental object, and define a mixed Lie co-algebra directly without dualizing.
\item The definition from \cite[exp.~II, 3.1]{SGA3} relies on the inner hom in the category $\widehat{\Sch} = \Hom(\Sch^\op,\Set)$. This works here too, but then there is no guarantee that the resulting functor is representable. 
\item If $\epsilon : S[\epsilon] \to S$ is a first order infinitesimal extension with section $\delta : S\to S[\epsilon]$ then the tangent bundle can be defined as $\Tan X = \epsilon_* \epsilon^* X$. Since the problem lies with the functor $\epsilon_*$, we could \emph{define} $\Tan X = \epsilon^* X$ instead. Instead of an $S$-linear diagram $X\to \Tan X = \epsilon_*\epsilon^* X\to X$, we would obtain a diagram
\[   
\begin{tikzcd} X \simeq \delta^*\epsilon^* X \dar\rar & \Tan X = \epsilon^* X \dar\rar & X \dar \\ S \rar["\delta"] & S[\epsilon] \rar["\epsilon"] & S \end{tikzcd}
\]
So perhaps we should take such non-linearity for granted.
\item Perhaps we should focus on mixed formal groups and ignore the Lie algebra altogether.
\end{itemize} 
A working definition of such a mixed Lie algebra would be of great interest, because it could lead the way to a construction of algebraic objects, such that their automorphism groups are precisely the mixed groups that we will introduce in \S\ref{sec:gps}. (I.e.~mixed versions of quadratic spaces, octonion algebras and Albert algebras.)
\end{remark}

\begin{remark} We have restricted ourselves to the study of what is strictly necessary for \S\ref{sec:mixed-groups} and Remark~\ref{rem:parabolics} in particular, where we study some groups that we find particularly interesting. Nonetheless, it could be fruitful to redefine some of the common adjectives from algebraic geometry (connected, smooth) in the mixed setting. This could shine light on natural questions---consider a mixed affine variety over a mixed field $(k,\ell)$; when do the partial dimensions add up to the dimension of each of the components?---and allow for an easy generalization of theorems that have been proven in a sufficiently `generic'---e.g.~not relying on arguments about rational points over an algebraic closure---manner.
\end{remark}

\subsection{Examples} \label{sec:examples}
As before all rings are assumed to be of characteristic $p$.
\begin{example}
	\label{ex:mixedrings} Twisted and mixed rings.
\end{example}
	\begin{enumerate}
	\item A pair $b=(k,\psi)$ where $\psi^2(x)=x^p$ is a blended field. These are also known as fields with \emph{Tits endomorphism} \cite{doublytwisted} or (for $p=2$) \emph{octogonal sets} \cite[(10.11)]{moufangpolygons} .
	\item \label{ex:nestedfields} Consider fields $k,\ell$ such that $k^p \subseteq \ell\subseteq k$. Then
	\[ m = (k,\ell,\mathrm{inc}_{k^p\into \ell}\circ\fr_k,\mathrm{inc}_{\ell\into k}) 	\] is a mixed field; we leave it to the reader to verify that in fact every mixed field is of this form. As an $(\mathbf F_p \times\mathbf F_p)$-algebra this is the twisted ring $(k\times \ell, (x,y)\mapsto (y,x^p))$ with the obvious structural morphism. The mixed field is visible in the extreme case where $\ell = k$ and it is anti-visible in the other extreme where $\ell = k^p$.
	\item \label{ex:A1} The pair $R = \big( \mathbf F_p[x,y], \phi\big)$, where $\phi :  f(x,y) \mapsto f(y,x^p)$ is a twisted ring. It is not mixed or blended. Taking the tensor product with a blended field $b=(k,\psi)$ gives the twisted ring (and $b$-algebra)
	\[ R_b = R\otimes b = \big(k[x,y], f(x,y) \mapsto f^\psi(y,x^p) \big), \] where $f^\psi$ means: apply $\psi$ to the coefficients of $f$. Taking the tensor product with a mixed field $m=(k,\ell,\kappa,\lambda)$ instead, we get a mixed ring (and $m$-algebra)
	\[  R_m = R\otimes m = \big(   k[x,y], \ell[x,y], \hat\kappa,\hat\lambda\big), \]
	where $\hat\kappa : f(x,y)\mapsto f^\kappa(y,x^p)$ and $\hat\lambda : f(x,y)\mapsto f^\lambda(y,x^p)$.
	\item For every ring $K$, the mixing functor $\mix$ defines the visible mixed ring $\mix(K)=(K,K,\fr,\mathrm{id})$. For instance, $\mix(\mathbf F_p) = \mathsf e$. If the Frobenius map $\mathrm{fr}_K : K\to K:x\mapsto x^p$ is injective, and in particular if $K$ is reduced, then $\mix(K)$ is isomorphic to the mixed ring $(K^p, K,\mathrm{inc},\fr)$, the isomorphism being given by $(\fr,\mathrm{id})$. We may identify a mixed ring $M=(K,L,\kappa,\lambda)$ with a certain ring extension $K'^p\subseteq L'\subseteq K'$ as we did for fields in example~\ref{ex:nestedfields}, provided that $\lambda$ is injective by taking $K'=K$ and $L'=\im\lambda$. Recall from Remark~\ref{remark:visible} that a mixed ring $M = (K,L,\kappa,\lambda)$ is visible precisely when $\lambda$ is an isomorphism; this corresponds to the extreme situation where $L'=K'$. 
	\item If $b=(k,\psi)$ is a blended field then
	\[ m = b\otimes \mathsf e = (k,k,\psi,\psi) \] is a mixed field and the diagonal $\chi : b\to m: a\mapsto (a,a)$ provides an embedding of twisted rings. This gives in turn rise to a functor of base change $\chi^*:\alg{b} \to \alg{m} : A \leadsto A\otimes_b m$, which corresponds to the twixting functor introduced in \S\ref{sec:functors}, but taking slices over the object $b$.
	\item \label{ex:extension} If $M=(K,L,\kappa,\lambda)$ is a mixed ring then any $a\in L$ defines the new mixed ring $M' = (K[x]/(x^p-a^\lambda),L,\hat\kappa,\hat\lambda)$ given by
	\[  \begin{cases} \hat\kappa  &: f(x) \mapsto f^\kappa(a), \\ \hat\lambda &: u \mapsto \lambda(u).  \end{cases} \]
	We can denote this construction by $M'= (K(a^{\kappa^{-1}}),L)$; the same construction can be carried out with a family of elements $a_i\in L$ $i\in I$ and the ring $K[x_i, i\in I]/(x^p - a_i^\lambda \mid i\in I)$.
	\item Consider field extensions $k/\ell$ such that $k^p\subseteq \ell$, and a pair of field extensions (or more generally \'etale algebras) $K/k$ and $L/\ell$ such that $K^p\subseteq L$. Then $(K,L)/(k,\ell)$ is an extension of mixed fields; moreover it is clear that every mixed field extension arises this way. 
	\end{enumerate}

\begin{example}
	Rational points of some of these examples:
\end{example}
\begin{enumerate}
\item It is clear that $\Gamma(b/b)$ is a singleton. Typically, one considers a field as the terminal object of its own slice category. This convention justifies the idea that $\Gamma(b)$ is a singleton.
\item Here too $\Gamma(m)$, an abuse of notation for $\Gamma(m/m)$, is a singleton.
\item Let us first compute $\Gamma(R_b/b)$.
  \begin{align*} \Gamma(R_b/b) &=  \hom_{\TSch}(\Spec b,\Spec R_b) \\ &=\hom_{\TRing}(R_b,b), \\
  &= \{\alpha \in \hom_{\Ring}(k[x,y],k) \mid \alpha\circ\phi =\psi\circ\phi  \}.
 \end{align*}
  Any such $\alpha$ is fully specified by $\alpha(x)=x_0$ and $\alpha(y)=y_0$ and the condition says that $y_0=\psi(x_0)$. So as a set, there is an identification $\Gamma(R_b/b) \cong k$.
  
  Similarly, it is easy to verify directly that there is an identification $\Gamma(R_m/m) \cong m$, but let us verify this again with Proposition~\ref{prop:mixedpoints}, using the notations $S_1 = \Spec k$, $S_2 = \Spec \ell$,  $X_1 = \Spec k[x,y]\cong  \mathbb A^2_k$ and $X_2 = \Spec \ell[x,y] \cong \mathbb A^2_\ell$. Then the proposition tells us that there is an identification of $\Gamma(R_m/m) = \tilde X(\tilde S)$ and $f^{-1}(X_1(S_1))$, where $f$ is the induced map $f:X_2(S_2)\to X_1(S_2)$ and $X_1(S_1) \subseteq X_1(S_2)$ in the natural manner. In our case, there is a natural identification of $X_1(S_1)$ with $k\times k$ on the one hand and $X_2(S_2)$ and $X_1(S_2)$ with $\ell\times \ell$ on the other hand; moreover the inclusion $X_1(S_1)\into X_1(S_2)$ corresponds to the inclusion $k\times k\into \ell\times\ell$ and the map $f$ can be identified with
  \[  \ell \times \ell : (a,b) \mapsto (b^p,a).  \]
  Therefore $f^{-1}(X_1(S_1))$ corresponds to the subset of all $(a,b)\in \ell\times\ell$ such that $b^p\in k$ and $a\in k$, where the first condition is of course always satisfied since $\ell^p\subseteq k$. So this corresponds precisely to the set $k\times \ell$ and the inclusion $\tilde X(\tilde S)\into X_2(S_2)$ is just the natural inclusion $k\times\ell \into \ell\times\ell  : (u,v)\mapsto (\kappa(u),v)$.
\item By fullness of the mixing functor $\mix$, we have that 
\[  \Gamma_{\MRing}(\mix(K)/\mix(k)) \simeq \Gamma_{\Ring}(K/k).  \]
\end{enumerate}

The next few examples are related to algebraic groups.

\begin{example} \label{ex:torus} Mixing tori.
	
Let $k$ be a field of characteristic $2$ and let $\beta : k\into K$ be a separable field extension of degree $2$ with Galois group $\langle\sigma\rangle$. Let $\mathsf{GL}_1$ denote the $k$-group `$k^\times$'; let $\mathsf{GL}_1^\sigma$ denote its Galois twisted form, i.e. the $1$-dimensional non-split torus corresponding to the kernel of the norm map $K^\times \to k^\times : x \mapsto xx^\sigma$; and let $\mathsf T_2 = \beta_*\beta^* \mathsf{GL}_1$ denote the non-split $2$-dimensional torus coming from the Weil restriction, i.e. $\mathsf T_2 = \mathsf R_{K/k}(\mathsf{GL}_1\otimes K)$. Explicitly, if $K=k[x]/(x^2+x+\delta)=k(u)$, then the Galois involution is given by $u^\sigma=u+1$. The coordinate algebras are given by
\begin{align*}
	\mathcal O(\mathsf{GL_1}) &= k[x,y]/(xy-1) \\
	\mathcal O(\mathsf{GL_1^\sigma}) &= k[x,y]/(x^2+xy+y^2\delta) \\
	\mathcal O(\mathsf T) &= k[x_1,x_2,y_1,y_2]/ \big( x_1x_2+dy_1y_2+1, (x_1+y_1)(x_2+y_2)+x_1x_2 \big)
\end{align*}
We identify a $k$-rational point $\mathsf{GL}_1$ with an element of $k^\times$; a $k$-rational point of $\mathsf{GL}_1^\sigma$ with an element of $K^\times$ of norm $1$; and a $k$-rational point of $\mathsf T_2$ with an element of $K^\times$. Then we may define isogenies on rational points by
\begin{align*}
	\mathsf{GL}_1 \times \mathsf{GL}_1^\sigma \to \mathsf T_2 &: (s,u) \mapsto su  \\
	\mathsf{T_2} \to \mathsf{GL_1} \times \mathsf{GL}_1^\sigma &: x \mapsto (xx^\sigma,x/x^\sigma)
\end{align*}
 If we have a field extension $\ell/k$ such that $\ell^2\subseteq k$, we may use all this data to construct a mixed variety $\tilde X$ over $(k,\ell)$ such that the set of rational points can be identified with the set
\[  \tilde X(k,\ell)=\{ x \in L^\times  \mid N(x) = xx^\sigma \in k \} \subseteq \mathsf T_2(\ell),  \]
where $L = \ell(u)$. In other words, this is the set $N^{-1}(k)$, where $N:L^\times\to \ell^\times$ is the usual norm of the Galois extension $L/\ell$.
\end{example}

\begin{example}
	Mixing an adjoint with a simply connected group.

 Consider algebraic $k$-groups corresponding to the adjoint and simply connected split groups of type $\mathsf A_{p-1}$, i.e. $G = \mathsf A_{p-1}^{\mathrm{sc}}=\mathsf{SL}_{p}$ and $H = \mathsf A_{p-1}^{\mathrm{ad}} =  \mathsf{PGL}_{p}$. For an $k$-algebra $K$, a $K$-rational point of $G$ can be represented by a $p\times p$-matrix $(x_{ij})$ with determinant $=1$ and a rational point of $H$ by an invertible $p\times p$-matrix determined up to a unit $[x_{ij}] = [\lambda x_{ij}]$, $\lambda\in K^\times$. Then there are maps, determined on rational points by
	\begin{align*}
	  \mathsf{SL}_p(K) &\longrightarrow \mathsf{PGL}_p(K) \longrightarrow \Fr^*\mathsf{SL}_p(K) \\
	  (x_{ij}) & \longmapsto [x_{ij}] \longmapsto (x_{ij}^p)/\det(x_{ij})
	 \end{align*}
 Since the first map is epic, we get a relative factorization by Lemma~\ref{lem:relfac} and so a mixed $k$-object by Proposition~\ref{prop:relfac}. Together with an absolute factorization of the base, i.e. a field extension $\ell/k$ such that $\ell^p\subseteq k$, we get a mixed object as in Corollary~\ref{cor:relfac}.

 \[
 \begin{tikzcd}
  (\mathsf{SL}_p)_k  \rar[shift left=0.5ex] \dar & (\mathsf{PGL}_p)_\ell \lar[shift left=0.5ex] \dar \\
  \Spec k  \rar[shift left=0.5ex] & \Spec \ell \lar[shift left=0.5ex]
 \end{tikzcd}
 \]
 To define the rational points of this object as embedded in $\mathsf{SL}_p(k)$ and defined by polynomials, we consider the map $f:\mathsf{SL}_p(\ell) \to \mathsf{PGL}_p(\ell)$ and compute $f^{-1}(\mathsf{PGL}_p(k))$. These are just the $p\times p$-matrices $(x_{ij})$ such that $\lambda x_{ij} \in k$ for some $\lambda \in \ell$. In other words, the rational points are described by the ordinary equation $\det(x_{ij}) = 1$ together with the `mixing equations' $p(x_{ij})\in \ell$ where $p$ runs through the monomials in the $(x_{ij})$ of degree $p$.
\end{example}

\section{Twisting and mixing groups} \label{sec:gps}
We will now explain how various classes of exotic groups can be integrated in our theory. Continuing our discussion about twisted and mixed schemes, we will define twisted and mixed group schemes by recycling Definitions 2.1.1.~and 2.1.2.~from \cite{SGA3}:

\begin{definition} \label{def:group-objects} A twisted group scheme is a group object in $\TSch$; a mixed group scheme is a group object in $\MSch$.
\end{definition}

\subsection{Informal statement of the theorems} \label{sec:gps-statement}
Let us first \emph{informally} state our main theorems and provide some context. 

\begin{theorem*} All twisted abstract groups arise as rational points of twisted group schemes.
\end{theorem*}
This is an informal statement because the notion of a twisted abstract group is not well defined in the literature.  Rather, there is a list of examples which pop up and are referred to as twisted Chevalley groups. We attempt to tell the full story in \S\ref{sec:history}, let us now just sketch an overview of known twisted groups, a list which includes in the perfect and in particular the finite case, the Suzuki \cite{SuzukiB2} and Ree \cite{ReeF4,ReeG2} groups. Although more ad-hoc constructions have been found by Wilson \cite{Wilson13}, the construction by Ree and the exposition by Carter~\cite{carter} remain the golden standard. With much of the research being focussed on the finite case, it is less widely known that this construction was extended to imperfect fields by Tits~\cite{suzukietree}. All subsequent research into these groups over imperfect fields is closely related to the theory of Moufang buildings or Moufang sets (a substitute for the Moufang buildings of rank $1$); to be precise ${}^2\mathsf B_2$ and ${}^2\mathsf G_2$ admit a BN-pair of rank $1$, so correspond to Moufang sets and ${}^2\mathsf F_4$ admits a BN-pair of rank $2$ and pops up in the classification of Moufang polygons \cite[(41.21.v)]{moufangpolygons} under the guise of the Moufang octagons --- although there they are grouped together with the mixed groups. As far as \emph{forms} of these twisted groups are concerned, we think that forms of ${}^2\mathsf F_4$ of relative rank $1$ (i.e.~Moufang sets) have been investigated in \cite{F4sets}; the main result of \cite{doublytwisted} states in some sense that twisted descent commutes with Galois descent. Of anisotropic forms of ${}^2\mathsf B_2$, ${}^2\mathsf G_2$ and ${}^2\mathsf F_4$, there is no trace in the literature, presumably because there is no geometric structure attached to them.\medskip

Because of all this, \emph{we interpret this theorem as speaking only about split twisted groups over perfect fields}. See Theorem~\ref{theorem:twisted-groups} for a precise statement and proof.

\begin{theorem*} All mixed abstract groups arise as rational points of mixed group schemes.
\end{theorem*}

Our main reference here is Tits' 1974 lecture notes on buildings \cite[(10.3.2)]{tits74}, where he defines classes of abstract groups that we think of as \emph{split mixed groups};  the construction works over an arbitrary field of the appropriate characteristic. These groups are adjoint groups; other isogeny types are not explicitly mentioned.

For \emph{forms of these groups}, however, the literature is quite confused. In \cite[Ch.~41]{moufangpolygons} the buildings related to forms of mixed $\mathsf{BC}_n$ are swept under the rug of the \emph{classical buildings} so they don't show up directly.  Somewhat further in \cite[(41.20)]{moufangpolygons}, Weiss lists the Moufang spherical buildings which are neither classical nor algebraic and it is suggested that they are associated to ``$(K,k)$-forms" of these split mixed groups but these are not further defined. In fact, the main reason Weiss considers these forms is his discovery of an exotic class of Moufang quadrangles which go by the name \emph{mixed quadrangles of type $\mathsf F_4$}. The situation is further confused since the octagons (related to \emph{twisted groups}) are on the list. Finally, in \cite{callensdemedts}, a certain Moufang set is constructed which we suspect to be a form of mixed $\mathsf F_4$ which arises by mixing together a split $\mathsf F_4$ and one of relative rank $1$.

Because of all this, \emph{we interpret these theorems as speaking only about the split mixed groups}. See Theorem~\ref{theorem:mixed-groups} for a precise statement and proof.\medskip

In \cite[(Ch.~7)]{PRG} Tits' constructions are revisited by Conrad, Gabber and Prasad using an alternative approach which is `well suited to working with arbitrary $k$-forms', but the context different from groups related to buildings. Rather, the subject of their study is a closely related class of algebraic groups whose construction relies on the existence of a particular type of isogeny. Our final theorem says how these groups are related to the mixed algebraic groups:

\begin{theorem*} The exotic pseudo-reductive groups \cite[(8.2.2)]{PRG} are Weil restrictions of (reductive) mixed group schemes.
\end{theorem*}

See Theorem~\ref{theorem:pseudo-reductive-groups} for a precise statement and proof.

\subsection{Existence of mixed groups} \label{sec:existence}
First we want to construct some mixed algebraic groups. We will show existence of some groups $(G_1,G_2,\phi_1,\phi_2)$ over \emph{visible} fields $(k,k,\fr_k,\mathrm{id}_k)$, relying on  Proposition~\ref{prop:relfac} for a construction from a relative factorization
\[ \Fr_{G_2/k} :  G_2 \overset{\phi_2}\to G_1 \to  G_2 \times_{\mathrm{fr}_k} \Spec k = \mathrm{Fr}_k^* G_2 = \flift G_2.  \]

Most often, $G_2$ is smooth over $k$ and therefore it is reduced. But then the Frobenius $\Fr_{G_2}$ is an epimorphism and this implies (as in Lemma~\ref{lem:relfac}) that  $\phi_2$ is epic as well. Since the kernel of $\phi_2$ is contained in the kernel of the relative frobenius $\Fr_{G_2/k}: G_2 \to \flift G_2$, we can use Borel's technique \cite[\S17]{borel-lag} of taking the quotient of a group by a Lie subalgebra to construct all mixed $k$-groups $(G_1,G_2)$ for a fixed smooth $G=G_2$ as follows.\medskip

Let $\mathfrak g$ denote the Lie-algebra of $G$ and $\mathfrak h$ a restricted subalgebra which is $\mathrm{Ad}$-invariant. Then there is a $k$-group $G/\mathfrak h$ and a $k$-isogeny $\pi : G\to G/\mathfrak h$ such that its differential factors as
\[ d\pi \: : \: \mathrm{Lie(G)} = \mathfrak g\onto\mathfrak g/\mathfrak h  \into \mathrm{Lie}(G/\mathfrak h). \]
By taking the quotient of $G/\mathfrak h$ with $\im(d\pi)$ we obtain a map $\overline \pi : G/\mathfrak h\to \Fr^* G$.

Note that the partial dimensions of the resulting mixed object are by construction given by $(\dim\mathfrak h,\dim\mathfrak g-\dim\mathfrak h)$. By Lemma~\ref{lem:relfac}, these isogenies correspond to relative factorizations in the sense of Definition~\ref{def:relfac}; applying Proposition~\ref{prop:relfac} then yields the mixed object $(G/\mathfrak h,G,p\circ\overline\pi,\pi)$ where $p$ is the (non $k$-linear) projection $\Fr^*G\to G$. 

\begin{proposition} \label{prop:veryspecialisogeny2}
 To every column of the following table corresponds a mixed object.
	\begin{center}\begin{tabular}{c|cccc|cc|cc}
			$p$ & $2$ & $2$ & $2$ & $3$ & $2$ & $2$ & $2$ & $2$\\
			${\mathsf X}$ & ${\mathsf B}_n^{\mathrm{sc}}$ &${\mathsf C}_n^{\mathrm{sc}}$ & ${\mathsf F}_4$ & ${\mathsf G}_2$  & ${\mathsf B}_n^{\mathrm{ad}}$ & ${\mathsf C}_{2n}^{\mathrm{ad}}$ & ${\mathsf B}_{n}^{\mathrm{ad}}$ & ${\mathsf C}_n^{\mathrm{ad}}$\\
			${\mathsf Y}$ & ${\mathsf C}_n^{\mathrm{sc}}$& ${\mathsf B}_n^{\mathrm{sc}}$ & ${\mathsf F}_4$ & ${\mathsf G}_2$ & ${\mathsf C}_n^{\mathrm{sc}}$& ${\mathsf B}_{2n}^{\mathrm{sc}}$ & ${\mathsf C}_{n}^{\mathrm{ad}}$ & ${\mathsf B}_{n}^{\mathrm{ad}}$
		\end{tabular}\end{center}
 More precisely: let $k$ be a field of characteristic $p$; let $\mathsf X$ denote a semi-simple $k$-group of that type. Then there exists a $k$-group $\mathsf Y$ of indicated type such that there is a mixed $k$-group $\mathsf {MXY}_n = (\mathsf Y, \mathsf X, p\circ\overline\pi,\pi)$.
	\end{proposition}
	\begin{proof}  The simply connected columns (in particular the $\mathsf F_4$ and $\mathsf G_2$ columns) are dealt with by \cite[(7.1.3--5)]{PRG} where the \emph{very special} isogenies are constructed as quotients of $\mathsf X$ with a Lie-subalgebra as in the preceding discussion.  Since the very special isogenies send the centers to the centers, we immediately get the similar result for the adjoint groups.

	Moreover, following the notation of \emph{loc.~cit~}(7.1.2) in the cases $\mathsf X  = \mathsf B_{n}^{\mathrm{sc}}$ or $\mathsf C_{2n}^{\mathrm{sc}}$ the irreducible submodule $\mathfrak z=\mathrm{Lie}(Z)$ is contained in  the kernel $\mathfrak n$ of the very special isogeny (on a separable closure). So the very special isogeny $\pi$ kills the center, and thus factors over the adjoint group $\mathsf X^{\mathrm{ad}}=\mathsf X/Z(\mathsf X)$ and we obtain a diagram of epimorphisms
	\[
		\begin{tikzcd} \mathsf X^{\mathrm{sc}} \ar[dr,swap,"p_1"]& & \mathsf Y^{\mathrm{sc}} \ar[ll,swap,"\pi'=p\circ\overline\pi"] \\ & \mathsf X^{\mathrm{ad}} \ar[ur,swap,"\alpha"]
		\end{tikzcd}
	\]
	such that $\alpha\circ p_1\circ\pi'$ and $\pi'\circ\alpha\circ p_1$ are the absolute Frobenius on the respective objects. Since $p_1$ is epic, we have that $p_1\circ\pi'\circ\alpha= \Fr$ iff $p_1\circ\pi'\circ\alpha\circ p_1 = \Fr\circ p_1=p_1\circ\Fr$, which is clearly the case and this gives the isogenies between simply connected $\mathsf B$ and adjoint $\mathsf C$ types and vice versa. 
	\end{proof}

\begin{remark} \label{remark:isogenies} \strut 
\begin{enumerate}
\item A base change will provide us now with a large number of mixed groups over arbitrary mixed fields. But, as we noted earlier in Remark~\ref{rem:less-accessible}, there is no reason to believe that every mixed group can be realised in this manner. In fact, the next bullets show that one can never obtain groups of type $\mathsf B_n/\mathsf C_n$ with the group of type $\mathsf C_n$ non-split by base changing one of the special isogenies between groups of types $\mathsf B_n\to\mathsf C_n$ and it seems likely that the same observation is true with the roles $\mathsf B$ and $\mathsf C$ swapped. So although there could a priori certainly exist mixed groups of type $\mathsf B/\mathsf C$ where both components are non-split, they do not arise via base change from a visible field.
\item The cases $\mathsf X_n=\mathsf B_n^{\mathrm{ad/sc}}$ and $\mathsf Y_n=\mathsf C_n^{\mathrm{sc}}$ correspond to a classical construction, see \cite[(7.1.6)]{PRG}. One starts  from a defective non-degenerate quadratic form on a $2n+1$-dimensional space $V$. The groups of automorphisms of the quadratic form preserves the $1$-dimensional radical $V^\perp$; therefore the automorphism group $\mathrm{Aut}(V,q)=\mathsf{SO}(q)$ acts on the space $V/V^\perp$ endowed with non-degenerate alternating form $\overline B(\overline x,\overline y)=q(x+y)-q(x)-q(y)$. This provides a morphism $\mathsf{SO}(q)\to\mathsf{Sp}(\overline B_q)$, which is a morphism between groups of types $\mathsf B_n^{\mathrm{ad}}\to\mathsf C_n^{\mathrm{sc}}$. 
\item \label{bullet:split} \emph{However}, note that the resulting symplectic group is always a split group, since every non-degenerate alternating form can be reduced to the standard form $\sum_i x_ix_i'$. Since \cite{PRG} claims that the isogeny $\overline\pi : \mathsf C_n \to \mathrm{Fr}^*\mathsf B_n$ is unique, it should be the case that $\Fr^* \mathsf{SO}_{2n+1}(q)$ is always a split group. This is indeed the case, since
\[ x_0^2 + \sum_{i=1}^n (a_ix_{i}^2 + x_ix_i' +a_i x_i'^2) \simeq_{k^{1/2}} \big(x_0+\sum_{i=1}^n (\sqrt a_i x_i + \sqrt{a_i'} x_i')\big)^2  + \sum_{i=1}^n  x_ix_i'.  \]
\item The construction for $\mathsf G_2$ requires more specialized knowledge. The issue is that in characteristic $3$ the adjoint representation of dimension $14$ is not irreducible, but contains the standard representation of dimension $7$; this provides an ideal in the Lie algebra. 

To see how this happens, we recall that any group of type $\mathsf G_2$ can be realized as automorphism group $\mathrm{Aut}(\mathbb O)$ of an octonion algebra $\mathbb O$ endowed with a binary product, a norm $q$ and unit $1$. The corresponding Lie algebra arises as derivations of the octonion algebra: $\mathfrak g_2=\mathrm{Der}(\mathbb O)$. An octonion algebra is is \emph{alternative}, this means that the associator $(x,y,z)=(xy)z-x(yz)$ is an alternating trilinear map. Therefore if we define a new product $[xy]$ on $\mathbb O$ by setting $[xy]=xy-yx$, then clearly $[xx]=0$ and moreover we have the identity
\begin{align*}
  [[x_1x_2]x_3]+[[x_2x_3]x_1]+[[x_3x_1]x_2] &= \sum_{\sigma\in\mathrm{Sym}(3)} (x_{\sigma(1)},x_{\sigma(2)},x_{\sigma(3)}) \\ &=  6(x_1,x_2,x_3). 
 \end{align*}
Therefore in characteristic $3$, the Jacobi identity holds and we have endowed $\mathbb O$ of dimension $8$ with the structure of a Lie algebra which we denote with $\mathfrak O$. Of course $I = \langle 1\rangle$ is an ideal in $\mathfrak O$ and the quotient $\mathfrak V = \mathfrak O/I$ is a $7$-dimensional Lie-algebra. (It is actually the Lie algebra $\mathfrak{psl}_3$.) The adjoint representation of this Lie algebra provides a map
\[  \mathfrak V \to \mathrm{Der}(\mathfrak V),  \]
The proof is completed by showing that $\mathrm{Der}(\mathfrak V)=\mathrm{Der}(\mathbb O)$. Of course there is a map $\mathrm{Der}(\mathbb O)\to\mathrm{Der}(\mathfrak V)$. A map in the other direction can be constructed by exploiting the orthogonal decomposition of quadratic spaces $\mathbb O = \langle 1\rangle  \perp 1^\perp$ and defining a product $\times$  on $1^\perp$ by $a\times b = ab - q(ab,1)/q(1,1) 1$, i.e. by projecting the octonion product back to $1^\perp$. Since every derivation of $\mathbb O$ sends $1$ to $0$, this construction allows one to extend a derivation on $(1^\perp,\times)$ to a derivation on $\mathbb O$. The final observation is then that $(1^\perp,\times)$ is (up to a multiple) equal to $(\mathfrak V,[.,.])$. 
\item We suspect that there exists a parallel construction for $\mathsf F_4$ along the following lines. The issue is again that that the standard representation of dimension $26$ is contained in the adjoint representation of dimension $52$, providing an ideal in the Lie algebra.

Any group of type $\mathsf F_4$ can be realized as automorphism group $\mathrm{Aut}(\mathbb A)$ of an Albert algebra $\mathbb A$ of dimension $27$ endowed with a quadratic operator $U$, a norm $q$ and unit $1$. An Albert algebra is a quadratic Jordan algebra, this means that if $p=2$ it also has the structure of a restricted Lie algebra by defining a Lie bracket and $p$-operation as follows:
\[  [xy] = (U_{x+y}-U_x-U_y)1  \text{ and } x^{[2]} = U_x1.  \]
For a proof, see \cite[p.30, \S1.4]{jacobson-qja}; the the important equation is QJ20 on page 24 which states essentially that 
\[ [x^{[2]}y]=[x[xy]]+2U_xy,  \]
and thus if $p=2$ we have $\mathrm{Ad}(x^{[2]})=\mathrm{Ad}(x)^2$ from which the Jacobi identity follows by linearization. From here on, one can concoct an argument which parallels the argument we sketched here for $\mathsf G_2$.
\end{enumerate}	
\end{remark}

\subsection{Twisted groups: the Suzuki--Ree groups} \label{sec:twisted-groups}
\begin{theorem} \label{theorem:twisted-groups} Let ${}^2\mathsf X_n(k,\theta)$ be a twisted Chevalley group, as defined by \cite[Ch.~13 \S4]{carter}. Then there is a twisted group scheme $\tilde X$ and a blended field $b$ such that $ {}^2\mathsf X_n(k,\theta) \cong \tilde X(b)$.
\end{theorem}
\begin{proof} It is well known that the graph automorphism of the Chevalley group $\mathsf X_n(\mathbf F_p) \to \mathsf X_n(\mathbf F_p)$ can be used to define an endomorphism of algebraic $\mathbf F_p$-groups $g:\mathsf X_n\to\mathsf X_n$. If $\theta \in\mathrm{Aut}(k)$ is chosen so that $\fr_k \circ \theta^2= \mathrm{id}_k$ then the twisted group is defined as the group of fixed points of $g\theta$ acting on $\mathsf X_n(k)$. On the other hand, we get a twisted group $\tilde X = (\mathsf X_n,g)$ over the twisted field  $(\mathbf F_p,\mathrm{id})$ and by Proposition~\ref{prop:fixedpoints} we have
\[  \tilde X(b) \cong \big(\mathsf X_n(k)\big)^{g\theta} = {}^2\mathsf X_n(k,\theta). \qedhere \]
\end{proof}

\begin{remark} \label{remark:perfect-fields} The twisted groups in \cite{carter} are only defined over perfect fields and therefore the previous theorem is only meaningful in that case. The restriction to perfect fields is not a shortcoming of the theory; rather it is a shortcoming of the construction of twisted abstract groups as fixed points of an involution. Close examination of Tits' alternative definition \cite[p. 66, ``{$\alpha_\pi(g)=\alpha_\sigma(g)$}'']{suzukietree}---which also works over non-perfect fields---shows that he \emph{defines} the twisted groups by the equation $x\circ\Phi = g\circ x$, i.e.~as groups of rational points of a twisted group scheme. So there is simply nothing to prove in that case and it is completely trivial that they are also groups of rational points of twisted group schemes. So in some sense our proof simply comes down to showing that Tits' non-perfect definition indeed generalizes the definition \cite{carter} more commonly accepted as standard.
\end{remark}

\subsection{Mixed groups and buildings of mixed type} \label{sec:mixed-groups}

\begin{theorem} \label{theorem:mixed-groups} Let $\mathsf X_n(k,\ell)$ be a split mixed group as defined by \cite[(10.3.2)]{tits74}. Then there exists a mixed group scheme $\tilde X$ and a mixed field $m$ such that $\mathsf X_n(k,\ell) \cong \tilde X(m)$.
\end{theorem}
\begin{proof}
 \textbf{Step 0.} Let us review the construction of the \emph{split mixed groups} from \cite[(10.3.2)]{tits74} in the notation from \emph{loc.~cit.}.  One starts with the adjoint $k$-split simple algebraic group $\mathsf X_n$ defined over the field $k$. If we choose a maximal $k$-split torus $T$, then there is a corresponding root system $\Phi = \Phi_> \cup \Phi_<$ consisting of a set of short roots and long roots. For each root $r\in\Phi$, there is an $k$-unipotent subgroup $U_r$ upon which $T$ acts, one for each $r\in\Phi$ together with an isomorphism $u_r : \mathbf G_a\to U_r$. Tits then defines the mixed group by providing the following set of generators: 
   \[  \mathsf X_n(k,\ell) = \langle T(k,\ell) \cup \{U_r(k) \mid r \in \Phi_>\} \cup \{U_r(\ell) \mid r \in\Phi_<\} \rangle \subseteq \mathsf X_n(\ell),   \]
  where we have
  \[  T(k,\ell) = \{ t\in T(\ell) \mid r(t)\in k  \text{ if } r \in\Phi_> \} \subseteq T(\ell). \]
  We know from \cite[(7.1.1)]{PRG} that if we expres a long root as a linear combination of fundamental roots, the coefficients of the short roots are all divisible by $p$. Therefore we can also define
    \[  T(k,\ell) = \{ t\in T(\ell) \mid r(t)\in k  \text{ if } r \in\Delta_> \} \subseteq T(\ell), \]
    where $\Delta$ is a system of fundamental roots and $\Delta_> = \Delta\cap \Phi_>$. In rough terms: $X(k,\ell)$ arises from $X(\ell)$ by restricting the long roots to the smaller field, both for the root subgroups and for the torus.
    
  \textbf{Step 1.} We will show that $\mathsf X_n(k,\ell)$ can be constructed as follows.
  Let $\mathsf Y_n$ denote the dual group, i.e. $\mathsf Y_n=\mathsf  B_n^{\mathrm{ad}}$ if $\mathsf X_n=\mathsf C_n^{\mathrm{ad}}$ etc and $\pi:\mathsf X_n\to \mathsf Y_n$ the corresponding very special isogeny between adjoint groups from Proposition~\ref{prop:veryspecialisogeny2} and let us also assume a choice of maximal torus, roots and fundamental roots in both groups, denoted by $T$, $\Phi$, $\Delta$ and $\bar T$, $\bar\Phi$, $\bar\Delta$ with a bijection $\Phi\to\bar\Phi : r\mapsto \bar r$ as constructed in \cite[(7.1.5)]{PRG}. Then there are maps
  \[
  \begin{tikzcd}
  \mathsf X_n(\ell) \ar[dr,swap,"f=\pi(\ell)"] & & \mathsf Y_n(k) \ar[dl,tail,"\mathsf Y_n(\mathrm{inc})"] \\ & \mathsf Y_n(\ell) &
  \end{tikzcd}
  \]
  We claim that $\mathsf X_n(k,\ell) = f^{-1}(\mathsf{Y_n}(k))$, in other words, let $x\in\mathsf X_n(\ell)$ be arbitrary, then we claim that
  \begin{equation} \label{eq:titsdefinition}
	  x\in\mathsf X_n(k,\ell) \iff f(x)\in\mathsf Y_n(k).
	\end{equation}
  \textbf{Step 1a.}  Let us first show that 
  \[  t \in T(k,\ell) \iff f(t) \in \bar T(k) \]
  Since the group is adjoint, we can rely on the isomorphism 
  \[ T(\ell) \to  \prod_{r\in \Delta} \mathsf{GL}_1(\ell) : t\mapsto (c(t))_{r\in \Delta} \]
  and a similar isomorphism for $\bar T$. The very special isogeny $f=\pi(\ell)$ induces the Frobenius on the $\mathsf{GL}_1$ corresponding to short fundamental roots and the identity on the long roots. This means that an element $(r(t))_{r\in\Delta}$ is sent to $(r(t)^{s_r})_{\bar r\in\bar \Delta}$, where $s_r=p$ if $r$ is short and $s_c=1$ if $r$ is long. Therefore the value on the long roots in $\bar\Delta$ automatically ends up in $\ell^p\subseteq k$. Therefore the condition $f(t)\in \overline T(k)$ states that $r(t)$ must be contained in $k$ whenever $\bar r$ is short and thus $r$ is long, i.e.~$t\in T(k,\ell)$.

  \textbf{Step 1b.} We can now show the equivalence \eqref{eq:titsdefinition}. One implication is clear, since $f(\mathsf X_n(k,\ell))\subseteq \mathsf Y_n(k)$ as is easily seen by evaluating $f$ on each of the generators. For the other implication, we consider an arbitrary $x\in\mathsf X_n(\ell)$. The normal form (see \cite{humphreys}) of $x$ is of the following form:
  \[ x = \prod_{r\in\Phi}  u_r(x_r) \cdot n(\sigma) t \cdot \prod_{r\in\Phi_\sigma} u_r(y_r) \]
  We can use \cite[(7.1.5)]{PRG} to compute $f(x)$:
  \[ f(x) = \prod_{\bar r\in\bar\Phi} \bar u_{\bar r}(x_{\bar r}^{s_r}) \cdot n(\bar\sigma) f(t) \cdot \prod_{\bar r\in\bar{\Phi_\sigma}} \bar u_{\bar r}(y_{\bar r}^{s_r}), \]
  where $s_r$ is the same number from earlier.                                                                                                                                                      Expressing that $f(x)\in\mathsf Y_n(k)$, relying in particular on the uniqueness of this normal form and the observation that $\overline{\Phi_\sigma}=\bar\Phi_{\bar\sigma}$, we get the conditions 
  \[ x_c^{s_c}, y_c^{s_c} \in k \text{ and } f(t) \in \bar T(k), \]
  where $\bar T$ is a maximal torus in $\mathsf Y_n$. 
  Recalling that $\ell^p\subseteq k$ and by relying on step 1a for the condition on $f(t)$, we conclude that that $x$ is indeed generated by Tits' generators.

  \textbf{Step 2.} We now use the same data, i.e.~the isogeny $\pi:\mathsf X_n\to \mathsf Y_n$ and the field extension $\ell/k$ together with Corollary~\ref{cor:relfac} to obtain a mixed group scheme $\widetilde G=((\mathsf Y_{n})_k,(\mathsf X_n)_\ell)$ over the mixed field $m=(k,\ell)$. Since a field is certainly reduced as a scheme, we may apply Proposition~\ref{prop:mixedpoints} to compute its set of rational points:
  \[ \widetilde G(m) = f^{-1}(\mathsf Y_n(k)).  \]
  Here $\mathsf Y_n(k)$ is considered a subset of $\mathsf Y_n(\ell)$ in the natural way and $f:\mathsf X_n(\ell)\to\mathsf Y_n(\ell)$ is the map induced by $\pi$. In Step 1 we have proven that
  \[  f^{-1}(\mathsf Y_n(k)) = \mathsf X_n(k,\ell),  \]
  and therefore Tits' mixed group is indeed realized as rational points of the mixed group scheme $\tilde G$.
 \end{proof}

\begin{remark} \label{rem:parabolics} The following remarks are some expectations that we have, but that will require some future work to verify.
\begin{enumerate}
\item Continuing the notation from the proof, we let $\Phi=\Phi_>\cup\Phi_<$ be a root system of type $\mathsf X_n$ with fundamental system $\Delta = \Delta_>\cup\Delta_<$.  Note that $|\Delta|=n$; we also denote $r=|\Phi|$, $n_> = |\Delta_{>}|$ etc. Since the dimension of the corresponding algebraic group of type $\mathsf X_n$ is $r+n$, this is also the dimension of eah of the components of the mixed group scheme $\tilde X$ over $(k,\ell)$. Nonetheless, the \emph{partial dimensions} are given by $(r_{>}+n_{>},r_{<}+n_{<})$. 
\item  Our original motivation for studying mixed schemes---and not just rings or affine schemes---was that the homogeneous spaces of a mixed group will be mixed projective varieties. More specifically we can start from a the mixed group $(G,G')$ and take a pair of Borel subgroups $B\subset G$, $B'\subset G'$ in such a way that the mixing maps send $B$ into $B'$ and conversely $B'$ into $B$. This allows one to construct a mixed variety $\begin{tikzcd} G/B \rar[shift left=0.5ex] & G'/B' \lar[shift left=0.5ex] \end{tikzcd}$ where both components have dimension dimension $r/2$ (since $\dim B = n+r/2$) but partial dimensions $(r_{>}/2,r_</2)$. Taking rational points of this variety over $(k,\ell)$, we obtain a set which can be identified with the flag complex of the mixed building.
\item Moreover, let us consider a set $\Gamma\subseteq \Delta$ of fundamental roots and denote $s=|\chi^{-1}(\Gamma)|$, where  $\chi$ is the projection $\chi : \Phi\to \Delta$.  To the choice of $\Gamma$ correspond parabolic subgroups $B\subseteq P\subseteq G$ and $B'\subseteq P'\subseteq G'$ of dimension $n+(r+s)/2$ and corresponding varieties $G/P$ and $G'/P'$ of dimension $(r-s)/2$. These subgroups give rise to a variety $\begin{tikzcd} G/P \rar[shift left=0.5ex] & G'/P' \lar[shift left=0.5ex] \end{tikzcd}$ of partial dimensions $(\frac{r_>-s_>}{2},\frac{r_<-s_<}{2})$. Typically the set $\Gamma$ is chosen as large as possible to obtain structures of reasonable dimension that are studied by algebraic or incidence geometers.
\item In particular, if $S$ contains all the long (or short) fundamental roots, one of the partial dimension will be $0$. For instance, we could start from a root system of type $\mathsf B_n$ and take the subset $\Gamma\subset\Delta$ consisting of all $n-1$ long fundamental roots. If we note that there are precisely $2n$ short roots, we expects to find a mixed variety of partial dimension $(0,n)$: the mixed quadric. It is very easy to describe this thing---or at least an affine part of it---more explicitly over a mixed field $(k,\ell,\kappa,\lambda)$: we define the $k$-algebra $Q$ and $\ell$-algebra $W$ by
\begin{align*}  Q &=  k[x_0,x_1,\dots,x_n] / (x_0^2 - q(x_1,\dots,x_n)), \\
  W &= \ell[y_1,\dots,y_n],
\end{align*}
where $q$ is a non-degenerate quadratic form, e.g.~the hyperbolic one. We then define maps between both algebras extending $\kappa$ and $\lambda$ as follows:
\begin{align*}
   \hat\kappa & : Q\to W : \begin{cases} x_i\mapsto y_i^2 & i\geq 1 \\ x_0\mapsto q^\kappa(y_1,\dots,y_n) & \end{cases} \\
   \hat\lambda &: W\to Q : y_i\mapsto x_i,
 \end{align*}
 where $q^\kappa$ simply applies $\kappa$ to the coefficients of $q$. It is easily verified that this defines a mixed ring $(Q,W,\hat\kappa,\hat\lambda)$ and therefore a mixed affine variety. Since $\hat\kappa$ sends all variables to squares, the differential vanishes and the mixed object has indeed partial dimensions $(0,n)$.
 
\end{enumerate} 
\end{remark}

\subsection{Exotic pseudo-reductive groups} \label{sec:gps-prg}
In \cite{PRG} the authors provide a structure theory for pseudo-reductive groups over a base field $k$. In rough terms the outcome is that almost every pseudo-reductive group arises from a \emph{standard construction}, which has as its starting point a Weil restriction $\mathsf R_{k'/k}(G)$ of a reductive $k' $-group $G$ through a purely inseparable field extension $k'/k$. Some exotic examples, introduced in Chapter~7 of \emph{op.\@cit.}~do not follow this pattern and require a more elaborate construction. Our next theorem states that these exotic groups,  actually \emph{do} fit this pattern, albeit within the category of mixed group schemes. In fact the so called basic exotic groups can be thought of as $\mathscr G = \mathsf R_{m/k}(\widetilde G)$, where $\widetilde G$ is a reductive mixed group over the mixed field $m$ and $\mathsf R_{m/k}$ denotes the mixed version of the Weil restriction. An arbitrary exotic group arises by further Weil restriction of a basic exotic group; of course this is also a Weil restriction.

\begin{theorem} \label{theorem:pseudo-reductive-groups} Let $\mathscr G$ be an exotic pseudo-reductive group as defined by \cite[(8.2.2)]{PRG}. Then there is a mixed group scheme $\widetilde G$ over a mixed algebra $M$ such that $\mathscr G = \comp_2\mathsf R_{M/k}(\widetilde G)$.
\end{theorem}

\begin{proof} Let us review the construction of the \emph{basic exotic groups} from \cite[(7.2.3)]{PRG} in the notation from \emph{loc.~cit.}: one starts with a very special isogeny of $k$-groups $\pi : G \to \overline G$ and a purely inseparable field extension $\ell/k$ of finite degree such that $\ell^p\subset k$. Base changing $\pi$ to $\ell$ followed by Weil restriction back to $k$ gives a map
\[  f  : \mathsf R_{\ell/k}G_\ell\to \mathsf R_{\ell/k}{\overline G}_\ell.  \]
On the other hand, the unit of the adjuction between base change and Weil restriction provides a map $\overline G \to \mathsf R_{\ell/k}\overline G_\ell$.
Then one defines $\mathscr G = f^{-1}(\overline G)$. Changing the notation to our own, we denote $\beta:\Spec \ell\to\Spec k$ so that $f=\beta_*\beta^*\pi$ and
\[ \mathscr G = \overline G \underset{\beta_*\beta^* \overline G}{\times} \beta_*\beta^* G, \]
Now, let us start from the same very special isogeny and field extension $\ell/k$, and construct the mixed group $\widetilde G=(\overline G,\beta^* G,.,.)$ as in Proposition~\ref{cor:relfac}. Applying the formula from Proposition~\ref{prop:weilrestriction} we find
\[ \mathsf R_{m/k}(\widetilde G) = h_*(\widetilde G) = (\overline G, \overline G\underset{\beta_*\beta^*\overline G}{\times} \beta_*\beta^*G, ., .)  \]
where $h : \Spec (k,\ell) \to \Spec k$ is the corresponding extension of mixed fields.
Note that the second component is indeed $\mathscr G$.

Taking the product of such field extensions, we see that every $K$-group $\mathscr G$  for a non-zero finite reduced $k$-algebra $K$ whose fibers are basic exotic pseudo-reductive groups can be realised as $\comp_2\mathsf R_{M/K}(\tilde G)$. Applying the Weil restriction $\mathsf R_{K/k}$ to this, we find that the exotic group $G=\mathsf R_{K/k}(\mathscr G)$ as in \cite[(8.2.2)]{PRG} can be realized as $G=\comp_2\mathsf R_{M/k}(\mathscr G)$.
\end{proof}

\begin{remark}  \label{rem:mixed-groups}
\begin{enumerate}
\item \label{bullet:c2appears} Although the catch prase in our abstract states that exotic pseudo-reductive groups are Weil restrictions of mixed groups, we now see  that this is not entirely correct. They arise from a \emph{Weil-mixtor} restriction: first a Weil restriction $\mathsf R_{M/k}$ and then a mixtor restriction $\comp_2$---see Remark~\ref{remark:mixtor}. 
\item \label{bullet:dimensions} The simension of a semi-simple group is completely determined by combinatorial data coming from a root system. If we apply an inseparable Weil restriction, we find a pseudo-reductive group with dimension determined by combinatorial data of the original group, and the degree of the restriction morphism. For the exotic pseudo-reductive groups on the other hand, there is a formula (7.2.1) in \cite{PRG} which states
\[  \dim \mathscr G = (r_>+n_>)+(r_<+n_<)[k':k].  \]
Since the corresponding mixed group has \emph{partial dimensions} (see Remark~\ref{rem:parabolics})
\[ (r_>+n_>, r_<+n_<)  \]
and a Weil restriction proceeds along an extension of mixed fields with degrees $(1,[k':k])$ we have this same separation of the dimension into combinatorial data of the original group and degree of a morphism.
\item If $[k':k]$ is infinite, the mixed group still exists: all the infiniteness is gobbled up by the mixing maps but the structure morphism is still of finite type. But in that case taking a Weil restriction becomes problematic.
\item \label{bullet:mixed-adjectives} Initially we had hoped to also state and prove a theorem which states that Weil restrictions of mixed reductive groups are always pseudo-reductive. It seems conceivable that the proof of the corresponding statement from \cite[{}1.1.10]{PRG} will generalize to the mixed setting, but only after a thorough study of mixed versions of the typical adjectives from algebraic geometry such as smooth and connected. For instance, one must first study a mixed notion of smoothness and prove a mixed version of the infinitesimal criterium. Perhaps this will eventually allow to circumvent some of the difficulties encountered by Conrad--Gabber--Prasad by starting from an arbitrary pseudo-reductive group, constructing its parental mixed reductive group directly and deducing the structure theory of pseudo-reductive groups from there, similar to how the current structure theory works away from characteristic $2$ and $3$.
\item \label{bullet:esotheric} The exotic groups are not the only strange cases which appear in the theory of Conrad, Gabber and Prasad: in characteristic $2$, there are other esotheric constructions. Nonetheless, these constructions are also related to mixed buildings, and admit a corresponding twisted group, a \emph{generalized Suzuki group}. (See \S\ref{sec:history} for some details.) This gives us some hope that these groups too fit into our framework. We see two observations which could be relevant. The first observation is that mixed algebraic groups related to regular but defective quadratic forms of defect $\geq 2$ are likely of interest \cite[\S1.6]{carter}. The second observation is that there could be mixed reductive groups over an invisible field which do not arise via base change from a mixed reductive group over a visible field, as we explained in Remark~\ref{remark:isogenies}. Such groups could not be constructed as exotic groups because the the presence of non $k$-linear maps cannot be circumvented so easily; this could explain why some of the constructions in the later chapters of \cite{PRG} are so indirect.
\end{enumerate}
\end{remark}

\section{Historical notes and references} \label{sec:history}
We will now provide a historical overview of people and facts related to mixed and twisted groups, with an attempt at being complete. Because the story is so heavily intertwined with the mathematical works of Jacques Tits, we will also try to tell his story, albeit with no attempt at being complete.

\subsection{1950--1960}
In 1955, Claude Chevalley published his \emph{Tohoku paper} where he showed that each of the known semi-simple Lie-algebras, known from the Killing-Cartan classification, gives rise to a class of groups, which can be defined over an arbitrary field: the \emph{Chevalley groups}. The next few years Chevalley lead the \emph{S\'eminaire Cartan-Chevalley} in 1955--56 and the \emph{S\'eminaire Chevalley} in 1956--58, where the foundations for modern scheme theory and algebraic group theory were laid.

In 1957, Rimhak Ree \cite{Ree57} showed that the Chevalley groups of types $\mathsf A,\mathsf B,\mathsf C,\mathsf D$ correspond to the classical linear, symplectic and some of the orthogonal groups over the corresponding fields as one would expect; he also showed that the groups of type $\mathsf G_2$ are isomorphic to groups defined much earlier by Dickson \cite{dickson1901,dickson1905}.

This raised the question whether the unitary groups could be seen as variations on this theme. In 1959, Robert Steinberg \cite{Steinberg59} publishes an article with that title presenting a new construction: he observed that in the cases $\mathsf A_n$, $\mathsf D_n$ and $\mathsf E_6$, the Dynkin diagram has an automorphism which gives rise to an automorphism of the Chevalley group. In combination with an automorphism of the underlying field this gives interesting involutory automorphisms whose fixed points are classes of simple groups now known as \emph{Steinberg groups} ${}^2\mathsf A_n$, ${}^2\mathsf D_n$, ${}^3\mathsf D_4$ and ${}^2\mathsf E_6$, where the former two provided the unitary groups and the missing orthogonal groups, and latter two were new.

Quite a different perspective was provided by Jacques Tits, who was looking into generalizations of the \emph{fundamental theorem of projective geometry}. This theorem says that every permutation of the points of projective space which is incidence preserving---i.e.~sends lines to lines---is induced by a semi-linear map on the underlying vector space. In other words: combinatorial axioms for a projective space characterize the Lie groups of type $\mathsf A_n$. The program that Tits had begun persuing in the early 50s was to provide axiomatic systems of points and lines which characterize the other Lie groups in a similar manner. In 1953 \cite{tits53}, in some of his earliest work, he investigated the real octonion plane, polarities of this plane, and the related Lie groups $\mathsf F_{4(-52)}$, $\mathsf F_{4(-20)}$, $\mathsf E_{6(-26)}$ and $\mathsf E_{7(-25)}$. A skeptical reviewer, named Chevalley, spotted a mistake in a proof which Tits fixed in a follow-up article \cite{tits54} in 1954 where also the announced construction of $\mathsf E_{7(-25)}$ was given. (The latter groups makes a come-back appearance in \cite{tits74} under the guise of the \emph{non-embeddable polar spaces}.) In 1955 \cite{tits55} and 1956 \cite{tits56} Tits published two long studies about homogeneous spaces of Lie groups, which can be interpreted as a precursor to his theory of buildings. One of his early achievements was a description of the (split) group $\mathsf E_6$ for the S\'eminaire Bourbaki \cite{tits58} as the automorphism group of some sort of `plane'---a parapolar space in later terminology. Investigating polarities in this $\mathsf E_6$-`plane' lead him to the independent discovery of the groups of type ${}^2\mathsf E_6$ over the reals.

\subsection{1960--1970}
In 1960, Michio Suzuki was investigating a class of groups named \emph{Zassenhaus groups}. A Zassenhaus group is a permutation group which (i) acts doubly transitively, (ii) with only the identity fixing $3$ elements and (iii) without a regular normal subgroup--- the latter case only excludes some degenerate cases such as the Frobenius groups.  Suzuki noted that a Zassenhaus group of odd degree is simple so he was more than interested in classifying them.

According to the reviewer of \cite{SuzukiB2}, it had been conjectured by Feit and Hasse that the only examples were the groups $\mathsf{SL}(2,2^n)$, but in \emph{loc.~cit.} Suzuki reported the discovery of a new class of Zassenhaus groups of odd degree. He constructed the groups $G(q)=\mathsf{Sz}(q)$ as subgroups of $\mathsf{GL}_4(q)$ generated by certain matrices, where $q$ is an odd power of $2$. About these groups, he writes: \emph{The series of groups $G(q)$ gives, therefore, the second infinite series of simple groups which are not of Lie type.}  Nonetheless, further in the article he also notes that his generators leave a bilinear form invariant, so they are also subgroups of $\mathsf{Sp}_4(q)=\mathsf B_2(q)$. (In 1962 \cite{Suzuki62} Suzuki classified the odd degree Zassenhaus groups and show that the Suzuki groups completed the list, and in 1964 he also classified the even degree Zassenhaus groups.)

Later that year, Ree realized that Suzuki's groups \emph{are} in fact closely related to the Chevalley groups of type $\mathsf B_2=\mathsf C_2$. Over a field $k$ admitting an automorphism $\theta$ such that $\theta(\theta(x))=x^2$, he could construct a certain involutory automorphism of $\mathsf{Sp}_4(k)$ such that the fixed subgroup is precisely $\mathsf{Sz}(q)$. Repeating the procedure for the Chevalley groups of types $\mathsf F_4$ and $\mathsf G_2$ he constructed what are now known as the \emph{large Ree groups} ${}^2\mathsf{F}_4$ \cite{ReeF4report,ReeF4} (for $p=2$) and \emph{small Ree groups} ${}^2\mathsf G_2$ \cite{ReeG2report,ReeG2} (for $p=3$).

By 1961 Tits too has turned his attention to algebraic groups and he reports on a geometric approach to the simple groups of Suzuki and Ree for the S\'eminaire Bourbaki \cite{suzukietree}. A thorough treatment of the Suzuki groups was later also published in \cite{tits62}. Tits work was an interesting variation on his earlier work on polarities: the `polarity' of a plane with itself had to be replaced with what he calls \emph{une sorte de dualit\'e} between two different varieties, embedded in $\mathbb P^3$ and $\mathbb P^5$ in the case of ${}^2\mathsf B_2$ and embedded in $\mathbb P^6$ and $\mathbb P^{13}$ in the case of ${}^2\mathsf G_2$. These varieties are actually the homogeneous spaces $G/P_1$ and $G/P_2$, where $G$ is the algebraic group of the corresponding type and $P_1$ and $P_2$ are the two classes of maximal parabolic subgroups. In later terminology the rational points of these varieties can be identified with the points and lines of a generalized quadrangle or hexagon, where the incidence relation can be read off from the flag variety $G/B \onto G/P_i$.

Tits must have felt that something remarkable was going on: his geometric construction provides maps whose composition which compose to the the Frobenius, rather than to the identity. For a perfect field, he could think of his construction as a polarity of the geometry with points and lines given by rational points of $(G/P_1,G/P_2)$, but since Tits had observed that he could also make the construction of the Suzuki and Ree groups over imperfect fields, he chose to phrase it rather carefully as \emph{some sort of duality}.

In the terminology that we introduced in this work and with the benefit of hindsight, we could say that Tits was looking at a mixed projective variety
\[
\begin{tikzcd}
G/P_2 \dar\rar & H/Q_1 \rar\dar &  G/P_2 \dar \\
\Spec k \rar & \Spec \ell \rar & \Spec k
\end{tikzcd}
\]
corresponding to a mixed group $(G,H)$ over a mixed field $(k,\ell)$, with $Q_1$ and $P_2$ maximal parabolic subgroups. To this mixed variety we could in principle associate an  incidence structure with points $G/P_2$ and lines $H/Q_1$ defined \emph{over different fields} but the situation was quite unclear at the time because $k \cong \ell$ and $H\cong G$. (This is a bit imprecise, see Remark~\ref{rem:parabolics} for some more details.)

Over the next few years, Tits drastically picks up the pace and makes many important contributions to group theory and geometry. At first this was often over perfect fields only, but after Alexander Grothendieck proved his deep theorem on the existence of maximal tori over arbitrary fields \cite[exp.~XIV]{SGA3} in 1964, this restriction could be lifted. Let us just mention some of these developments: in \cite{boreltits} Borel and Tits provided their structure theory for reductive groups; in \cite{tits62c,tits64} Tits initiated the theory of groups with a $BN$-pair; in \cite{tits66} he provided a structure theory for semi-simple groups in terms of their Tits index and anisotropic kernel; in collaboration with Francois Bruhat, he investigated the structure of algebraic groups over local fields. Meanwhile, he worked on lecture notes on the theory of buildings, which must have circulated by the end of the 60's but weren't formally published until 1974 \cite{tits74}.

The first time a mixed group makes an explicit appearance in the literature is, as far as we can tell, in Steinbergs Yale lecture notes from 1967--68 \cite{steinberg} on groups of Lie type, in the following remark on page 153:

\begin{quotation}
If $k$ is not perfect and $\phi : G\to G$ then $\phi G$ is the subgroup of $G$ in which $X_\alpha$ is parametrized by $k$ if $\alpha$ is long and $k^p$ if $\alpha$ is short. Here $k^p$ can be replaced by any field between $k^p$ and $k$ to yield a rather weird simple group.
\end{quotation}

It is probably not a coincidence that Tits was also in Yale around that the time.  In fact, Hendrik Van Maldeghem has suggested to us that the first time a mixed group (or variety) was observed in the wild may have been in Tits' unpublished classification of Moufang hexagons. We could not date this classification precisely but given that Tits introduced generalized polygons as early as 1959 and that they were probably the first class of buildings he seriously investigated, it seems very plausible that it were indeed the hexagons which lead Tits to these groups for the first time.

One year earlier, in his own Yale 1966--67 lecture notes on algebraic groups, Tits had investigated unipotent groups in positive characteristic. The lecture notes were never formally published until they appeared as appendix B1 of his collected works in 2014, although the results on unipotent groups had appeared earlier in the works of Oesterl\'e and, in a revised form, in appendix B of \cite{PRG}. Although we found no written evidence of this hypothesis, we believe it is likely that Tits thought that a thorough study of algebraic groups in positive characteristic could lead to a more satisfying explanation for why mixed groups are required to complete his classification of buildings.

\subsection{1970--1990}
The 70's and early 80's were the golden years for the classification of finite simple groups. While the mixed groups and mixed buildings began gathering dust, the twisted groups, at least the finite ones, were an important part of the classification and as such well known and studied by group theorists. In particular we should mention that the characterization of the (small) Ree groups proved to be one of the hardest steps in the classification: it cost John G.~Thompson three difficult papers \cite{thompson1,thompson2,thompson3} in 1967, 1972 and 1977 to reduce it to a number-theoretic problem which was solved by Enrico Bombieri \cite{bombieri} in 1981 in an dazzling application of elimination theory; the reviewer remarked that \emph{ordinary mortals such as the present reviewer are overawed by the author's tour de force}.

Also the representation theory of these twisted groups was studied thoroughly. In fact, we noted a strange occurence of the mixing functor $\mix$, introduced in \ref{sec:functors} in the Pierre Deligne and George Lusztig's work \cite[\S11]{delignelusztig} on representations for finite groups of Lie type where remark that their approach to the Suzuki and Ree groups works equally well for ``groups of the form $G=G_1\times G_1$ with $F'(x,y)=(F(y),x)$'' but it is unclear to us what the wider significance here is.

During that period, from the early 70s to the late 80s, with many researchers focusing on finite groups, and Tits himself lecturing at the Coll\`ege de France about sporadic groups in 1976--1977, and about the monster group in 1982--1983, 1985--1986 and 1986--1987, it could seem that not much was happening in the theory of buildings and algebraic groups. But at the same time Tits was actually working on the classification of Moufang polygons, on affine buildings, on Kac-Moody groups and algebras and twin buildings. We will not go into all these developments but focus on the Moufang polygons, since these are most relevant to our story.

As we mentioned earlier, we suspect that Tits completed the classification of Moufang triangles, hexagons and octagons quite early, perhaps in the early 70s. By 1974, he finally published his lecture notes, where he classified (spherical) buildings in rank $\geq 3$. In this classification other mixed buildings pop up, namely those related to groups of type $\mathsf F_4$ and those to groups of type $\mathsf B_n$ and $\mathsf C_n$, $n\geq 3$. These groups---together with the $\mathsf G_2$-variant---are precisely the groups for which we show in our Theorem~\ref{theorem:mixed-groups} that they arise as groups of rational points of a mixed group scheme. We note however, that the $\mathsf B/\mathsf C$-class admits further generalization to groups which are defined over a pair of fields $K,L$ and an additional $K$-vectorspace contained in $L$. We have not yet managed to relate these groups to our own work; the only insight that we have to offer here is that they are probably related to defective quadratic forms (see Remark~\ref{rem:mixed-groups}.\ref{bullet:esotheric}).

So by 1974 all Moufang buildings of rank $\geq 2$ were classified except for the the Moufang quadrangles. In Van Maldeghem's book on generalized polygons \cite[\S3.4.2]{boekhvm} we find a reference to a preprint from 1976 on this subject and it seems that after this Tits did not touch the subject in the next 20 years. One interesting feature is that since $\mathsf B_2\cong\mathsf C_2$, the buildings with a pair of fields and a vector space from the previous paragraph, can be generalized to a `doubly exotic' class of Moufang quadrangles defined over a pair of fields together with a pair of vector spaces over each field, contained in the other field. It is noteworthy that there is no analogon for the case of $\mathsf G_2$ and hexagons in characteristic $3$. The class is also notable because there exists a twisted variant which generalizes the Suzuki-groups ${}^2\mathsf B_2(k,\theta)$. (These groups are hinted at in \cite[\S7.6]{boekhvm} and more explicitly studied by Van Maldeghem in 2007  \cite{generalized-suzuki}. They do not appear explicitly in Tits' overview of Moufang sets in his 1999--2000 lecture notes.)\medskip 

Another interesting development from that time was a program, proposed by Francis Buekenhout \cite{buekenhout}, to study and eventually classify sporadic groups by associating certain diagram geometries to them---some sort of generalizations of buildings. To some extent, geometric ideas \emph{do} play an important role in the proof of the classification of finite simple groups, but these recognition theorems can only be applied deep into the proof, after a very difficult group theoretical analysis and case distinction. Even though Buekenhout's program was largely outpaced by the rapid developments in finite group theory, it marked the beginning of research in \emph{pure} incidence geometry, with (algebraic or finite) groups coming in a posteriori or not at all.

With the end of the classification announced by Daniel Gorenstein in 1983---perhaps prematurely so---there was a definitely a renewed interest in all these ideas and with the appearance of textbooks such as \cite{brown-buildings} and \cite{ronan-lectures} the subject also became more accessible to newcomers.

\subsection{1990--2000} \label{sec:his-1990}
Most of Tits' later research interests can only be found in his \emph{R\'esum\'es des Cours au Coll\`ege de France 1973--2000}. Of particular interest are to our story are the 1991--1992 and 1992--1993 courses on algebraic groups in positive characteristic with a focus on inseparable phenomena and pseudo-reductive groups. One could consider it Tits' metastrategy for doing mathematics to collect all the examples and then study their common features and eventually weave them into an elegant theory; as far as pseudo-reductive groups are concerned it seems that---with the benefit of hindsight and relying on \cite{PRG}---Tits was in the process of constructing all the examples but a few crucial constructions were still missing. It's remarkable that although Tits' constructions are very remniscent of the mixed buildings he discovered decades earlier, he never makes the connection explicit. After Tits' lectures the subject would lay dormant again for many years.\medskip

In his 1994--1995 lectures then, he returns to the classification of Moufang polygons. Relying on his own unpublished work, he proposes a strategy to carry out the classification, lists the known types, and conjectures there are no others. To our own surprise, a definition in his 1994--1995 lecture notes \S3 speaks of a pair of fields with maps $\kappa:K\to L$ and $\lambda : L\to K$ such that the compositions are the square operators. It is a subtle change in view point that seems to have gone unnoticed by subsequent authors: although Richard Weiss recollects that Tits expressed a certain fondness of the symmetry between $K$ and $L$ on many occasions, it is an observation which seemed hard to exploit.\medskip

Tits' lectures clearly worked inspiring because by 1996--1997 Weiss actually manages to complete this classification \emph{faisant preuve d'une virtuosit\'e technique remarquable}\footnote{`demonstrating a remarkable technical virtuosity'} as Tits puts it on the first page of his 1999--2000 lecture notes. To Tits' surprise, Weiss had discovered in Februari 1997\footnote{According to Norbert Knarr in his review for \cite{F4quadrangles}} a \emph{new} and highly exotic class of Moufang quadrangles. Weiss recollects that at first, Tits was somewhat sceptical about the discovery, but became very enthousiastic about it later on; in fact he decided to lecture about it at the Coll\`ege later that year. After the end of the course, in which Weiss' discovery had been presented, Hendrik Van Maldeghem and Bernhard M\"uhlherr \cite{F4quadrangles} found a way to realise these quadrangles as fixed buildings associated to a `Galois-like' involution on the mixed buildings of type $\mathsf F_4$ and so the class of quadrangles was dubbed \emph{mixed quadrangles of type $\mathsf F_4$} when the classification of Moufang polygons appeared in print \cite{moufangpolygons}. According to the R\'esum\'e de Course 1997--1998, Tits gave six lectures on the subject \emph{exotic groups and Galois cohomology} where he adapts the notion of Tits index and anisotropic kernel to give a \emph{Galoisian} proof of the existence of these quadrangles, unfortunately no further details are given. We speculate that in our terminology, these \emph{Weissian quadrangles} arise by mixing together two groups of type $\mathsf F_4$ of relative rank $1$. This would explain why these quadrangles can only exist in the exotic situation of a pair of fields $k^2\subsetneq \ell\subsetneq k$ with strict inclusions: if one of the mixing maps is linear this would split on of the quadratic forms which underly the anisotropic kernel as in Remark~\ref{remark:isogenies}.\ref{bullet:split} and then one of the components group would be of relative rank~$4$.\medskip

Tits' last set of lecture notes dates from 1999--2000. Inspired by the succes of his lectures on Moufang polygons, he lectured on \emph{groups of rank 1 and Moufang sets}. He clearly has some hope that at some point a classification may be achieved, although we add that to this date, most experts believe that this is still far out of reach. The final section is titled \emph{immeubles de Moufang de rang $1$ (suite mais non fin)} (Moufang buildings of rank $1$ --- sequel but not the end). 

\subsection{2000--2016}
Around the year $2000$, Tits retired from public mathematical life but there were many other mathematicians ready to take up the baton. Nonetheless it seems that no one could  oversee the fields of incidence geometry and (pseudo-)reductive algebraic group theory the way Tits could and it seems that research in both fields developed more independently from that point onwards. Two major developments that are very relevant for our history of twisted and mixed groups.

A first development arose in the process of \emph{collecting all the examples} of twisted and mixed groups. Since no further examples were to be expected for rank $\geq 2$, the innovations concern the rank $1$ case. In 2006 \cite{F4sets}, M\"uhlherr and Van Maldeghem found new examples of Moufang sets, arising by some sort of Galois descent from the mixed quadrangles of type $\mathsf F_4$. Since these quadrangles are themselves some sort of twisted $\mathsf F_4$-buildings, these Moufang sets are sometimes called \emph{doubly twisted Moufang sets of type $\mathsf F_4$}. They stand out because, together with the small Ree groups ${}^2\mathsf G_2$, they are the only Moufang sets with root groups of nilpotency class $3$, rather than $1$ or $2$.  Later, in \cite{doublytwisted}, Tom De Medts, Yoav Segev and Richard Weiss showed that these groups can also be obtained starting from groups of type ${}^2\mathsf F_4$ related to the Moufang octagons, resulting in a `commuting diagram' of groups or geometries as depicted below. In fact this diagram commutes in a strong sense: \emph{every} doubly twisted group of type $\mathsf F_4$ can be obtained via either route.  We suspect that in our terminology, these doubly twisted groups arise by taking rational points of twisted group schemes as in the left diagram; if something like this is true the main result of \cite{doublytwisted} could be paraphrased as saying that \emph{Galois descent commutes with twisted descent} (as introduced in \S\ref{sec:twisted-descent}). 
\[ 
\begin{tikzcd}
 {}^2\mathsf F_{4,4} \dar[swap,"\text{Galois}"]  & \text{mixed }\mathsf F_{4,4} \lar[swap,"\text{twisted}"] \dar["\text{Galois}"]  \\ {}^2\mathsf F_{4,1} &\lar["\text{twisted}"] \text{mixed }\mathsf F_{4,1},
\end{tikzcd}
\begin{tikzcd}
 \text{octagon} \dar[swap,"\text{Galois}"]  & \text{mixed }\mathsf F_{4} \text{ building} \lar[swap,"\text{twisted}"] \dar["\text{Galois}"]  \\ \text{doubly twisted } \mathsf F_4 &\lar["\text{twisted}"] \text{Weiss quadrangle}
\end{tikzcd}
\]
where an arrow signifies that its target arises from its source via some sort of descent---which can be twisted descent or Galois descent over an extension of twisted or mixed fields. \medskip

An important further development was the trilogy \cite{weiss-spherical,weiss-affine,descent-in-buildings}, with the first two monographs authored by Weiss and the last one by Bernhard M\"uhlherr, Holger Petersson and Richard Weiss. The first two books aim to provide proofs of Tits' classification theorems for spherical and affine buildings, accessible to a wide audience and without invoking existence theorems from algebraic group theory which ultimately rely on Lie algebra considerations. The final book completes this decoupling of algebraic groups and buildings by providing a purely combinatorial descent theory for buildings (amongst other things). This provides a solid foundational background in which one can study the second diagram that we drew above and in particular it provides a far reaching generalization of (what we presume was) Tits' \emph{Galoisian} proof of the existence of the mixed quadrangles of type $\mathsf F_4$ and the construction in \cite{F4quadrangles}. \medskip

In 2015, Elizabeth Callens and Tom De Medts  \cite{callensdemedts} found Moufang sets related to groups of type $\mathsf F_4$ and relative rank $1$. We speculate that these groups are mixed groups which arise from mixing two groups of type $\mathsf F_4$, one split and one of relative rank $1$. In particular such groups  cannot admit twisted descent since the components are never isomorphic, as we saw in Corollary~\ref{cor:twisted-descent}. It is remarkable that these easier groups were only discovered later: one reason is probably that incidence geometric intuition becomes frail in low rank. In particular, such intuition breaks down in rank $0$ so although we are now encouraged to speculate about, say, anisotropic groups of type ${}^2\mathsf F_4$, they haven't been described thus far.\medskip

A second independent development that took place around the same time was the development of a structure theory and classification theorems for pseudo-reductive algebraic groups. This work appeared in the monographs \cite{PRG} by Brian Conrad, Offer Gabber and Gopal Prasad and \cite{CPRG} by Conrad and Prasad. (Bertrand R\'emy wrote an accessible exposition with some of the key ideas of the first book  \cite{remy-prg} and there is an as of yet unpublished survey of both books by Conrad and Prasad.) The outcome of this classification effort is that \emph{with few exceptions} pseudo-reductive groups arise through a standard construction which requires as input the Weil restriction of a reductive group (amongst other things). The only exceptions arise in characteristic $2$ and $3$; an important class, which includes all characteristic $3$ exceptions is the class of \emph{exotic groups}. Our Theorem~\ref{theorem:pseudo-reductive-groups} states precisely that these exotic groups arise as Weil restrictions too, but starting from \emph{mixed} reductive groups. It should be noted that this theory encounters many difficulties beyond these exotic groups too, which we have not related to our theory of mixed group schemes yet. Some of the difficulties are certainly related to the mixed groups associated to the general class of mixed groups of type $\mathsf B_n$/$\mathsf C_n$ associated to a pair of fields together with a vector space, that we mentioned earlier---and the extra complication when $n=2$ with two vector spaces. Actually in characteristic $2$, the first edition of \cite{PRG} often assumes the situation of a base field $k$ such that $[k:k^2]\leq 2$ to avoid having to deal with these situations. This shortcoming is absent in the second edition where new ideas lead to a complete theory, regardless of $[k:k^2]$. We are mildly optimistic that in time it will turn out these groups fit into our framework. What we are not so optimistic about, but still seems worth looking into, is whether the standard construction of Conrad--Gabber--Prasad admits a more natural interpretation in the mixed setting and whether the commutative pseudo-reductive groups, which are currently treated as a black box, can be investigated deeper in the mixed setting.

Perhaps in time, a theory will emerge which states that every Moufang building is either classical or associated to a mixed reductive algebraic group, and that every pseudo-reductive group arises somehow from the latter class via a Weil restriction. The present work should be seen as a first step in this direction.
 
\appendix
\renewcommand{\thesection}{\Alph{section}}

 \section{$\mathfrak m$-categories} \label{sec:Mcategories}
 After reading through \S\ref{sec:mixing-objects}, the reader may feel encouraged to consider other, similar constructions. This is not directly relevant in the study of groups related to reductive groups, because of the contraints imposed by the combinatorics of root systems, but conceivably such constructions---in particular in conjunction with Weil restrictions---could produce other interesting objects.
 
 In this Appendix we will briefly suggest a general approach to such constructions. Exclusively for this section, we will denote composition by concatenation in diagrammatical order, i.e.~we write $fg$ for $g\circ f$.
 
Let us first make explicit some observations that we were just below the surface throughout \S\ref{sec:mixing-objects}.
 
 \begin{enumerate}
 \item It seems natural to study a category of pairs $(\C,F)$ where $\C$ is a category and $F$ an endomorphism of $\mathrm{id}_\C$, with arrows between such pairs being functors $H:\C\to \D$ such that $H(F_x)=G_{H(x)}$ for every object $x\in\C$---in more technical terms, the whiskerings $F\lhd H = H \rhd G$ agree. (See Proposition~\ref{prop:functoriality} or Definition~\ref{def:bewitched}.)
 
  For instance, consider the \emph{walking endomorphism} $\mathscr N$; this is a category with a single object $\bullet$ and the monoid $\mathbb N$ as endomorphisms $\End_{\mathscr N}(\bullet)$. Its identity functor has an endomorphism which is called the \emph{step} $s : \mathbf 1_{\mathscr N}\to \mathbf 1_{\mathscr N}$, defined by
 \[ \begin{tikzcd}[row sep=small,column sep=small] \bullet \rar["n"] \dar[near start,"1"] & \bullet \dar[near start,"1"] \\ \bullet \rar["n"] & \bullet.  \end{tikzcd} \]
 This provides us with such a pair $(\mathscr N,s^n)$ for every natural number $n$. We can then verify that
  \begin{align*}
  \hom((\mathscr N,s^0),(\C,F)) &\simeq \C^F \\
  \hom((\mathscr N,s^1),(\C,F)) &\simeq \C \\
  \hom((\mathscr N,s^2),(\C,F)) &\simeq t\C,
  \end{align*} where $\C^F$ is the full subcategory of objects $X$ such that $F_X=\mathrm{id}_X$ and $t\C$ the twisted category from Definition~\ref{def:tC}.
 
  As a second example, we consider the \emph{hopping endomorphism} $\mathscr H$. This is a category with two objects $\bullet$ and $\circ$, with arrows $\bullet\overset\alpha\to\circ$ and  $\circ\overset\beta\to\bullet$ and everything these arrows generate. Here too there is an endomorphism $h$ of the identity functor completely determined by
 \[ \begin{tikzcd}[row sep=small,column sep=small] \bullet \rar["\alpha"] \dar[near start,"\alpha\beta"] & \circ \dar[near start,"\beta\alpha"] \\ \bullet \rar["\alpha"] & \circ  \end{tikzcd} \text{ and } \begin{tikzcd}[row sep=small,column sep=small] \circ \rar["\beta"] \dar[near start,"\beta\alpha"] & \bullet \dar[near start,"\alpha\beta"] \\ \circ \rar["\beta"] & \bullet.  \end{tikzcd} \]
 And in this case $\hom((\mathscr H,h),(\C,F))=m\C$, as introduced in Definition~\ref{def:mC}.
 
 This makes certain observations easier. For instance, the reader can try to interpret some of the functors that we defined in \S\ref{sec:functors} as coming from arrows between $(\mathscr N,s^1)$, $(\mathscr N,s^2)$ and $(\mathscr H,h)$.
 \item One inconvenience with this notion arises as follows. The category $t\C$ is naturally endowed with an endomorphism $\Phi$ of the identity functor, provided by the twisters. This tells the full story of $t\C$ in some sense and it would certainly be much better if we could write
 \[  
 \Hom((\mathscr N,s^2),(\C,F)) = (t\C,\Phi).
 \]
 But in the category $m\C$, the situation is more complicated since we must incorporate information about the \emph{mixing maps} into the picture. These produce a collection of maps $\Phi_X : X \to \tau^*(X)$ which combine into  natural transformations $\Phi:\mathrm{id}_{m\C}\to\tau^*$ and $\tau^*\Phi:\tau^*\to\mathrm{id}_{m\C}$. We see that somehow $\tau^*$, $\Phi$ and $\tau^*\Phi$ must come from the endofunctor $T:\mathscr H\to\mathscr H$ which is defined by $T(\bullet)=\circ$, $T(\circ)=\bullet$, $T(\alpha)=\beta$, $T(\beta)=\alpha$, and the natural transformations $u:\mathrm{id}_\C\to T$ and $v:T\to \mathrm{id}_\C$ which satisfy $uv=h$. Clearly, we need to incorporate this in the picture to tell the full story of $m\C$.
 \item Another intuition is that $F\in\End(\mathrm{id}_\C)$ tears a \emph{hole} in the category $\C$. This became apparent already in \S\ref{sec:def-cat} where we drew diagrams
 \[ \begin{tikzcd} \bullet \ar[r,shift left=0.5ex,"\alpha"] & \circ \ar[l,shift left=0.5ex,"\beta"]  \end{tikzcd} \]
 but had to warn the reader that this diagram does not commute but rather $\alpha\beta=F_\bullet$ and $\beta\alpha=F_\circ$, as if there was a hole in the middle of the diagram, preventing us from contracting paths. A related difficulty was encountered in \S\ref{sec:bookkeeping}, where we extended  functors $f^*$ and $f_*$ to semi-linear maps. Somehow this keeps track of how many times a morphism has encircled such a hole, with the monoid $(\mathbb N,+)$ playing the role of a \emph{fundamental monoid} underlying this phenomenon.
 
 A similar situation occurs in semi-linear algebra when we are studying objects over a base object, say schemes $X$ over a field $K$, and suddenly become interested in morphisms $X\to Y$ which are not linear over $K$ but rather over a deeper lying object $k$, say for a Galois extension $K/k$. Every such morphism projects to an element of the Galois group $\mathrm{Gal}(K/k)$ which keeps track of the semi-linearity.
 \end{enumerate}
 
 Although it is straightforward to generalize the notion of a category with endomorphism of the identity $(\C,F)$ to a category with \emph{a monoid} $M$ of endomorphisms of the identity functor, this cannot be the right approach for $m\C$ or for the examples from Galois theory. (In fact $\End(\mathrm{id}_\C)$ is always a commutative monoid.) To include these cases, we will need a definition that is most elegantly stated in the language of (strict) $2$-categories.
 
 In general we will denote $2$-categories with the Fraktur alphabet and in particular we will use $\Catt$ to denote the $2$-category of categories, functors and natural transformations.
 
 \begin{definition}
 	Let $\mathfrak m$ be a $2$-category with a single object $\bullet$. An \emph{$\mathfrak m$-category} is a category $\C$ with a strict $2$-functor $\mathfrak m\to \Catt:\bullet\leadsto\C$.
 \end{definition}
Let us first explain the name. If a monoid $M$ acts on a set $X$, we call $X$ an $M$-set. There is then a morphism $f:M\to \End_{\Set}(X)$ of monoids and thus a functor 
 \[  \M \to \Set : \bullet\leadsto X.  \]
 where $\M$ is the categorification of $M$. So a $\mathfrak m$-category is just the $2$-analogon of an $M$-set.
 
What extra conditions should apply to $\mathfrak m$ remains to be seen. For instance, in many of our examples, the images of the $1$-morphisms in $\mathfrak m$ are invertible. Let us nonetheless what investigate which $2$-categories $\mathfrak m$ are responsible for the examples that we had in mind.
 
\begin{itemize}
	\item If $M$ is a commutative monoid, there is always a $2$-category $\mathfrak c(1,M)$ with a single object $\bullet$, a single arrow $\mathrm{id} : \bullet\to\bullet$ and a collection of $2$-cells $m:\mathrm{id}\to\mathrm{id}$ for every $m\in M$.  (In fact, $M$ \emph{must} be commutative for the interchange law to hold.)
	\item If $G$ is an arbitrary group, then we may form the group $G\rtimes G$ in the usual way with product defined by
	 \[ (a,x)(b,y) = (ab,x^by).  \]
	  We use this to define a $2$-category $\mathfrak{c}(G,G)$ as follows: there is a single object $\bullet$, the $1$-cells correspond to elements $a:\bullet\to\bullet$ of $G$ and $2$-cells are given by  $(a,x):a\implies ax$ for all $(a,x)\in G\rtimes G$.
	  The composition laws are the natural ones:
	  
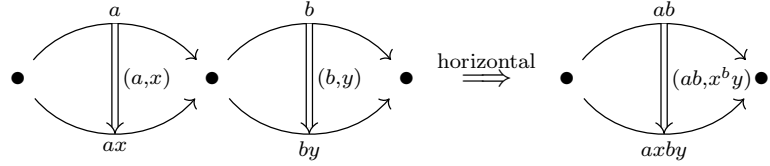
\[ \begin{tikzcd}[column sep=huge] \bullet \rar[bend left=50,"a"{name=U}] \rar[bend right=0,swap,"ax"{description,name=V}] \rar[bend right=50,swap,"axt"{name=W}] & \bullet \ar[Rightarrow, from=U, to=V,"{(a,x)}"] \ar[Rightarrow, from=V, to=W,"{(ax,t)}"]\end{tikzcd} \overset{\text{vertical}}\implies  \begin{tikzcd}[column sep=huge] \bullet \rar[bend left=50,"a"{name=U}] \rar[bend right=50,swap,"axt"{name=V}] & \bullet \ar[Rightarrow, from=U, to=V, "{(a,xt)}"] \end{tikzcd} \]
	  \[ \begin{tikzcd}[column sep=huge] \bullet \rar[bend left=50,"a"{name=U}] \rar[bend right=50,swap,"ax"{name=V}] & \bullet \rar[bend left=50,"b"{name=W}] \rar[bend right=50,swap,"by"{name=X}] & \bullet \ar[Rightarrow, from=U, to=V,"{(a,x)}"] \ar[Rightarrow, from=W, to=X,"{(b,y)}"]  \end{tikzcd}  \overset{\text{horizontal}}\implies \begin{tikzcd}[column sep=huge] \bullet \rar[bend left=50,"ab"{name=U}] \rar[bend right=50,swap,"axby"{name=V}] & \bullet \ar[Rightarrow, from=U, to=V, "{(ab,x^by)}"] \end{tikzcd} \]
	  \item Similarly, if $G$ is a group and $M$ a monoid which acts on $G$ then we may form a direct product $G\rtimes M$ and this gives rise to a category $\mathfrak c(G,M)$; this generalizes the previous two examples.
\end{itemize}
 
This relates to our examples as follows.
 
 \begin{itemize}
 	\item A category $\C$ with an endomorphism $F$ of the identity functor is a $\mathfrak c(1,\mathbb N)$-category as follows:
 	
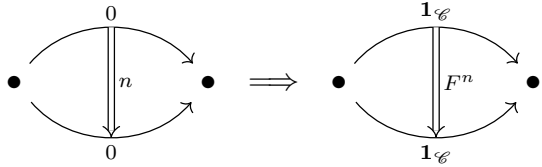
\[  \begin{tikzcd}[column sep=huge] \bullet \rar[bend left=50,"0"{name=U}] \rar[bend right=50,swap,"0"{name=V}] & \bullet \ar[Rightarrow,from=U,to=V,"{n}"] \end{tikzcd} \implies \begin{tikzcd}[column sep=huge] \bullet \rar[bend left=50,"\mathbf 1_\C"{name=U}] \rar[bend right=50,swap,"\mathbf 1_\C"{name=V}] & \bullet \ar[Rightarrow,from=U,to=V,"{F^n}"] \end{tikzcd}  \]
 	\item The mixed category $m\C$ becomes a $\mathfrak c(\mathbb N/2\mathbb N,\mathbb N)$-category as follows:
 	
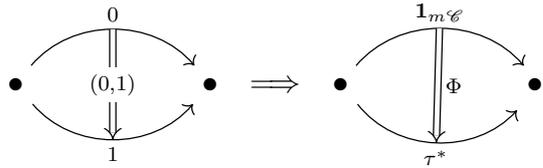
\[  \begin{tikzcd}[column sep=huge] \bullet \rar[bend left=50,"0"{name=U}] \rar[bend right=50,swap,"1"{name=V}] & \bullet \ar[Rightarrow,from=U,to=V,"{(0,1)}" description] \end{tikzcd} \implies \begin{tikzcd}[column sep=huge] \bullet \rar[bend left=50,"\mathbf 1_{m\C}"{name=U}] \rar[bend right=50,swap,"\tau^*"{name=V}] & \bullet, \ar[Rightarrow,from=U,to=V,"{\Phi}"] \end{tikzcd}  \]
 	where we leave the other assignments to the reader.
 	\item Consider a category $\D$ with an object $K$ and a subgroup $G\leq\mathrm{Aut}(K)$. The \emph{fairy} $\C$ is just the slice category over $K$ with $G$-semilinear arrows---in detail: the objects are the arrows $q_X:X\to K$ in $\D$ and the arrows $X\to Y$ are the pairs of arrows $(f,f^\natural)$ such that $f^\natural\circ q_X=q_Y\circ f$. (We denote such an arrow succinctly as $\begin{tikzcd}X \rar["{f,f^\natural}"] & Y\end{tikzcd}$.)
 	
 	Then $\C$ acquires the structure of a $\mathfrak c(G,G)$ category as follows
 	 
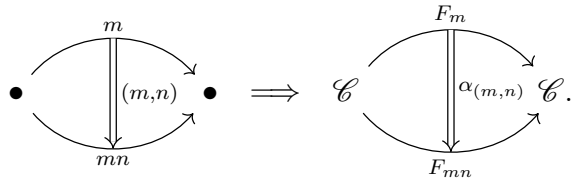
\[  \begin{tikzcd}[column sep=huge] \bullet \rar[bend left=50,"m"{name=U}] \rar[bend right=50,swap,"mn"{name=V}] & \bullet \ar[Rightarrow,from=U,to=V,"{(m,n)}"] \end{tikzcd} \implies \begin{tikzcd}[column sep=huge] \C \rar[bend left=50,"F_m"{name=U}] \rar[bend right=50,swap,"F_{mn}"{name=V}] & \C. \ar[Rightarrow,from=U,to=V,"{\alpha_{(m,n)}}"] \end{tikzcd}  \]
 	The functors $F_m$, one for every $m\in G$, are given by
%
 	\[ F_m : \C\to \C : \begin{tikzcd} X \rar["f"] \dar["q_X"] & Y \dar["q_Y"] \\ K \rar["{f^\natural}"] & K \end{tikzcd}  \mapsto \begin{tikzcd} X \rar["f"] \dar["q_Xm"] & Y \dar["q_Ym"] \\ K \rar["{(f^\natural)^m}"] & K, \end{tikzcd}  \]
 	denoted more succinctly by
 	\[  F_m : \C\to\C: \big( \begin{tikzcd}X \rar["{f,f^\natural}"] & Y\end{tikzcd}\big) \mapsto \big(\begin{tikzcd}X^m \rar["{f,(f^\natural)^m}"] & Y^m\end{tikzcd}\big).  \]
 	The natural transformations $\alpha_{(m,n)}:F_{m}\to F_{mn}$, one for every pair $m,n\in G$, are given by $(\alpha_{(m,n)})_X = (\mathrm{id}_X,n)$:
 	\[   \begin{tikzcd} X^m \rar["{f,(f^\natural)^m}"] \dar[swap,"{(\mathrm{id}_X,n)}"] & Y^m \dar["{(\mathrm{id}_Y,n)}"]  \\ X^{mn} \rar["{f,(f^\natural)^{mn}}"] & Y^{mn}.  \end{tikzcd} \]
 \end{itemize}

 So we believe that a careful study of $\mathfrak m$-categories (or anything like it) might be of higher explanatory value than a straightforward generalization of our construction.
 
 \section{Fields}
We will now prove two propositions on twisted and mixed field. We have two reasons for doing so. The first reason is that we believe these propositions can be of direct interest for anyone willing to undertake the study of groups and geometries of types ${}^2\mathsf G_2$ and mixed $\mathsf F_4$ from a Galois cohomology point of view, for instance see \cite{callensdemedts} for a wild occurence of a mixed field extension. A second reason is that it provides a peek behind the curtains of what to expect from a twisted or mixed Galois theory.

In this section, we will use exponential notations such as $x^\theta = \theta(x)$ and $x^{\theta-1}=\theta(x)/x$.

 \subsection{Twisted fields and $p=3$}
 Let us first investigate blended fields $(k,\theta)$ also known as fields with Tits endomorphism. Surprisingly at first, the underlying field is never algebraically closed. For $p=2$, it is shown in \cite{doublytwisted} that the equation $x^2+x+1$ has no solutions. For other characteristics, we have:
 
 \begin{proposition} \label{notalgcl} If $p>2$ then equations $x^{\theta-1}=-1$ and $x^{p-1}=-1$ have no solutions.
 \end{proposition}
 \begin{proof}
 	For the first part apply $\theta$ to $x^\theta = -x$ to obtain $x^p=-x^\theta=x$. This implies that $x\in\mathbf F_p$, but $\theta$ acts trivally on the prime field and therefore $x^{\theta-1}=1$. For the second part, observe that that $p-1=(\theta-1)(\theta+1)$.
 \end{proof}
 
 This shows that $\theta$ cannot be extended to the algebraic closure $k_{\mathrm a}$, and not even to the separable closure $k_{\mathrm s}$. On the other hand, $\theta$ can always be extended to the \emph{perfect closure}  $k_{\mathrm p}$.
 
 \begin{proposition}  There exists a field extension $k_{\mathrm p}/k$ such that $k_{\mathrm p}$ is perfect and $\theta$ can be extended to $K$.
 \end{proposition}
 \begin{proof}
 	Clearly $\theta$ can be extended to $k^{p^{-n}}\subseteq k_{\mathrm a}$ for all $n\in \mathbb N$, by the isomorphism $k^{p^{-n}} \to k : x\mapsto x^{p^n}$, so $\theta$ can be extended to $k_{\mathrm p} = k^{p^{-\infty}} = \cup_n k^{p^{-n}}$.
 \end{proof}
 
 It turns out that when $p=2$ resp.~$p=3$, the unsolvability of the equation $x^2+x+1$ resp.~$x^2=-1$ is essentially the \emph{only} obstruction for extending $\theta$ to a quadratic extension. For $p=2$ this is implicit in \cite{doublytwisted}, so from now on we will focus on $p=3$.
 
 More precisely, we will show that for $p=3$, $\theta$ can be extended to a field $K$ where there are only two classes of squares: the class of $1$ and the class of $-1$.

 \begin{lemma}
 	Let $p=3$ and assume $\delta\in k^\times$. Then exactly one of the following occurs.
 	\begin{itemize}
 		\item $\delta=x^2$ for some $x$;
 		\item there exists a field extension $\ell/k$ of degree $4$ such that $\delta$ has a square root in $\ell$ and $\theta$ can be extended to $\ell$;
 		\item $\delta=-x^{\theta-1}$ for some $x$ and there exists no field extension $\ell/k$ such that $\delta$ has a square root in $\ell$ and $\theta$ can be extended to $\ell$.
 	\end{itemize}
 \end{lemma}
 \begin{proof}
 	Assume that $\delta$ and $\delta^\theta$ belong to a different square class. Let
 	\[  \ell=k(\sqrt{\delta},\sqrt{\delta^\theta}), \]
 	then clearly $[\ell:k]=4$. Extend $\theta$ to $\ell$ by setting
 	\[ \sqrt{\delta}^\theta = \sqrt{\delta^\theta}\quad \sqrt{\delta^\theta}^\theta = \sqrt{\delta}^3=\delta\sqrt{\delta},   \]
 	then we verify that this gives rise to an endomorphism of $\ell$. It is sufficient to verify on the basis $1,\sqrt{\delta},\sqrt{\delta^\theta},\sqrt{\delta^{\theta+1}}$ that $(uv)^\theta = u^\theta v^\theta$ and this is a quick exercise.
 
 	Otherwise, $\delta\delta^\theta = x^2$ for some $x$. Then $\delta^{\theta+1} = (x^{\theta-1})^{\theta+1}$.  Applying $\theta-1$ we get $\delta^2 = (x^{\theta-1})^2$. So either
 	\begin{itemize}
 		\item $\delta = x^{\theta-1}$. But now, either $x^\theta$ and $x$ belong to the same square class, then $\delta$ is a square, or they belong to a different class and the field extension $k(\sqrt{x},\sqrt{x^\theta})$ can be constructed by the first item and does the job, since $(\sqrt{x^\theta}/\sqrt{x})^2 = \delta$.
 		\item $\delta = -x^{\theta-1}$, and then assume $\delta=y^2$ in the extension field $\ell$. Then applying $\theta+1$ to $y^2 = (y^{\theta+1})^{(\theta-1)} = -x^{\theta-1}$ gives $(y^{\theta+1})^2 = x^2$ so $x=\pm y^{\theta+1}$. But then $x^{\theta-1}=y^2$ and thus $\delta=-x^{\theta-1}=-y^2$, contradiction. \qedhere
 	\end{itemize}
 \end{proof}
 
 \begin{proposition}
 	There exists a field extension $K/k$ such that $\theta$ can be extended to $K$ and every element of $K^\times$ is either of a square or minus a square.
 \end{proposition}
 \begin{proof}
 	Iteratively apply the previous lemma to obtain a field where every element $\delta$ is either a square or of the form $\delta = -x^{\theta-1}$.
 
 	If $\delta$ is a square, say $\delta = x^2$, and then it is of the form $\delta = y^{\theta-1}$ with $y=x^{\theta+1}$.
 
 	If $\delta = -x^{\theta-1}$, then either $x$ is a square, say $x=y^2$ and then $\delta = - (y^{\theta-1})^2$ is minus a square, or $x$ is of the form $x=-y^{\theta-1}$ and thus $\delta  = - ((-y^{\theta-1})^{\theta-1}) = -y^{\theta^2-2\theta+1}=-y^{4-2\theta}=-(y^{2-\theta})^2$ and $\delta$ is minus a square again.
 \end{proof}
 
 \subsection{Mixed fields and $p=2$}
 Recall from Example \ref{ex:mixedrings}.\ref{ex:nestedfields} that a mixed field $(k,\ell,\kappa,\lambda)$ always originates from a field extension $\ell/k$ such that $\ell^p\subseteq k$. Let $\Omega/\ell$ be another field extension and assume $L/K$ are subfields of $\Omega$ such that $L^p\subseteq K$, $\ell\subseteq L$ and $k\subseteq K$ then $M=(K,L)$ is an extension of the mixed field $m=(k,\ell)$; moreover it is easily verified that every extension if mixed fields arised this way. In particular, taking $\Omega = \ell^{\mathrm a}$ the algebraic closure of $\ell$, we see that there is an algebraic closure (in contrast with Proposition~\ref{notalgcl}).
 
 From now on, let $p=2$ and let us study \'etale algebras over a field of degree $2$, by which we mean extensions $(K,L)/(k,\ell)$ such that $K/k$ and $L/\ell$ are \'etale algebras of degree $2$. In Grothendieck's Galois theory, these should correspond to sets of order $2$ with a continuous action of the absolute Galois group, although it remains to be seen what this means for mixed fields. Recall the following fact:
 
 \begin{proposition} \label{prop:etale-ord}The \'etale extensions of $k$ of degree $2$ are classified by the elements of $\coker\wp = k/\im\wp$, where
 \[ \wp :  k \to k : u \mapsto u^2 +u. \]
 \end{proposition}
In this correspondence the element $u\in k$ correspond to the extension $k[X]/(X^2+X+u)$ and in particular the trivial element of $\coker\wp$ corresponds to $k\oplus k$. 
 
This is the mixed analog of this fact \ref{prop:etale-ord}:
 
 \begin{proposition} The \'etale extensions of $m=(k,\ell,\kappa,\lambda)$ of degree $2$ are classified by the elements of $\coker \tilde\wp = (k\oplus \ell)/\im \tilde\wp$ where
 \[ \tilde\wp : k\oplus\ell\to k\oplus \ell : (x,y)\mapsto (x+y^\lambda,x^\kappa+y). \]
 \end{proposition}
  
 \begin{proof} Clearly every \'etale extension can be realised as
 
 \[ \begin{tikzcd}
 \ell(\sqrt[s]{e}) \rar[shift left=.5ex] & k(\sqrt[s]{d}) \lar[shift left=.5ex] \\
 \ell \uar[tail] \rar[shift left=.5ex,"\lambda"] & k \uar[tail] \lar[shift left=.5ex,"\kappa"]
 \end{tikzcd} \]
 Where we have used the notation $\ell(\sqrt[s]{e})$ for the extension $\ell[X]/(X^2+X+e)$. If we denote the extensions of $\kappa$ and $\lambda$ by the same letters, we must have
 \begin{align*}
  \lambda(\sqrt[s]{e}) &=  x + x'\sqrt[s]{d}, x,x'\in k \\
  \kappa(\sqrt[s]{d}) &= y + y'\sqrt[s]{e}, y,y'\in \ell
 \end{align*}
 We may now apply $\kappa$ to the first equation, substitute the second and express the result with respect to the $\ell$-basis $1,\sqrt[s]{e}$, to obtain $(x')^\kappa y' = 1$ and $x^\kappa+(x')^\kappa y = e$.  Mutatis mutandis, we also have $(y')^\lambda x' = 1$ and $y^\lambda + (y')^\lambda x = d$. This implies that $x'=y'=1$ and thus
 \begin{align*}
  e &= x^\kappa+y  \\
  d &= x+y^\lambda
 \end{align*}
 So every element $(x,y)$ of $k\oplus\ell$ determines a mixed \'etale extension, since it determines both $e,d$ and $\kappa,\lambda$. Two such elements determine the same extensions if and only if their difference $(s,t)$ satisfies
 \begin{align*}
 s^\kappa + t  &= a^2 + a \\
 s + t^\lambda &= b^2 + b,
 \end{align*}
 for some $a\in\ell$ and $b\in k$. Applying $\kappa$ to the second equation and adding to the first, we obtain $t+t^2 = (a+b^\kappa) + (a+b^\kappa)^2$. This implies that $t = a+b^\kappa+0$ or $t = a+b^\kappa+1$. Analogously $s=a^\lambda + b+0$ or $s=a^\lambda+ b + 1$. It is also clear that either both solutions are $+0$ or both are $+1$. So, after relabeling $a+1$ by $a$ in the latter case, we see that $t = a+b^\kappa$ and $s=a^\lambda + b$, so $(s,t)\in\im \tilde\wp$.
 \end{proof}

 \begin{corollary} There is a bijective correspondence between \'etale $k$-algebras of degree $2$ and \'etale $m$-algebras of degree $2$ provided by:
 \begin{align*}
 	\coker \wp \to \coker \tilde \wp &: u \mapsto (u,0) \\
 	\coker \tilde\wp \to \coker \wp &: (u,v) \mapsto u+v^\lambda.
 \end{align*}
 \end{corollary}
 \begin{proof} It is immediately verified that the maps are well defined and inverses of each other, using the identities $(a^2+a,0)=\tilde\wp(a,a^\kappa)$ and $(b^\lambda,b)=\tilde\wp(0,b)$.
 \end{proof}

 \section{Conjectural taxonomy of \tops{$\mathsf F_4$}{F4}}
 The confusion surrounding mixed and twisted groups in the literature culminates around the particularly interesting case of groups of absolute type $\mathsf F_4$. In this appendix, we attempt to clear up the confusion and \emph{conjecturally} postulate a taxonomy for $\mathsf F_4$.
 
 Let us first focus on the mixed groups. Recall  that the exceptional group $\mathsf F_4$ can admit the possible forms $\mathsf F_{4,4}$, $\mathsf F_{4,1}$ and $\mathsf F_{4,0}$ over an arbitrary field $k$. Applying the mixing functor from \S\ref{sec:functors}, we get the same groups, interpreted as mixed algebraic groups. We assign a  a Dynkin diagram simply by drawing parallel Dynkin diagrams for the two fibres, i.e.~we double up the standard diagrams. We emphasize that these groups are nothing special, they are just the standard groups of type $\mathsf F_4$, but seen as a visible mixed group. Such a group is defined over a visible field $(k,k,\mathrm{fr}_k,\mathrm{id}_k)$ and then by base change over other fields as well.
 
 \begin{center}
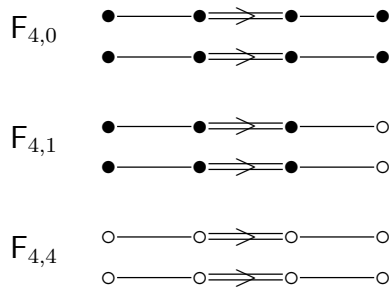
 	\begin{tabular}{cc}
 		$\mathsf F_{4,0}$  & \dynkin \begin{tikzcd} \bullet \rar[dash] & \bullet \fatarrow   & \bullet \rar[dash] & \bullet \\ \bullet \rar[dash] &  \bullet \fatarrow & \bullet \rar[dash] & \bullet \end{tikzcd} \\
 		& \\
 		$\mathsf F_{4,1}$ &  \dynkin \begin{tikzcd}  \bullet \rar[dash] & \bullet \fatarrow   & \bullet \rar[dash] & \circ \\ \bullet \rar[dash] &  \bullet \fatarrow & \bullet \rar[dash] & \circ\end{tikzcd} \\
 		& \\
 		$\mathsf F_{4,4}$ &  \dynkin \begin{tikzcd}  \circ \rar[dash] & \circ \fatarrow   & \circ \rar[dash] & \circ \\ \circ \rar[dash] &  \circ \fatarrow & \circ \rar[dash] & \circ \end{tikzcd}
 	\end{tabular}
 \end{center}
 
 By mixing them the other way, or applying the functor $\tau^*$, we also get a class of corresponding anti-visible groups, defined over anti-visible fields and then by base change over other fields too.
 
 If $p=2$ however, there are extra invisible options. Thanks to the very special isogeny $\mathsf F_4\to\mathsf F_4$, we can mix $\mathsf F_4$ with itself in a non-trivial manner over a visible field $k$. (See Proposition~\ref{prop:veryspecialisogeny2}.) The most straightforward case is that of $\mathsf F_{4,4}$ where we mix two split groups. The resulting group is the group associated to a mixed building of type $\mathsf F_4$ as defined in \cite{tits74}. That this group actually corresponds to those mixed groups introduced by Tits is precisely the content of our Theorem~\ref{theorem:mixed-groups} in the case $\mathsf F_4$.  We associate to this group the following diagram, which indicates that the mixing maps align the long roots of one $\mathsf F_4$ with the short roots of the other.
 
 \begin{center}
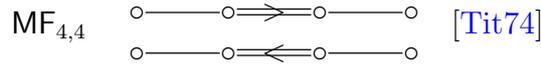
 \begin{tabular}{ccc}
 $\mathsf{MF}_{4,4}$ &  \dynkin \begin{tikzcd}  \circ \rar[dash] & \circ \fatarrow   & \circ \rar[dash] & \circ \\ \circ \rar[dash] &  \circ \fatlarrow & \circ \rar[dash] & \circ \end{tikzcd} & \cite{tits74}
 \end{tabular}
 \end{center}

  What if we try to mix \emph{non-split} groups? There the situation gets more interesting. If the mixed base field is \emph{visible}, this implies that one of the mixing maps must be linear. We suspect that this implies that one of the groups must be split as in \ref{rem:mixed-groups}. This could give rise to the mixed Moufang sets of type $\mathsf F_4$ from \cite{callensdemedts} and in principle there could also be a variant with an anisotropic $\mathsf F_4$. So we get the following diagrams---where the [$\star$] means: hypothetical.
 
 \begin{center}
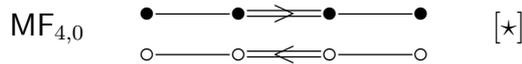
 \begin{tabular}{ccc}
 $\mathsf{MF}_{4,0}$  & \dynkin \begin{tikzcd} \bullet \rar[dash] & \bullet \fatarrow   & \bullet \rar[dash] & \bullet \\ \circ \rar[dash] &  \circ \fatlarrow & \circ \rar[dash] & \circ \end{tikzcd} & [$\star$]\\
 & & \\
 $\mathsf{MF}_{4,1/4}$ &  \dynkin 
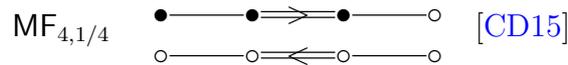
\begin{tikzcd}  \bullet \rar[dash] & \bullet \fatarrow   & \bullet \rar[dash] & \circ \\ \circ \rar[dash] &  \circ \fatlarrow & \circ \rar[dash] & \circ\end{tikzcd}  & \cite{callensdemedts}
 \end{tabular}
 \end{center}
 If the  field is non-visible, the condition that one of the groups must be split vanishes. A useful metaphor is that the non-perfect field $k^2 \subsetneq k$ generates a pool of non-splitness and by choosing an intermediary field $\ell$ strictly in between these extremes, both groups can tap into this pool which gives rise to extra possibilities where both components of the mixed group are non-split. This could give rise to the following diagrams, two of which are hypothetical and the latter of which we suspect is responsible for the mixed quadrangles of type $\mathsf F_4$.
 \begin{center}
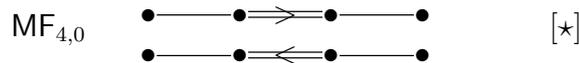
 \begin{tabular}{ccc}
 $\mathsf{MF}_{4,0}$  & \dynkin \begin{tikzcd} \bullet \rar[dash] & \bullet \fatarrow   & \bullet \rar[dash] & \bullet \\ \bullet \rar[dash] &  \bullet \fatlarrow & \bullet \rar[dash] & \bullet \end{tikzcd} & [$\star$] \\
 & & \\
 $\mathsf{MF}_{4,0/1}$ &  \dynkin 
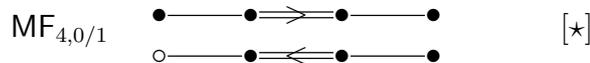
\begin{tikzcd}  \bullet \rar[dash] & \bullet \fatarrow   & \bullet \rar[dash] & \bullet \\ \circ \rar[dash] &  \bullet \fatlarrow & \bullet \rar[dash] & \bullet\end{tikzcd}  & [$\star$]\\
 & & \\
 $\mathsf{MF}_{4,1}$ &  \dynkin 
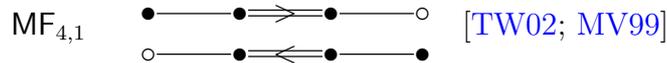
\begin{tikzcd}  \bullet \rar[dash] & \bullet \fatarrow   & \bullet \rar[dash] & \circ \\ \circ \rar[dash] &  \bullet \fatlarrow & \bullet \rar[dash] & \bullet\end{tikzcd}  & \cite{moufangpolygons,F4quadrangles} 
 \end{tabular}
 \end{center}
 
 To figure out what twisted groups exist, we must ask ourselves which of the mixed groups admit twisted descent. Of course, the ground field over which they are defined must admit twisted descent but more importantly, both components must be isomorphic. This suggests three forms ${}^2\mathsf F_{4,r}$, which arise by twisting $\mathsf F_{4,r}$ for $r=0,1,4$. The $r=4$ corresponds to the large Ree groups ${}^2\mathsf F_4$ found by Ree \cite{ReeF4}, as we have shown in Theorem~\ref{theorem:twisted-groups}. The case $r=1$ we conjecture to exist and correspond to the \emph{doubly twisted Moufang sets of type $\mathsf F_4$}, introduced in \cite{F4sets} and studied thoroughly in \cite{doublytwisted}. Finally, the anisotropic ${}^2\mathsf F_{4,0}$ is just hypothetical.
 
The connection with the theory of BN-pairs on the groups of rational points a mystery. Why would in the group $\mathsf{MF}_{4,1/4}$ the entire $4$-dimensional torus be gobbled up by the group of rank $1$ leading to a group of rank $1$, whereas in  $\mathsf{MF}_{4,1}$, somehow both $1$-dimensional tori manage to survive the mixing proces and combine to a group of rank $2$?  

\printbibliography

\end{document}